\documentclass[]{amsart}
\usepackage{amsmath,amsfonts,amssymb,amsthm}
\usepackage{enumerate}
\usepackage[usenames]{color}
\usepackage{graphicx}

 \newcommand{\obra}[3]{{\sc #1} {\em #2}. {#3}.}

\newtheorem{theorem}{\bf Theorem}
 \newtheorem{lemma}[theorem]{\bf Lemma}
 \newtheorem{proposition}[theorem]{\bf Proposition}
 \newtheorem{definition}[theorem]{\bf Definition}
 \newtheorem{corollary}[theorem]{\bf Corollary}
 \newtheorem{remark}[theorem]{Remark}
 \newtheorem{scholium}[theorem]{Scholium}
 \newtheorem{remarks}[theorem]{Remarks}

 \renewenvironment{proof}{{\em Proof \/.-}}
   {\hfill $\square$\newline}

\DeclareMathOperator{\Sing}{Sing}

\DeclareMathOperator{\Sat}{Sat}

 \newcommand{\R}{\mathbb{R}}

 \newcommand{\SSS}{\mathbb{S}}

 \newcommand{\FF}{\mathcal{F}}
 \newcommand{\EE}{\mathcal{E}}
 \newcommand{\PPP}{\mathcal{P}}
  
 \newcommand{\CC}{\mathcal{C}}
  \newcommand{\KK}{\mathcal{K}}
  
   \newcommand{\RR}{\mathcal{R}}
   \newcommand{\DD}{\mathcal{D}}
    \newcommand{\LL}{\mathcal{L}}
    \newcommand{\NN}{\mathcal{N}}
    \newcommand{\MM}{\mathcal{M}}
    \newcommand{\TT}{\mathcal{T}}
    \newcommand{\VV}{\mathcal{V}}
 \newcommand{\AAA}{\mathcal{A}}

 \newcommand{\xx}{{\bf x}}

  \newcommand{\vp}{\varphi}
  \newcommand{\wt}[1]{{\widetilde{#1}}}
   \newcommand{\wh}[1]{{\widehat{#1}}}

\newcommand{\eps}{\varepsilon}
\newcommand{\g}{\gamma}
\newcommand{\G}{\Gamma}

\begin{document}

\author{C. Alonso-Gonz{\'a}lez}
\address{Universitat d'Alacant. Departamento de Matem\'{a}ticas.
	Carretera de San Vicente del Raspeig s/n 03690 San Vicente del
	Raspeig Alicante (Spain)} \email{clementa.alonso@ua.es}
\author{F. Sanz S\'{a}nchez}
\address{Departamento de \'{A}lgebra, An\'{a}lisis Matem\'{a}tico, Geometr\'{\i}a y Topolog\'{\i}a\\
	Universidad de Valladolid, Spain} \email{fsanz@agt.uva.es}
\date{\today}
\title[Stratification of the dynamics]
{Stratification of three-dimensional real flows I: Fitting Domains \footnotemark[1]}
\maketitle

\begin{abstract}
	Let $\xi$ be an analytic vector field in $\mathbb{R}^3$ with an isolated singularity at the origin and having only hyperbolic singular points after a reduction of singularities $\pi:M\to\R^3$. The union of the images by $\pi$ of the local invariant manifolds at those hyperbolic points, denoted by $\Lambda$, is composed of trajectories of $\xi$ accumulating to $0 \in \mathbb{R}^3$. Assuming that there are no cycles nor polycycles on the divisor of $\pi$, together with a Morse-Smale type property and a non-resonance condition on the eigenvalues at these points, in this paper we prove the existence of a fundamental system $\{V_n\}$ of neighborhoods well adapted for the description of the local dynamics of $\xi$: the frontier $Fr(V_n)$ is everywhere tangent to $\xi$ except around $Fr(V_n)\cap\Lambda$, where transvesality is mandatory.
\end{abstract}

	
	\footnotetext[1]{The authors were supported by Ministerio de Ciencia e Innovaci\'on (MTM2016-77642-C2-1-P and PID2019-105621GB-I00).
	The second author was also supported by Junta de Castilla y León (VA083G19).}

\section{Introduction}\label{sec:intro}

This is the first of two papers dedicated to investigating the structure of the trajectory space of a real analytic vector field $\xi$ in a neighborhood of an isolated singular point $0\in\R^3$. In our study, we impose some non-degeneracy conditions that can be read after reduction of singularities of $\xi$ (in the sense of Panazzolo \cite{Pan}). In fact, under such conditions, our results also apply to a singular one-dimensional oriented foliation $\LL$ defined in an analytic three-dimensional manifold $M$ with boundary and corners such that $\partial M$ is homeomorphic to the sphere, regardeless of whether $\LL$ is the pull-back of a local vector field at the origin of $\R^3$.

Roughly, our aim is to establish a theorem of stratification of the dynamics of $\xi$ that generalizes to dimension three the classical one, coming from Poincar\'{e} \cite{Poi}, of decomposition of a planar analytic vector field dynamics into {\em parabolic}, {\em elliptic} or {\em hyperbolic} invariant sectors (see \cite{And}). The second of our papers \cite{Alo-S2} is devoted to the precise statement and proof of this result.

In the text at hand, we establish the existence of a fundamental system of compact neighborhoods which are specially adapted to the flow of $\xi$ in order to provide such a stratified structure. Those neighborhoods, called {\em fitting domains}, have the important property that their boundary is everywhere tangent to the flow, except for compulsorily controlled zones: near those points where the boundary crosses the local invariant manifolds associated with singular points after reduction of singularities. Apart from providing the technical preparation for our stratification theorem, we are convinced that the construction of fitting domains has interest in itself. They certainly provide good representantives for the ``germ of the space of leaves'' of the foliation generated by $\xi$, and they may be useful to undertake further studies of the local three-dimensional dynamics. Namely, classification under topological equivalence (this topic is treated by Alonso-González, Camacho and Cano in \cite{Alo-C-C1,Alo-C-C2}, references which constitute partially the motivation for this text); establishment and study of the partial ``Dulac's transition maps'' between transversal zones on the boundary (these maps have been extensively studied for planar vector fields, see for instance some recent references \cite{Ily-Y,Kai-R-S,Dol-S,Mar-V2}, but not yet for three-dimensional vector fields, as far as we know); suitable generalizations of the $\lambda$-lemma result for non-hyperbolic singularities; pieces for a ``surgery'' construction of global line foliations in three-dimensional spaces with certain prescribed local behavior at singular points, etc. Let us justify this conviction by motivating and describing more precisely the steps of our construction.

Consider first an analytic vector field $\xi$ with a hyperbolic singularity at $0\in\R^n$. By Hartmann-Grobman's Theorem, the germ of any trajectory of $\xi$ accumulating to the origin is contained in the union $W^s\cup W^u$ of the stable and unstable manifolds. To be able to properly state that $W^u\cup W^s$ is equal to the union of such ``genuine'' trajectories (not just germs) in a given fixed neighborhood $U$, one needs to require transversality conditions to the boundary of $U$. Typically, we
impose that the frontier $Fr(U)$ is transversal to $\xi$ only along sufficiently small neighborhoods $T^s$ and $T^u$ of $W^s\cap\partial U$ and $W^u\cap\partial U$, respectively. Then the whole space of trajectories of $\xi$ in $U$ can be described easily: each of them is either contained in $(W^s\cup W^u)\cap U$ and accumulates to $0$ or it is a segment from a point of $T^s\setminus W^s$ to a point of $T^u\setminus W^u$ and does not accumulate at $0$. Such a neighborhood, called below of {\em chimney type}, is the simplest example of a fitting domain.  

It is quite manifest that the above considerations can also be carried out for a generic non-hyperbolic planar vector field $\xi$ by virtue of the Seidenberg's {\em reduction of singularities} \cite{Sei}. To be more precise, using a real version of such a reduction of singularities by Dumortier in \cite{Dum}, there is a proper morphism $\pi:(M,D)\to(\R^2,0)$, where $M$ is an analytic surface with boundary and corners with $\partial M=D\simeq\SSS^1$ such that the pull-back $\pi^*(\xi|_{\R^2\setminus\{0\}})$ extends to an analytic foliation $\LL$ on $M$ with only finitely many singularities on $D$, all of them with non-nilpotent linear part.  Assuming the generic conditions that all the singularities are hyperbolic, that $D$ is invariant by $\LL$,  and that $D$ is not a polycycle of $\LL$ (in other words, that the vector field is not of the {\em center-focus type}), we can ``connect'' the chimney-type neighborhoods at the singular points by means of flow-boxes in order to obtain a neighborhoods basis $\{\wt{U}_n\}_n$ of $D$ in $M$ with the following properties (see Figure \ref{Fig:EntornoDimension2}): for any singular point $p\in D$, fixed a realization $W(p)$ of the local invariant manifold of $\LL$ at $p$ being transversal to $D$, the frontier $Fr(\wt{U}_n)$ is tangent to $\LL$ except on small neighborhoods $T_{p,n}$ of $Fr(\wt{U}_n)\cap W(p)$ in $Fr(\wt{U}_n)$ for each $p$. 
\begin{figure}[h]
	\begin{center}
		\includegraphics[scale=0.60]{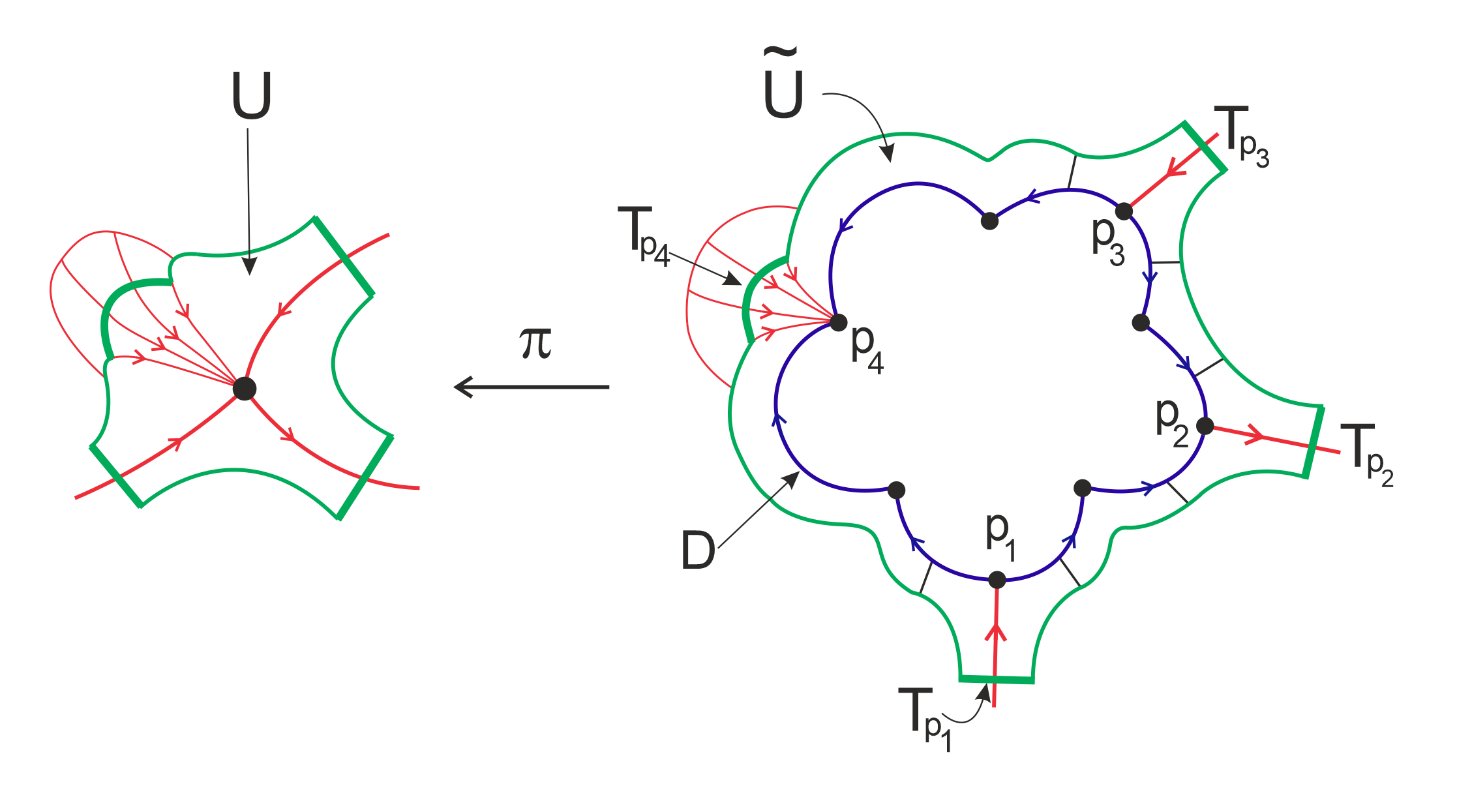}
	\end{center}
	\caption{A generalized chimney-type neighborhood in dimension two.}
	\label{Fig:EntornoDimension2}
\end{figure}

We get a basis $\{U_n=\pi(\wt{U}_n)\}_n$ of ``generalized chimney-type'' neighborhoods (fitting domains) of the origin where the space of trajectories of $\xi$ has a simple description: each trajectory either escapes in finite time (negative, positive or both) crossing the controlled transversal zone $\pi(\bigcup_pT_{p,n})$, or it remains in $U_n$ for $t\to\pm\infty$, and accumulates asymptotically to the origin in both senses.
It is worth noting that the (subanalytic) invariant set $\Lambda=\pi(\bigcup_p W(p))\setminus\{0\}$ is  formed by and contains any germ of a trajectory of $\xi$ accumulating at the origin (the so called {\em characteristic orbits} in Dumortier's terminology \cite{Dum}), thus being a realization of such {\em family of germs}. 
The transversal parts $\bigcup_pT_{p,n}$ in the above statement can be chosen to be  contained in a given fixed basis of neighborhoods of $\Lambda$.

The central objective in the current paper is to generalize the previous construction to dimension $3$. To this end, we use the reduction of singularities developed by Panazzolo in \cite{Pan}. So we have that
there exists a real analytic proper morphism $\pi:(M,D)\to(\R^3,0)$, where $M$ is a real analytic manifold with boundary and corners with $\partial M=D$, which restricts to an isomorphism outside the {\em divisor} $D$ and such that the foliation generated by $\pi^*(\xi|_{\R^3\setminus\{0\}})$ extends to an analytic foliation $\LL$ in $M$ having only elementary singularities, that is, the linear part of a local generator at any $p\in\Sing(\LL)$ is non-nilpotent.
Assuming that $0$ is an isolated singularity of $\xi$, we have, in addition, that $D$ is homeomorphic to $\SSS^2$ and $\Sing(\LL)\subset D$ (notice that in Panazzolo's procedure we just blow up along centers contained in the singular locus of the intermediate transformed foliations).

Our object of study is then a triple $\MM=(M,D,\LL)$, coming or not from a reduction of singularities of a local vector field, where $M$ is a real analytic manifold with boundary and corners, $D=\partial M$ is a normal crossings divisor homeomorphic to $\SSS^2$ and $\LL$ is a one-dimensional orientable singular foliation over $M$ such that $\Sing(\LL)\subset D$. We impose the following non-degeneracy conditions:
\begin{enumerate}
	\item {\em Non-dicriticalness.-} Every component of $D$ is invariant by $\LL$.
	\item {\em Hyperbolicity.-} All singularities are hyperbolic (therefore $\Sing(\LL)$ is a finite set).
	\item {\em Acyclicity.-} The restriction $\LL|_D$ has neither closed regular orbits nor {\em polycycles}.
	%
	\item {\em Morse-Smale condition.-} There are no leaves connecting two-dimensional saddle points of the restriction $\LL|_D$ and contained in the regular locus of the divisor.
	\item {\em No saddle-resonances.-} A condition that avoids certain specific resonances of the eigenvalues at the singular points appearing in a multiple saddle connection.
\end{enumerate}
The first two conditions are generic and commonly assumed. The third condition is essential for our study. It corresponds to the center-focus exclusion for planar vector fields. The last two conditions are more specific and deserve to be commented. They appear in precedent papers, for instance those already mentioned by Alonso-Gonz\'{a}lez et al. \cite{Alo-C-C1,Alo-C-C2}.
In general, once we are in a scenario where all the singularities are already hyperbolic, the ``classical Morse-Smale condition'' means that two saddle singularities can never be connected along the corresponding invariant stable-unstable manifolds (see for instance \cite{Pal-M}). The fourth condition above is an adaptation of such a property to the consideration of the divisor, where we only permit saddle connections along the skeleton, that is, the set of points where at least two components of $D$ intersect. Morally a two-dimensional saddle connection outside the skeleton can be avoided by ``perturbation'' whereas, those in the skeleton persist unless we completely break the divisor.

The fifth condition is more involved and it is described in detail in Section~\ref{sec:MS-foliations}. To have an idea of its meaning, consider the situation of a single connection between two saddle points $p,q$ along their common one-dimensional invariant manifolds, contained in the skeleton. The non $s$-resonance condition in this case is expressed saying that the quotients between the two (necessarily real) eigenvalues of the same sign at $p$ and at $q$ do not coincide with. Dynamically, it prevents the situation depicted in Figure \ref{Fig:ConexionSillas}, where the flow saturation of a small curve that accumulates in the ``middle'' of the two-dimensional invariant manifold at $p$ goes to the ``middle'' of the corresponding one at $q$.

\begin{figure}[h]
	\begin{center}
		\includegraphics[scale=0.45]{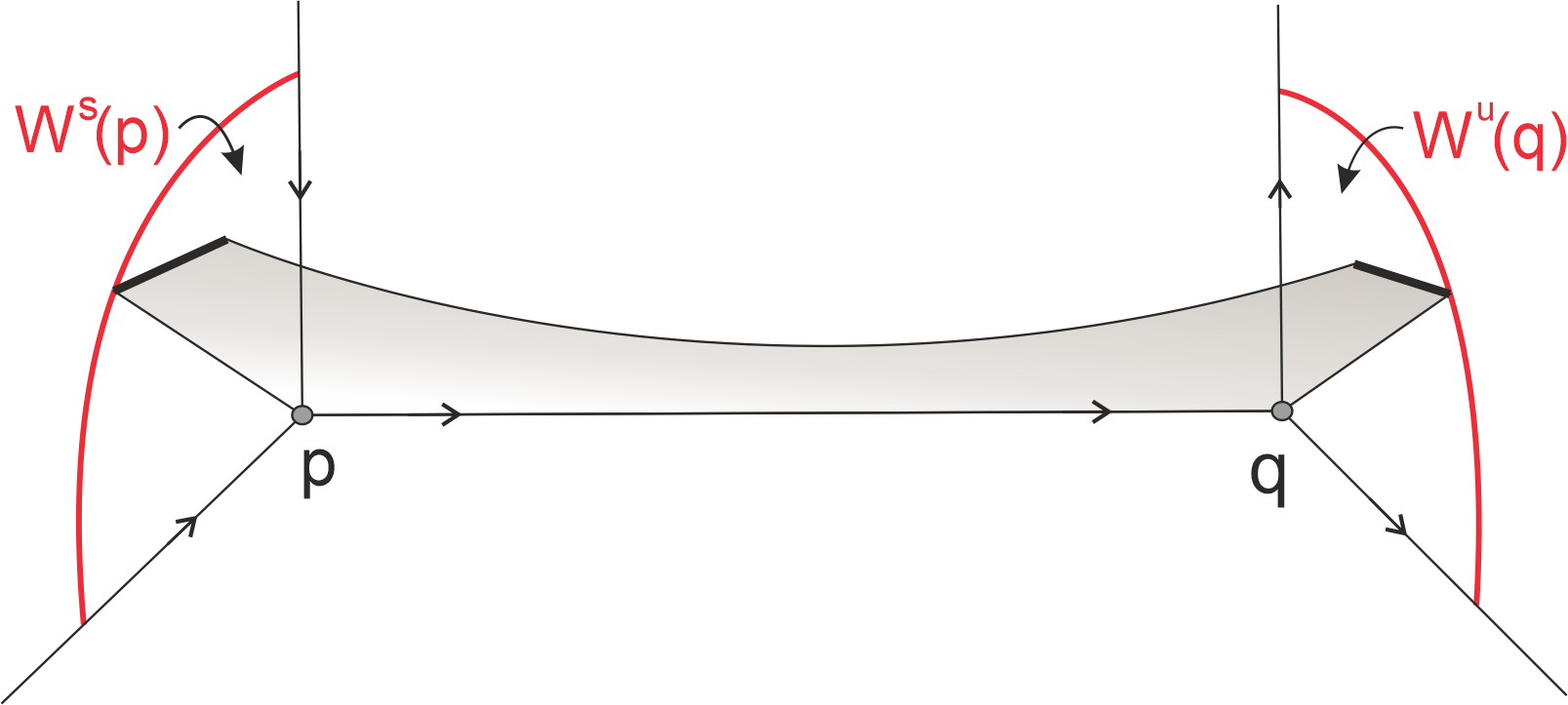}
	\end{center}
	\caption{A resonant saddle-connection.}
	\label{Fig:ConexionSillas}
\end{figure}

By means of the first and second condition, if $p$ is a singular point of $\LL$, between the two local stable and unstable manifolds of $\LL$ at $p$, either both are contained in $D$, or one of them is contained in $D$ and the other one, denoted by $W(p)$, is transversal to (and not contained in) $D$. In the first case, $p$ is called a {\em tangential saddle} point. Points $p\in\Sing(\LL)$ for which $\dim W(p)=2$ will play a special role in what follows and will be called {\em transversal saddle} points.

 As a matter of notation, if $A\subset M$, we denote by $A^\smallsmile$ the set of points where $A$ is locally invariant by $\LL$, that is, $a\in A^\smallsmile$ if and only if $a\in A$ and there exists a neighborhood $U_a$ of $M$ at $a$ such that any leaf of $\LL|_{U_a}$ through a point in $U_a\cap A$ is contained in $A$. Points of $A^\smallsmile$ where $A$ is a smooth submanifold not belonging to $\Sing(\LL)$ are, of course, points of tangency between $\LL$ and $A^\smallsmile$.
 
 The main result in this paper is the following. 
\begin{theorem}\label{th:main}
	Assume that $\MM=(M,D,\LL)$ satisfies conditions (1)-(5). Denote by $H$ the set of singular points of $\LL$ that are not tangential saddles and fix realizations of the local invariant manifolds $W(p)$, for any $p\in H$. Then, given a neighborhood $V$ of $D$ in $M$ and neighborhoods $V_p$ of $W(p)\cap V$ in $M$, for any $p\in H$, such that $V_p\cap V_q=\emptyset$ if $p\ne q$, there exists a compact semianalytic neighborhood $U\subset V$ of $D$ in $M$ and compact semianalytic discs $T_p\subset Fr(U)\cap V_p$, for each $p\in H$, satisfying the following:
	\begin{enumerate}[(i)]
		
		\item  The frontier $Fr(U)$ is a topological, piecewise smooth surface given by the disjoint union		
		$$
		Fr(U)= Fr(U)^\smallsmile\cup\bigcup_{p\in H}T_p.
		$$  
		\item Each disc $T_p$ contains $Fr(U)\cap W(p)$ and, in turn, it is contained in a smooth surface of $V_p$ everywhere transversal to $\LL$.  
	\end{enumerate}
\end{theorem}

A neighborhood $U$ as in the previous statement will be called a {\em fitting domain} for $\MM$ (see Picture \ref{Fig:EntornoDimension3}). The subset $Fr(U)^\smallsmile$ is called the {\em tangential frontier} of $U$ while $Fr(U)^\pitchfork:=Fr(U)\setminus Fr(U)^\smallsmile$ is called the {\em transversal frontier}. Roughly speaking, the result says that there exists a basis $\{U_n\}$ of fitting domains such that the sequence of their transversal frontiers $\{Fr(U_n)^\pitchfork\}$ ``aproximate'', when $n\to\infty$, the germ of the analytic invariant set $\Lambda=\cup_{p\in H} W(p)$ (in a sense that can be made precise).
Notice that $\Lambda\setminus D$ realizes the {\em family of germs} at $p$ of leaves of $\LL$ that accumulate at $D$ (the ``characteristic orbits'' in analogy with the planar situation). Any such germ of a leaf enters ultimately in the fitting domain, which justifies that the transversal discs $T_p$ contain the corresponding sets $Fr(U)\cap W(p)$.

\begin{figure}[h]
	\begin{center}
		\includegraphics[scale=0.50]{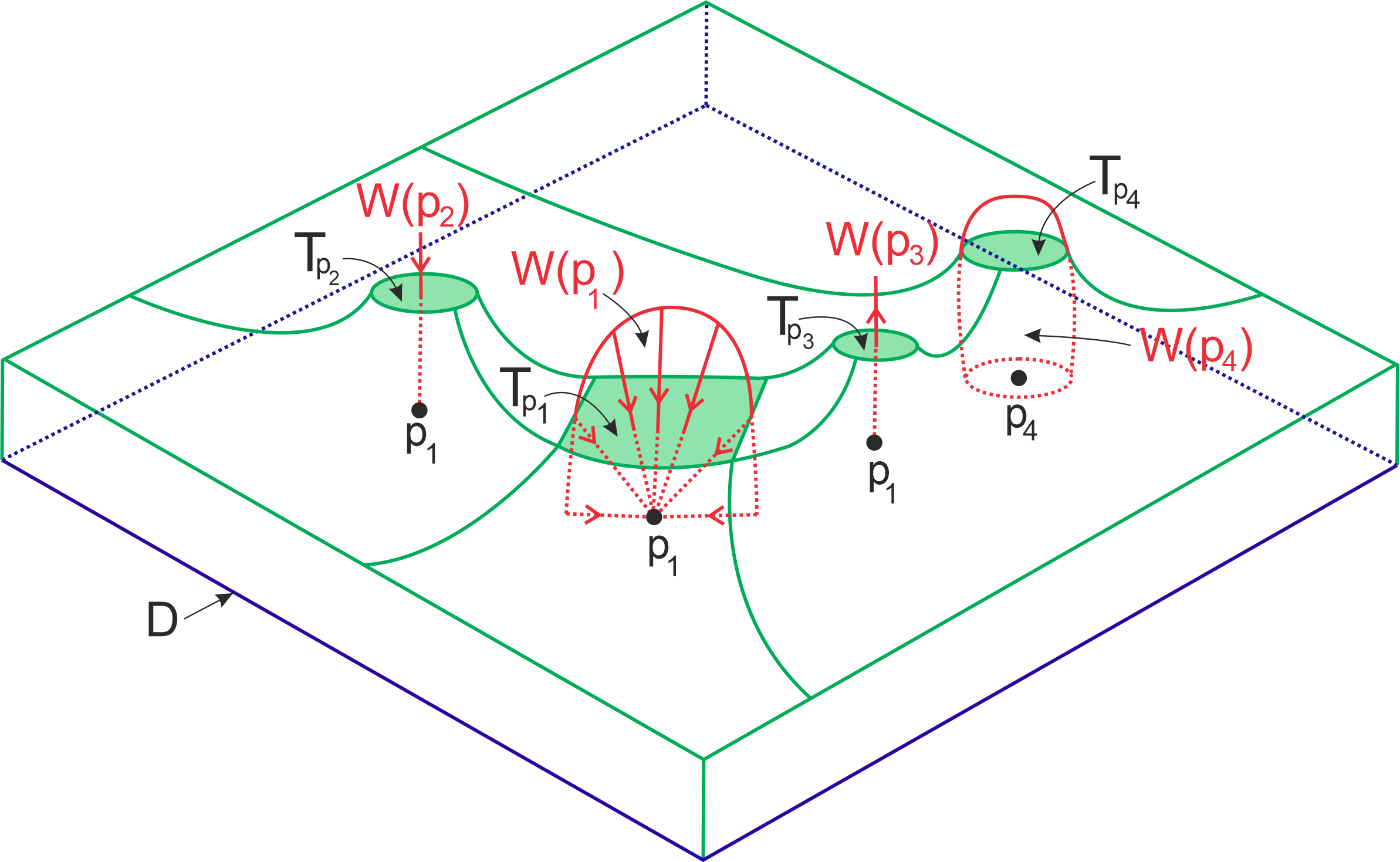}
	\end{center}
	\caption{A fitting domain in dimension three.}
	\label{Fig:EntornoDimension3}
\end{figure}

At this point, it is worth to remark that, although we have called transversal frontier to the union of the discs $T_p$, only the points in the interior of those discs are ``properly'' transversal points. On the contrary, a point $x$ in the boundary of $T_p$ may not be smooth for $Fr(U)$ and it is a transversal point only in a weak sense: one of the two sides of the leaf through $x$ may be contained in $Fr(U)^\smallsmile$, while the other one escapes or enters the neighborhood $U$. In Theorem~\ref{th:main2} below we will detail more in precise the relative position of $Fr(U)$ with respect to $\LL$ along the points in $\partial T_p$. The description is relatively simple if $p$ is not a transversal saddle: namely, if $\dim W(p)=3$, then $T_p= W(p)\cap Fr(U)$ and $\LL$ traverses from inside to outside (or viceversa) at any point of $T_p$; if $\dim W(p)=1$ then $T_p\cap W(p)$ is a singleton and $\LL$ passes from $Fr(U)$ to the exterior of $U$ (or viceversa) through any point of $\partial T_p$. But, if $\dim W(p)=2$, then there is a finite set $Z_p\subset\partial T_p$ (in fact $Z_p$ has four points) such that the leaf through a point $z\in Z_p$ stays locally inside $Fr(U)$ while, through any point of a connected component $I$ of $\partial T_p\setminus Z_p$, the leaf has one side contained in $Fr(U)$ and the other one is contained either in the interior or in the exterior of $U$. As we can see, the description of the frontier of a fitting domain is more intricate around transversal saddles. 

On the other hand, when $\MM$ comes from a morphism $\pi:M\to\R^3$ of reduction of singularities of a vector field $\xi$ at $0\in\R^3$, the image of a basis of fitting domains of $\MM$ by $\pi$ provide a basis of {\em fitting} neighborhoods of the origin for $\xi$, i.e., their frontiers satisfy the same properties of Theorem~\ref{th:main} with respect to the foliation generated by the vector field $\xi$. Under this point of view, the fitting domains provided by Theorem~\ref{th:main} can be considered as ``generalized chimney-type'' neighborhoods for non-hyperbolic singularities of three-dimensional vector fields.

\strut

The paper is structured as follows. In Section~\ref{sec:SHNR-foliations} we review some properties of the foliation $\LL$ that can be derived from conditions (1)-(3). Besides, we describe a planar directed graph $\Omega$, supported in $D$, which schematizes the dynamics inside the divisor $D$ and that will be our main combinatorial tool in the rest of the article. Its set of vertices is precisely $\Sing(\LL)$ and its edges are regular leaves, either contained in the skeleton of $D$ or those containing the local unstable or stable curves at saddle two-dimensional points of the restriction of $\LL$ to $D$. The acyclicity condition assures the $\Omega$ has no cycles as a directed graph. In particular, we can assign a {\em length} to any vertex $p$ as the maximal number of edges of paths of edges starting at $p$.

In Section~\ref{sec:MS-foliations}, we discuss conditions (4) and (5) and results that can be obtained from them. The principal one is Theorem~\ref{th:path of a trace mark}. It  asserts that, generalizing the concept of \textit{point-path} introduced in \cite{Alo-C-C1}, we can associate a path of edges $\Theta(\nu)$ of $\Omega$ with any three-dimensional saddle $p\in\Sing(\LL)$ and with any (local) side $\nu$ with respect to the two-dimensional invariant manifold at $p$, denoted as $W^2_p$, such that the saturation by $\LL$ of any small curve $J_\nu$ contained in $\nu$ and cutting transversally $W^2_p$, produces an invariant topological surface which accumulates to the divisor $D$ just along the support of $\Theta(\nu)$. In this way, we have a control over the saturations of such small curves $J_\nu$ and also over the saturation of the two-dimensional local invariant surface $W(p)=W^2_p$ in the important case where $p$ is a transversal saddle. These last saturations have the role of ``new'' two-dimensional components of the divisor, locally defined along $D$, that separate the dynamics of $\LL$ ({\em separant surfaces}). In particular, we prove that the family composed of those separant surfaces, together with any finite family of saturations of curves $J_\nu$ as above, is a family of pairwise disjoint elements, as long as we take a sufficiently small neighborhood of $D$. Besides facilitating the proof of Theorem~\ref{th:main}, these results will be crucial in our second article \cite{Alo-S2}.

The core of Theorem~\ref{th:main} proof is found in Sections \ref{sec:distinguished}, \ref{sec:good-saturations} and \ref{sec:extendable}. Following the general ideas in the two-dimensional situation, fitting domains $U$ are obtained as follows: first we consider small chimney-type neighborhoods at singular points (where the transversal frontier of $U$ will concentrate); next, we extend these neighborhoods adding flow-boxes (with appearance of tubes) along the edges of the graph $\Omega$; finally, we add new flow-boxes covering the connected components of the complement of the graph in $D$ (with appearance of plateaux). Nevertheless, it is clear that the proof is considerably much more complicated here than in dimension two since we have to match perfectly all the added flow-boxes in such a way that the frontier of $U$ is everywhere tangent outside the chimney-type local neighborhoods already considered. In the one hand, by using recurrence with respect to the maximal length of the vertices of $\Omega$, the tubes can be glued with the chimney-type neighborhoods after convenient refinements of the latter, creating a compact neighborhood $\KK$ of the support of the graph that we call a {\em distinguished fattening}. This is the content of Section~\ref{sec:distinguished}. After that, when we try to glue perfectly also the plateaux ``closing'' the remaining transversal parts of $Fr(\KK)$ (called {\em free doors} of $\KK$), we are led, {\em a priori}, to refine once more time chimneys and tubes. In order to avoid an endless sequence of refinements, in Section~\ref{sec:good-saturations} we establish a result asserting that the distinguished fattening $\KK$ can be assumed to have the property that the flow saturations of different free doors do not intersect (the so called property of {\em good saturations}). For this purpose, we make use of the results in Section 3 about the flow saturations of the aforementioned curves $J_\nu$ since the boundary of a free door is an example of such a curve. At the same time, the separant surfaces generated at transversal saddle points turn into play and we need to control as well their behavior inside the fattening support. Finally, in Section~\ref{sec:extendable}, we show that a fattening with the property of good saturations can be completed, up to trimming the free doors, to a fitting domain by adding convenient plateau blow-boxes. This will conclude the proof of the main theorem.

All in all, Sections \ref{sec:distinguished} and  \ref{sec:good-saturations} are the longest and more technical ones. However, we believe that our constructions are flexible and versatile enough to be used to get fitting domains with further interesting properties. Also, the construction could be carried out in much more general situations in which not all hypothesis (1)-(5) are necessarily fulfilled.

\section{Hyperbolic Acyclic Spherical Foliations}\label{sec:SHNR-foliations}

\subsection{Generalities on line foliations over manifolds with boundary}
Let $M$ be a three-dimensional real analytic manifold with boundary
and corners. This means that, for any $a\in M$, there exists an open neighborhood $V_a$ of $a$, a value $e=e(a)\in\{0,1,2,3\}$, and a homeomorphism $\phi_a:V_a\stackrel{\sim}{\to}(\R_{\ge 0})^{e}\times\R^{3-e}$ with $\phi_a(a)=0$ such that, whenever $V_a\cap V_b\ne\emptyset$, the map $\phi_a\circ\phi^{-1}_b$ is analytic. As usual, such a map $\phi_a$ is called a chart and its components are called analytic coordinates {\em at} $a$.
The number $e(a)$ does not depend on the chosen chart.  The point $a\in M$ will be called an {\em interior, trace, angle} or {\em corner} point if $e(a)$ is equal to either $0,1,2$ or $3$, respectively. Notice that $M$ is a topological manifold with boundary and the point $a$ belongs to $\partial M$ iff $e(a)>0$. In addition, the boundary $\partial M$ is a normal crossings divisor. We use the notation $D=\partial M$ and just say that $D$ is the {\em divisor} in $M$. The set of points where $e(a)>1$ is called  the {\em skeleton} of $D$ and it is denoted by $Sk(D)$.
 
The connected components of the fibers of the map $e:M\to\{0,1,2,3\}$ are the strata of a locally finite stratification of $M$ into smooth analytic subvarieties, called the {\em standard stratification}, that we will denote by $St(M)$. An stratum contained in the fiber $e^{-1}(k)$ has codimension equal to $k$ in $M$.
The closure of a two-dimensional stratum, that is, of a connected component of $D\setminus Sk(D)$, is
called a {\em component} of $D$. Any component $E$ of $D$ is a two-dimensional analytic manifold with boundary and corners satisfying $\partial E=E\cap Sk(D)$. Also, for any $a\in M$, the value $e(a)$ is precisely the number of components of $D$ containing $a$.

 An important example of manifold with boundary and corners is  the  ambient space obtained after a sequence of {\em real blow-ups} starting from an open set $M_0$ of $\R^3$ (for detailed definitions, see \cite{Mar-R-S}, for instance). More precisely, consider a finite composition
$$\pi:M=M_r\stackrel{\pi_r}{\longrightarrow}M_{r-1}\stackrel{\pi_{r-1}}{\longrightarrow}
\cdots\stackrel{\pi_2}{\longrightarrow}M_1\stackrel{\pi_1}{\longrightarrow}M_0
$$
where, for any $j=1,2,...,r$, the map $\pi_j$ is the blow-up with a non-singular closed center $Y_{j-1}\subset M_{j-1}$ having normal crossings with the ``intermediate divisor'' $D^{(j-1)}=(\pi_1\circ\cdots\circ\pi_{j-1})^{-1}(Y_0)$, with $D^{(0)}=\emptyset$.
Thus, $M$ is a real analytic manifold with boundary and corners with divisor $D=\partial M=\pi^{-1}(Y_0)$. Notice that, in the sequence above, if the first center is a point $Y_0=\{p_0\}$, and for any $j>1$ the center $Y_{j-1}$ of $\pi_{j}$ is contained in $D^{(j-1)}$, then $D$ is homeomorphic to the sphere $\SSS^2$.

\strut

Let $\LL$ be an {\em analytic oriented one-dimensional singular foliation} over $M$ (just called a {\em foliation}, for short). This means that for every $a\in M$, there exists an analytic vector field $\xi_a$ defined in a neighborhood of $a$ such that, for any pair of points $a,b\in M$, the vector fields $\xi_a,\xi_b$ are positively proportional along the intersection of their domains of definition. Any such a vector field $\xi_a$ is called a local {\em generator} of $\LL$ at $a$.
The {\em singular locus} of $\LL$ is the set $\Sing(\LL)=\{p\in M\,:\,\xi_p(p)=0\}$. 

From now on, we will always assume that $\LL$ is {\em tangent} to $D=\partial M$, that is, for any $a\in D$ any local generator $\xi_a$ is tangent to any component of $D$ through $a$. This condition is also sometimes called {\em non-dicriticalness}.

 A {\em leaf} (of $\LL$ {\em in} $M$) is a maximal connected subset $\ell$ of $M$ with the following property: for any $a\in\ell$, if $\xi_a$ is a local generator of $\LL$ at $a$ and $\g:(-\eps,\eps)\to M$ is the integral curve of $\xi_a$ with $\g(0)=a$, then there exists $0<\eps'\le\eps$ and some neighborhood $U_a$ of $a$ in $M$ such that $\g((-\eps',\eps'))$ is equal to the connected component of $\ell\cap U_a$ containing $a$. 
For any $a\in M$, there is exactly one leaf containing $a$, denoted by $\ell_a$ and called the leaf {\em at} $a$, so that the family of leaves gives a partition of $M$. Given a leaf $\ell$, there are two possibilities: either  $\ell=\ell_p=\{p\}$ for some $p\in\Sing(\LL)$ (a {\em singular leaf}), or  $\ell=\g(J)$, where $J$ is an open interval in $\R$ and $\g:J\to M$ is an injective immersion (a {\em non-singular leaf}, and $\g$ is called a {\em parametrization} of $\ell$). 
Any non-singular leaf $\ell$ will always be considered with the natural orientation induced by $\LL$, i.e., we choose a parametrization $\g:J\to M$ so that the tangent vector at each point is a positive multiple of the corresponding local generator. With such a parametrization, if $a=\g(t)\in\ell$, we put the sets $\ell^+_a=\g(J\cap[t,\infty))$, $\ell^-_a=\g(J\cap(-\infty,t])$, and call them the {\em positive} and {\em negative leaf} through $a$, respectively. 
Besides, we define the $\alpha$ and the $\omega$-limit set of $\ell$ in $M$ as
\begin{equation}\label{eq: alfa y omega limite}
\alpha(\ell)=\bigcap_{t_1<t<t_2}\overline{\g((t_1,t))},\;\;
\omega(\ell)=\bigcap_{t_1<t<t_2}\overline{\g((t,t_2))},	
\end{equation}
where $\mbox{int}(J)=(t_1,t_2)$. 
As usual, we write $\alpha(\ell)=p$ when $\alpha(\ell)=\{p\}$, and so on.

Given a subset $A\subset M$, an given $a\in A$, the {\em restricted leaf in $A$ at $a$} (or simply the {\em $A$-leaf at $a$}) is the connected component of $\ell_a\cap A$ containing $a$. We define similarly the {\em positive $A$-leaf} and the {\em negative $A$-leaf} {\em at} $a$ by using $\ell^+_a$ and $\ell^-_a$, instead of $\ell_a$, respectively. If $A$ is open in $M$, the $A$-leaves are the leaves of the restricted foliation $\LL|_A$.
 If $B\subset A\subset M$, we say that $B$ is {\em saturated in $A$} if, for any $b\in B$, the $A$-leaf at $b$ is contained in $B$. When $A=M$, we simply say that $B$ is saturated. For instance, $D$ is saturated by the assumed condition that $\LL$ is tangent to $D$. Finally, if $B\subset A\subset M$, the {\em saturation of $B$ in $A$}, denoted by ${\rm Sat}_A(B)$, is the minimal subset of $A$ containing $B$ that is saturated in $A$. It consists of the union of all the $A$-leaves at points of $B$. We also define the {\em positive saturation} (resp. {\em negative saturation}) of $B$ {\em in} $A$, denoted by ${\rm Sat}^{+}_A(B)$ (resp. ${\rm Sat}^{-}_A(B)$) as the union of positive $A$-leaves (resp. of negative $A$-leaves) at points of $B$. 

\strut

Let $A$ be a compact semianalytic subset of $M$ with non-empty interior. Let $a\in Fr(A)$ and assume that $a\not\in\Sing(\LL)$. Taking into account that $Fr(A)$ is also semianalytic, and due to the analytic nature of the foliation $\LL$, we have that each one of the two local components of $\ell_a\setminus\{a\}$ at $a$ is entirely contained either in the interior $int(A)$, or in the exterior $ext(A)=M\setminus A$, or in the frontier $Fr(A)$. Let us now denote by $Y_a^-,Y^+_a$ such local components, where $Y^\epsilon_a\subset\ell^\epsilon_a$ for $\epsilon\in\{+,-\}$, and consider the map $\sigma:\{i,e,t\}\to\{int(A),ext(A),Fr(A)\}$ defined by $\sigma(i)=int(A)$, $\sigma(e)=ext(A)$, $\sigma(t)=Fr(A)$. If $u,v\in\{i,e,t\}$, we will say that $a$ is of {\em type u-v with respect to $A$} in case $Y^-_a\subset\sigma(u)$ and $Y^+_a\subset\sigma(v)$.

As said in the introduction, the {\em tangential frontier} of $A$, denoted by $Fr(A)^\smallsmile$, is the set of points in $Fr(A)$ having a neighborhood in $Fr(A)$ composed of points of type t-t. Hence, the tangential frontier is open and coincides with the set of points where $Fr(A)$ is locally saturated for $\LL$. For the points $a\in Fr(A)^\smallsmile$ such that $Fr(A)$ is a smooth submanifold at $a$, the foliation $\LL$ is tangent to $Fr(A)$ in a neighborhood of $a$.

As a matter of notation, we set $Fr(A)^\pitchfork:=Fr(A)\setminus Fr(A)^\smallsmile$ and call it the {\em transversal frontier} of $A$. By definition, it is a closed subset of $Fr(A)$. Notice also that  $Fr(A)^\pitchfork$ contains any point in $Fr(A)$ which is not of type t-t. Nevertheless, it might also contain points of type t-t or even points where $Fr(A)$ is smooth and tangent to $\LL$ (it will not be the case for fitting domains). Of course, points of type either i-e or e-i, that is, points having a genuine property of ``transversality'', are included in $Fr(A)^\pitchfork$. To end this subsection, we define the {\em inner frontier} $Fr(A)^{in}$ and the {\em outer frontier} $Fr(A)^{out}$ of $A$, as the respective closures of the sets of points in $Fr(A)$ of type e-i or of type i-e. Hence, we clearly have that $Fr(A)^{in}\cup Fr(A)^{out}\subset Fr(A)^\pitchfork$, although, in general, equality may not be satisfied. We will see that it always holds for fitting domains.

\subsection{Hyperbolic acyclic spherical foliations}

Let us assume certain conditions on the topology of $M$ as well as on the nature of the singularities of the foliation in order to have a reasonable control of the asymptotic behavior of the leaves near $D$.

\begin{definition}\label{def:SHAFD}
 A {\em hyperbolic acyclic foliated variety with spherical divisor} (from now on a {\em HAFVSD}, for short) is a triple $\mathcal{M}=(M,D,\LL)$ where $M$ is a real analytic manifold with boundary and corners, $D=\partial M$, and $\LL$ is an analytic foliation over $M$ such that
 $\LL$ is tangent to $D$, $\Sing(\LL)\subset D$, and
\begin{itemize}
  \item Each $p\in\Sing(\LL)$ is a hyperbolic singular point.
   \item The foliation $\LL$ does not have cycles or polycycles contained in $D$. 
    \item The divisor $D$ is homeomorphic to the sphere $\SSS^2$.
  \end{itemize}
\end{definition}
Recall that a singular point $p$ is hyperbolic if all the eigenvalues of the linear part of one (thus, of any) local generator of $\LL$ at $p$ have non-zero real part. On the other hand, a cycle is a non-singular leaf homeomorphic to $\SSS^1$ and a polycycle is a union of a finite number of non-singular leaves $\ell_1,...,\ell_n$ and a finite number of singular points $p_1,...,p_n$ such that $p_{j+1}=\alpha(\ell_{j+1})=\omega(\ell_j)$, for $j=1,...,n-1$, and $p_1=\alpha(\ell_1)=\omega(\ell_n)$.

The condition of $D$ being tangent (also called sometimes  \emph{non-dicriticalness}) implies that $D$ is saturated for $\LL$ and, in fact, that any stratum of $St(M)$ contained in $D$ is also saturated. In particular, every corner point is singular and, if $H$ is a 1-dimensional stratum, the connected components of $H\setminus\Sing(\LL)$ are non-singular leaves of $\LL$. We consider the restriction $\LL|_D$ as a continuous one-dimensional foliation in $D$, analytic on each component of $D$.
As already pointed out above, we have that any leaf in $M$ is parameterized by an open interval and the singular leaves correspond exactly to the singular points.

By the hyperbolicity condition, the singular locus $\Sing(\LL)$ is finite. At any $p\in\Sing(\LL)$, we consider the associated {\em local stable} and {\em local unstable} manifolds, denoted by $W^s(p)$ and $W^u(p)$, respectively. They are analytic manifolds with boundary and corners, saturated in a neighborhood of $p$ and uniquely determined as germs at $p$ (for further information on such manifolds, see for instance \cite{Hir-P-S} and \cite{Car-S}).
When the singularity $p$ is a three-dimensional saddle point, that is, both stable and unstable manifolds have positive dimension, we denote by
$W^1_p$ and $W^2_p$ the elements of dimension 1 and dimension 2 in the set $\{W^s(p),W^u(p)\}$,  respectively. It holds $W^1_p\subset D$, if $e(p)\geq 2$, and
$W^2_p\subset D$, if $e(p)=3$.

Let us introduce the following terminology. Given $p \in \textrm{Sing}(\mathcal{L})$, we say that $p$ is a \emph{$D$-saddle} if there is at least one component $D_i$ of $D$ at $p$ such that $p$ is a two-dimensional saddle point for the restriction $\LL|_{D_i}$. Otherwise, we say that $p$ is a \emph{$D$-node}. A $D$-node can be a \emph{$D$-attractor} or a \emph{$D$-repeller} if the restriction of $\mathcal{L}$ to each component of $D$ at $p$ is respectively a two-dimensional attractor (both eigenvalues with negative real part) or a repeller (both eigenvalues with positive real part). We denote by $S$, $N^{a}$ and $N^r$ the sets of $D$-saddles, $D$-attractors and $D$-repellers, respectively.
Notice that a $D$-saddle point is a three dimensional saddle, whereas a $D$-node can be a three dimensional saddle (in which case $e(p)=1$ and $W^2_p\subset D$) or a three dimensional attractor or repeller (the three eigenvalues have real part with the same sign). 
A $D$-saddle point $p$ is called a \emph{transversal saddle point} if  $W_p^2\not\subset D$ or a \emph{tangential saddle point} if $W_p^2\subset D$. In the last case, $W^2_p$ coincides with one of the components of $D$, as germs at $p$. Notice that a transversal saddle is either an angle or a trace point, whereas a tangential saddle is either an angle or a corner point (see Figure \ref{Fig:Transversal-tangential}). In the remain of the paper we denote by $S_{tr}\subset S$ the set of transversal saddle points and by $S_{tg}\subset S$ the set of tangential saddle points.

\begin{figure}
\begin{center}
	\includegraphics[scale=0.65]{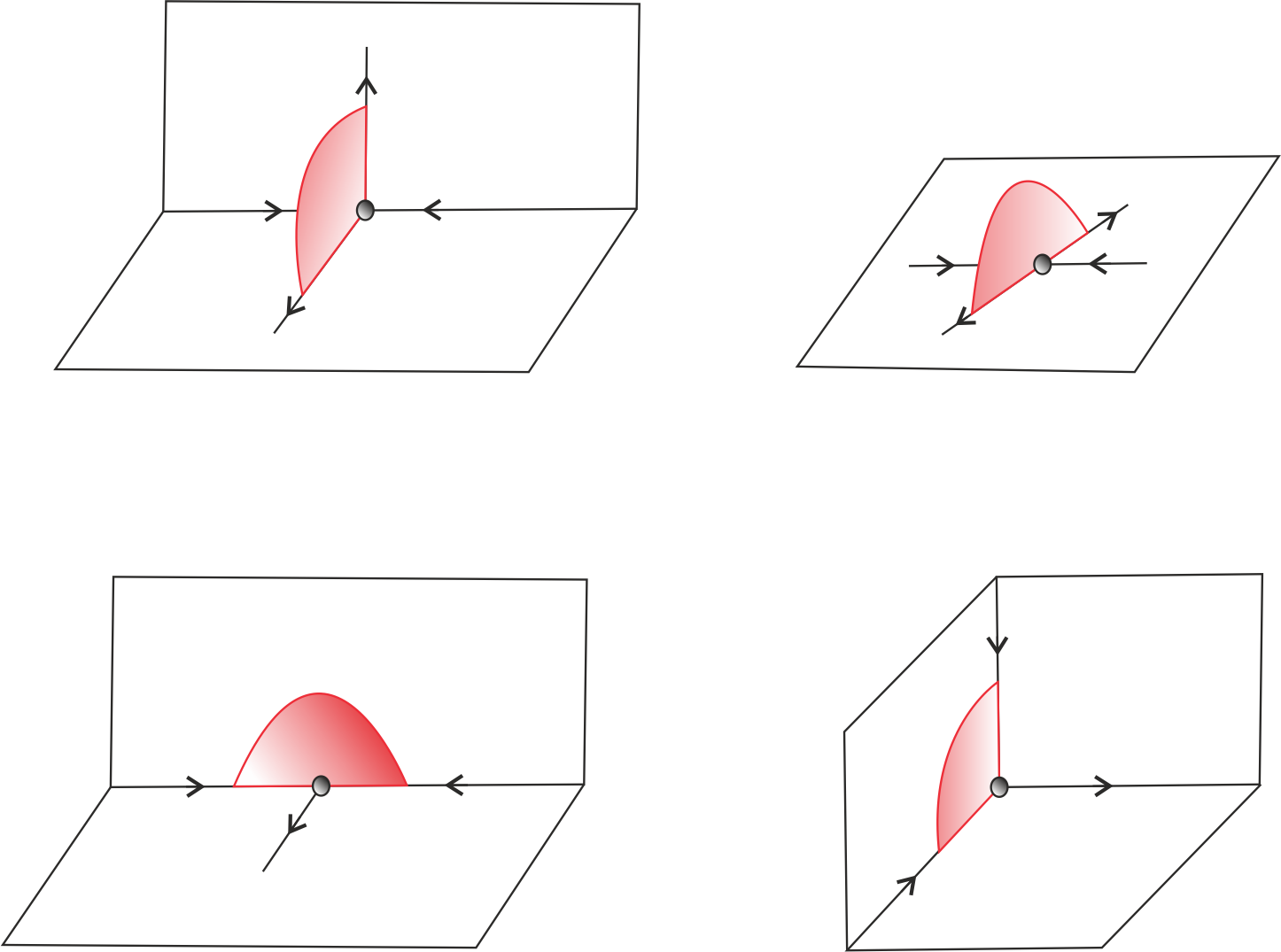}
\end{center}
\caption{Transversal saddles (above) and tangential saddles (below).}
\label{Fig:Transversal-tangential}
\end{figure}

\begin{remark}\label{rk:asymptotics-HAS}
\em{Concerning the leaves asymptotic behaviour, observe that, using the acyclicity condition and the Poincar\'{e}-Bendixson Theorem on the sphere (see for instance \cite{Pal-M,Per}), for any leaf
$\ell\subset D$, the sets $\alpha(\ell)$ and $\omega(\ell)$ are singletons and they do not coincide. On the other side, in light of the Hartman-Grobman's Theorem (see \cite{Har,Gro}) at hyperbolic points, we can extend this property in a sufficiently small open neighborhood $U$ of $D$: the $\alpha$-limit set (and also the $\omega$-limit set) of an $U$-leaf is either a singular point in $D$ or an empty set. }
\end{remark}

\subsection{The associated graph.}

Let $\mathcal{M}=(M,D,\LL)$ be a HAFVSD. We consider the directed planar graph $\Omega=\Omega_\mathcal{M}$ defined in the following way:
\begin{itemize}
  \item The set of vertices of $\Omega$ is $V(\Omega):=\Sing(\LL)$. 
  \item The set of edges of $\Omega$, denoted by $E(\Omega)$, consists of the non-singular leaves $\sigma$ of $\LL|_D$ satisfying one of the following (non-exclusive) properties:
  \begin{enumerate}[(a)]
  	\item either $\sigma$ is contained in the skeleton, 
  	\item or $\sigma$ is contained in a component $D_i$ of $D$ for which at least one of the limit points $p\in\{\alpha(\sigma),\omega(\sigma)\}$ is a two-dimensional saddle point of the restriction $\LL|_{D_i}$ (hence, at any such limit point $p$ the germ of $\sigma$ coincides with one of the local invariant one-dimensional manifolds of $\mathcal{L}|_{D_i}$ at $p$).
  	\end{enumerate}
  \item An edge $\sigma$ is adjacent to a vertex $p$ if $p=\epsilon(\sigma)$ for $\epsilon=\alpha$ or $\epsilon=\omega$.
  \item The orientation of the edges of $\Omega$ is the one induced by the one of $\mathcal{L}$.
\end{itemize}

By Remark~\ref{rk:asymptotics-HAS}, adjacency in $\Omega$ is well defined and any edge $\sigma$ is adjacent to exactly two distinct vertices, its $\alpha$ and $\omega$-limit points. Notice also that both $V(\Omega)$ and $E(\Omega)$ are finite sets. We put $\sigma=[p,q]$ if $p=\alpha(\sigma)$ and $q=\omega(\sigma)$ and use expressions of the form ``$\sigma$ starts at $p$ and ends at $q$''. An edge contained in $Sk(D)$ is called a \emph{skeleton edge}, otherwise is called a \emph{trace edge}. As usual, a {\em path of edges}  is a sequence of edges $\g=(\sigma_1,...,\sigma_r)$ such that $\alpha(\sigma_{j+1})=\omega(\sigma_j)$ for $j=1,...,r-1$ and we say that $\g$ starts at $\alpha(\sigma_1)$ and ends at $\omega(\sigma_r)$. A path of edges starting and ending at the same vertex is called a {\em cycle}.
We want to emphasize the following important property of the graph $\Omega$ which is a consequence of the acyclicity condition in Definition~\ref{def:SHAFD}:
\begin{quote}
{\em The oriented graph $\Omega$ has no cycles}.
\end{quote}

We will adopt the usual terminology for graphs. For instance: a {\em subgraph} of $\Omega$ is a directed graph $G$ for which $V(G)\subset V(\Omega)$ and $E(G)\subset E(\Omega)$ (we write $G<\Omega$); the {\em subgraph generated by} a subset of vertices $W\subset V(\Omega)$ is the subgraph $G=G(W)<\Omega$ such that $V(G)=W$ and $\sigma\in E(G)$ if, and only if, the vertices adjacent to $\sigma$ belong to $W$; the {\em subgraph generated by} a subset of edges $F\subset E(\Omega)$ is the subgraph $G=G(F)<\Omega$ such that $E(G)=F$ and $V(G)$ is the set of vertices adjacent to some $\sigma\in F$; the {\em edge-complement} of a subgraph $G$ {\em in } $\Omega$ is the subgraph, denoted by $G^c$, generated by the set of edges that do not belong to $E(G)$, that is $G^c=G(E(\Omega)\setminus E(G))$. For instance, if $\g=(\sigma_1,\dots,\sigma_n)$ is a path of edges, we see $\g$ also as the subgraph generated by the set of edges $\{\sigma_1,\dots,\sigma_n\}$. If $G<\Omega$ is a subgraph, the {\em support} of $G$ is the compact subset of $D$ defined by
$$
|G|=V(G)\cup\bigcup_{\sigma\in E(G)}\sigma.
$$
 A {\em face} of $\Omega$ is a connected component of $D\setminus|\Omega|$. It is an open subset of $D$, saturated by $\LL$ and contained in a unique component of $D$. In addition, $\Omega$ induces a stratification on $D$, denoted by $St_\Omega(D)$, where faces, edges and vertices are the strata of dimension 2, 1 and 0, respectively. Due to the possible presence of trace edges, such a stratification is finer than the stratification $St(M)|_D$ induced by the standard stratification in $M$.

We summarize some of the properties satisfied by $\Omega$ in the following result.
\begin{lemma}\label{lm:graph} Let $\Omega$ be the graph associtaed to a HAFVSD $\mathcal{M}$. Then it holds: 
\begin{enumerate}[(i)]
  \item For any $p\in S$ there exist at least a path of edges $\g$ passing through $p$ such that $\alpha(\g)\in N^r$ and $\omega(\g)\in N^a$.
    \item We have $\#N-\#S_{tr}=2$. In particular $N\ne\emptyset$. Moreover, $N^a\neq\emptyset$ and $N^r\neq\emptyset$.
     \item A vertex $p\in V(\Omega)$ is a transversal saddle point if, and only if, there are exactly four edges adjacent to $p$ such that two of them start at $p$ and the other two end at $p$.
     \item If there exists a vertex $p$ which is isolated in $\Omega$ (no edge is adjacent to $p$), then $V(\Omega)=\{p,q\}$, where $p,q$ are two different $D$-nodes, and $E(\Omega)=\emptyset$.
  \item For any face $\G$ of $\Omega$ there exist two different vertices $p,q\in V(\Omega)$ such that
   any leaf $\ell$ in $\G$ satisfies $\alpha(\ell)=p$ and
   $\omega(\ell)=q$. (We write $p=\alpha(\G)$ and $q=\omega(\G)$). Moreover, either we are in the situation of item (iv), that is, $\G=D\setminus\{p,q\}$, or    
   the topological frontier of $\G$ in $D$ is the support of a subgraph $\FF(\G)<\Omega$ consisting of exactly two paths of edges from $p$ to $q$ (which may share some edges but only in an initial and/or final segment of the path).
\end{enumerate}
\end{lemma}
\begin{proof}
{\em (i)} Let $D_i$ be a component of $D$ for which $\LL|_{D_i}$ has a two-dimensional saddle at $p$. The stable and unstable manifolds of $\LL|_{D_i}$ at $p$ are contained in respective edges $\sigma$, $\tau$ of $\Omega$ such that $\omega(\sigma)=p$ and $\alpha(\tau)=p$. If $\alpha(\sigma)$ is a $D$-node, this is a starting point for the required path $\g$. Otherwise, we restart the same argument for the $D$-saddle point $\alpha(\sigma)$. Doing analogously with the point $\omega(\tau)$ and using the acyclicity condition, the result follows.

{\em (ii)} Following the ideas of Brunella in \cite{Bru}, it is possible to assign a {\em Poincar\'{e} index} to the restricted (continuous) foliation $\LL|_D$ at any singularity $p\in V(\Omega)$. It is easy to check that this index is equal to 1, 0 or -1 when $p$ is a $D$-node, a tangential saddle or a transversal saddle point, respectively. Thus, the formula is a consequence of the Theorem of Poincar\'{e}-Hopf in the sphere (see \cite{Per} for differentiable vector fields, although it also works for continuous vector fields with isolated singularities). To show the last claim, assume that $p$ is a $D$-node and $p\in N^a$, for instance. We take a leaf $\ell$ of $\LL$ contained in $D$ such that $\omega(\ell)=p$ and we conclude that $\alpha(\ell)=q\in\Sing(\LL)$, where $q$ is either a $D$-saddle or a $D$-repeller point. In the second case, we are done. In the first case, we use item (i).

{\em (iii)} Consider $p\in\Sing(\LL)$. If $p$ is a $D$-node, then either any edge adjacent to $p$ starts at $p$ (when $p\in N^r$) or any edge adjacent to $p$ ends at $p$ (when $p\in N^a$). On the other hand, suppose that $p$ is a tangential saddle and assume, for instance, that $W^2_p$ is the stable manifold at $p$. Then there is a unique edge $\sigma$ that starts at $p$. Indeed, if $D_j$ is a component of $D$ at $p$ not containing $W^2_p$, then $\LL|_{D_j}$ has a two-dimensional saddle point at $p$ and hence its unstable manifold, equal to $W^1_p$, is a curve and $W^1_p\setminus\{p\}$ is contained in an edge $\sigma$ (notice also that $W^1_p$ has only one side in this case, since it is transversal to $D$). Finally, suppose that $p\in S_{tr}$. If for instance $W^2_p$ is stable, then the two connected components of $W^2_p\cap D\setminus\{p\}$ are contained in corresponding edges ending at $p$, whereas the two connected components of $W^1_p\setminus\{p\}$ are contained in corresponding edges starting at $p$.

{\em (iv)}
If $p$ is an isolated vertex, then $p$ is a $D$-node by item (i). Moreover, there is a unique component, say $D_0$, of $D$ at $p$ (i.e., $e(p)=1$). Assume, for instance, that $p\in N^r$.  Let $B$ be a compact neighborhood of $p$ in $D$, homeomorphic to a disc, whose boundary $\partial B$ is smooth and transversal to $\LL$. We have that any nonsingular leave of $\LL|_D$ issued at a point in $B$ cuts $\partial B$ at a single point. In fact, we may identify $\partial B$ with the set $H$ of non-singular leaves $\ell$ in $D$ with $\alpha(\ell)=p$, so that we equip $H$ with the topology of $\partial B$. Using Remark~\ref{rk:asymptotics-HAS}, for any $\ell\in H$ we must have $\omega(\ell)=q_\ell$, where $q_\ell\ne p$. It remains to show that $q_\ell$ does not depend on $\ell$. Let $q$ be the $\omega$-limit point of a fixed element $\ell_0$ of $H$ and let us show that the set $H_{q}=\{\ell\in H\,:\,\omega(\ell)=q\}$ is open and closed in $H$. 
To see that $H_{q}$ is open in $H$, notice that if $\ell\in H_{q}$, then the germ of $\ell$ at $q$ is contained in the stable manifold of $\LL$ at $q$. Moreover, if $\ell$ is contained in a component $D_i$ of $D$, then $q$ must be a two-dimensional attractor of $\LL|_{D_i}$ (otherwise, $\ell$ would be contained in $W^1_q$ and $\ell$ would be an edge adjacent to $p$). Similarly, we show that $H_q$ is closed in $H$: given $\ell\in\overline{H_q}$ and taking $q_\ell=\omega(\ell)$, we have that $\ell\in H_{q_\ell}$ and we have shown above that $H_{q_\ell}$ is open. Hence any $\ell'\in H$ in a neighborhood of $\ell$ would satisfy $\omega(\ell')=q_\ell$. In particular, if we take such $\ell'$ in $H_q$ we have $q=q_\ell$, which shows that $\ell\in H_q$. 

We can see that $|H|:=\bigcup_{\ell\in H}\ell$ is a face of $\Omega$: it is connected, open inside $D_0$, disjoint with $|\Omega|$ and closed inside the two-dimensional stratum of $St(M)$ contained in $D_0$ (using similar arguments as above).

Let us show that $q$ is also an isolated vertex. Notice that $q\in D_0$, since $\ell_0\subset D_0$. Moreover, the germ of any element $\ell\in H$ at $q$ is contained in the stable manifold $W^s(q)$. Since $H$ is infinite, $W^s(q)$ must be of dimension at least 2 and must contain the germ of $D_0$ at $q$. If $q$ is not an isolated vertex, we must have an edge $\sigma_1$ ending at $q$ which is contained in the closure of $|H|$ (if there are edges adjacent to $q$ not in $D_0$ we must have $e(q)>1$ so that necessarily there are such edges contained in $D_0$). Let $q_1=\alpha(\sigma_1)$. We have that $q_1\ne p$ because $p$ is an isolated vertex. Moreover, since $q_1$ is in the closure of $|H|$, we could not have that $q_1$ is a repeller point of the restriction $\LL_{D_0}$ (since, otherwise, we would have $\alpha(\ell)=q_1$ for several $\ell\in H$, contrary to $q_1\ne p$). Thus, $q_1$ is a two-dimensional saddle point of $\LL_{D_0}$ and we obtain an edge $\sigma_2$ ending at $q_1$ and contained in the closure of $|H|$. Repeating the argument, we find a sequence of vertices $q,q_1,q_2,...$, all saddle points of the restriction $\LL_{D_0}$, contradicting the acyclicity condition.
 
 We conclude that $D_0$ is the unique component of $D$ at $q$ and that $q\in N^a$. As we have shown above, $H$ is the family of leaves $\ell$ satisfying both $\alpha(\ell)=p$ and $\omega(\ell)=q$. Taking a compact neighborhood $B'$ of $q$ in $D_0$ whose boundary cuts exactly once each element $\ell\in H$, the flow establishes a homeomorphism $\partial B\to\partial B'$ given by $a\mapsto\ell_a\cap\partial B'$. By standard arguments, one proves that the subset $\{p,q\}\cup\bigcup_{\ell\in H}\ell$ of $D$ is homeomorphic to $\SSS^2$ and then it is equal to $D$, which gives the result.

{\em (v)} Notice that the face $\G$ is saturated by $\LL$ and contained in a single component of the divisor. Denote such component by $D_\G$. If $\ell$ is a leaf contained in $\G$, then the limits $p=\alpha(\ell)$ and $q=\omega(\ell)$ are different singular points that belong to $D_\G$. Let us see that they do not depend on $\ell$. We have that $p$ and $q$ are nodes (repeller and attractor, respectively) of the restriction $\LL|_{D_\G}$: for instance, if $p$ is a saddle of $\LL|_{D_\G}$, its invariant stable and unstable manifolds, being both disjoint with $\ell$, would be contained in $\alpha(\ell)$, which is not possible. Similarly to the item (iv), we obtain that the subset $R\subset\G$ defined by the union of leaves inside $\G$ with limits at $p$ and $q$, is open and closed in $\G$. Hence, $R=\G$, as wanted. The second part of statement (v) is a general result in the theory of planar graphs.
\end{proof}

\begin{figure}
\begin{center}
	\includegraphics[scale=0.55]{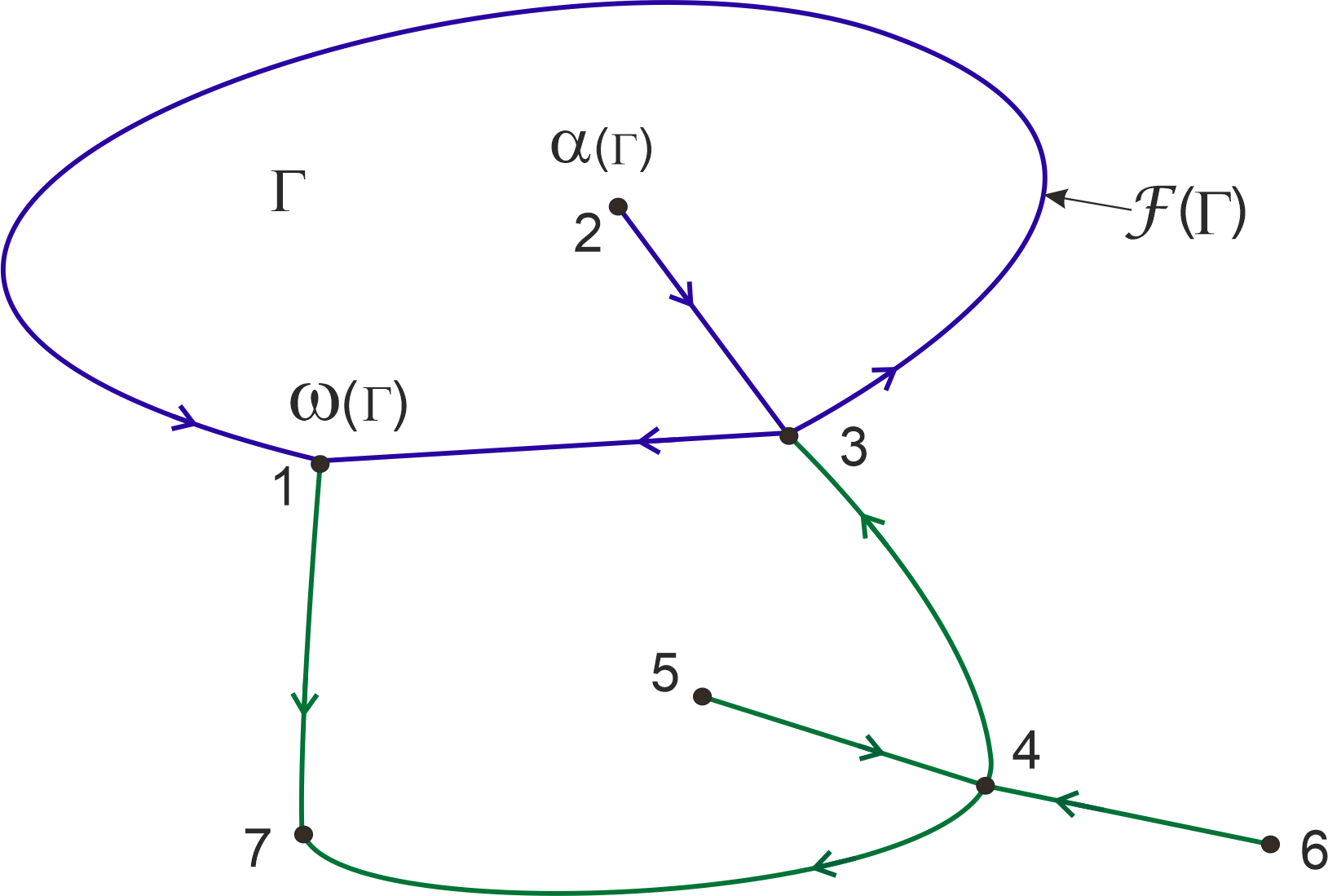}
\end{center}
	\caption{A graph $\Omega$ with three components. Points $3,4$ are transversal saddles, $1$ is a tangential saddle whereas $2,5,6$ and $7$ are nodes.}
	\label{Fig:Transversal-tangential}
\end{figure}

The acyclicity condition permits us to state the following definition.

\begin{definition}\label{def:length}
For any vertex $p \in V(\Omega)$ we define the {\em length} of $p$, and denote it by $l(p)$, as the maximal number of edges in a path of edges starting at $p$. The value $l(\Omega)=\max\{l(p):p \in V(\Omega)\}$ is called the {\em length} of $\Omega$. 
\end{definition}

Put $l:=l(\Omega)$. For $j=0,1,...,l$, we denote by $\Omega^j$ the subgraph of $\Omega$ whose set of vertices is $V(\Omega^j)=\{p\in V(\Omega)\,:\,l(p)\leq j\}$ and whose set of edges is the set of edges of $\Omega$ starting at a point in $V(\Omega^j)$. We have the following filtration of $\Omega$, which will be very useful in what follows: 
\begin{equation}\label{eq:filtration-graph}
	\Omega^0<\Omega^1<\cdots<\Omega^{l-1}<\Omega^l.
\end{equation}
Notice that we have  $V(\Omega^j)\setminus V(\Omega^{j-1})=\{p\in V(\Omega)\,:\,l(p)=j\}$, for $j=1,...,l$. Concerning the edges, we have $E(\Omega^0)=\emptyset$ and, for $j\ge 1$,
\begin{equation}\label{eq:edges-j}
	E(\Omega^j)\setminus E(\Omega^{j-1})=\bigcup_{p\in V(\Omega^j)\setminus V(\Omega^{j-1})}\alpha^{-1}(p).
\end{equation}
Notice also that $V(\Omega^0)=N^a$, except for the exceptional case described in Lemma~\ref{lm:graph} item (iv), in which $\Omega^0=\Omega$ consists of two isolated vertices (by the way, the unique case where $\Omega$ has length equal to zero).

\subsection{The local s-components} 
From now on, it will be convenient to consider separately the two sides in which the two-dimensional invariant manifold divides a neighborhood at a transversal saddle point. To unify notation, we define such ``sides'' at any singular point $p\in\Sing(\LL)$. 
More precisely, take an analytic chart $(U_p,\xx)$ at $p$ where the domain $U_p$ is small enough so that, when $p$ is a three-dimensional saddle, the two-dimensional invariant manifold $W^2_p$ is well defined in $U_p$ and given by a plane coordinate in the $\xx$ variables. In this case, $\wt{D}_p=(D\cap U_p)\cup W^2_p$ is a normal crossing divisor, contained in the union of the plane coordinates of the chart. When $p$ is not a three-dimensional saddle point, we simply put $\wt{D}_p=D\cap U_p$ and require the condition of $U_p\setminus\wt{D}_p$ being connected.
\begin{definition}\label{def:s-component}
	The germ at $p$ of a connected component of $U_p\setminus\wt{D}_p$ does not depend on the chart $(U_p,\xx)$ and will be called a {\em local saddle-component} (or just a {\em local s-component} for short) {\em at} $p$.
\end{definition}
By definition, if $p$ is a $D$-node or a tangential saddle point then there is exactly one local s-component at $p$. On the contrary, if $p$ is a transversal saddle point, then there are two local s-components at $p$.

We use the generic letters $\nu,\eta$, etc. either to denote local s-component at some point or representatives of them. In this way, the notation  $\overline{\nu}$ refers, either to the germ of the closure of a representaive of $\nu$ or to a representative of it. Besides, if $p$ is a $D$-node or a tangential saddle point, we denote $\nu=\nu_p$ the unique local s-component at $p$. If $p$ is a transversal saddle point, we denote by $\nu^+_p,\nu^-_p$ the two local s-components at $p$. Denote also by $\VV(\Omega)$ the set of all local s-components at singular points. More generally, if $G<\Omega$ is a subgraph, we denote by $\VV(G)$ the set of local s-components at points in $V(G)$.

We consider a partial ordering in the set $\VV(\Omega)$ of local s-components by putting $\nu\leq\mu$ if, and only if, either $\nu=\mu$ or there exists a path of edges $(\sigma_1,...,\sigma_r)$ such that $\sigma_1\cap\overline{\nu}\ne\emptyset$ and $\sigma_r\cap\overline{\mu}\ne\emptyset$.

Finally, we extend the notions of $\alpha$ and $\omega$-limits of edges stated in (\ref{eq: alfa y omega limite}) as follows: 
given an edge $\sigma\in E(\Omega)$ and $\nu\in\VV(\Omega)$, we say that $\nu$ is an $\tilde{\alpha}$-{\em limit} (respectively a $\tilde{\omega}$-{\em limit}) of $\sigma$, if the germ of $\sigma$ at $\alpha(\sigma)$ (respectively at $\omega(\sigma)$) is contained in $\overline{\nu}$. We consider then $\tilde{\alpha},\tilde{\omega}$ as correspondences from $E(\Omega)$ to $\VV(\Omega)$ so that we use, for instance, the notation $\sigma\in\tilde{\alpha}^{-1}(\nu)$ to indicate that $\nu$ is an $\tilde{\alpha}$-limit of $\sigma$. In addition, if $\G$ is a face of $\Omega$, we put $\tilde{\alpha}(\G)$ and $\tilde{\omega}(\G)$ to denote the unique local s-component at the vertex $\alpha(\G)$ and $\omega(\G)$, respectively.

\section{Morse-Smale non-resonant foliations}\label{sec:MS-foliations}

Let $\mathcal{M}=(M,D,\LL)$ be a HAFVSD and $\Omega$ its associated graph. An edge $\sigma=[p,q]$ of $\Omega$ is called a {\em saddle connection} if there exists a component $D_i$ of $D$ such that $\sigma\subset D_i$ and $p,q$ are both two-dimensional saddles of the restriction $\LL|_{D_i}$.
In this situation, notice that $\sigma$ is locally contained in the unstable manifold (respectively stable manifold) of $\LL|_{D_i}$ at $p$ (respectively at $q$). A {\em multiple saddle connection} is a path of edges $\g=(\sigma_1,\dots,\sigma_r)$ such that each $\sigma_j$ is a saddle connection.

\begin{definition}\label{def:Morse-Smale}
The HAFVSD $\mathcal{M}$ is said to be of {\em Morse-Smale type} if any saddle connection is a skeleton edge.
\end{definition}

The Morse-Smale condition implies, in particular, that if $\sigma$ is a trace edge of $\Omega$ and  $D_i$ is the (unique) component of $D$ containing $\sigma$, then one of the extremities of $\sigma$ is a saddle and the other one is a node of the restricted foliation $\mathcal{L}|_{D_i}$.
More precisely, being $\sigma=[p,q]$, we must have:

- If $p$ is a saddle point of $\LL|_{D_i}$, then $q$ can be either a $D$-node or a tangential saddle for which $\mathcal{L}|_{D_i}$ has a node at $q$. Notice that in this last case, $q$ is either an angle or a corner point.

- If $p$ is node point of $\LL|_{D_i}$, then $q$ is a saddle point of $\LL|_{D_i}$ (in this case $q$ cannot be a corner point).

\strut

As discussed in the introduction, we impose also a requirement concerning a ``non-resonance'' condition of multiple saddle connections. The rest of this section is devoted to explain what is behind this condition and to stablish some relevant consequences for the dynamics. Let us mention that such condition was introduced for the first time in Alonso-González et al. \cite{Alo-C-C1, Alo-C-C2}, where it was stated as the absence of ``infinitesimal saddle connections'' in the context of a topological classification of vector fields with Morse-Smale type reduction of singularities. 

\subsection{Trace marks}
First, we define what we call trace marks, curves attached to points in the ``middle'' of the two dimensional invariant manifold at a saddle point. 

If $p\in\Sing(\LL)$ is a three dimensional saddle point, we denote by $\wt{D}_p$ the germ at $p$ (or a representant of it) of $D\cup W^2_p$, a normal crossing divisor extending $D$. Notice that $\wt{D}_p$ coincides with the germ of $D$ at $p$ if $p$ is a tangential saddle point. Denote by $Sk(\wt{D}_p)$ the union of intersections of pairs of components of $\wt{D}_p$ (in particular $Sk(D)\subset Sk(\wt{D}_p)$). 

\begin{definition}\label{def:trace-mark}
	Let $p\in\Sing(\LL)$ be a three-dimensional saddle point and let $\nu$ be a local s-component at $p$. 
	A {\em trace mark on $\nu$} is the image $T=\beta([0,\infty))$ of an injective analytic parameterized curve $\beta:[0,\infty)\to\nu$ such that there exists $b:=\lim_{t\to\infty}\beta(t)$ with $b\in W^2_p\setminus Sk(\wt{D}_p)$. If we need to specify the limit point, we will say that $T$ is a trace mark {\em attached to $b$}.
\end{definition}

\begin{remark}\label{rk:mark-from-w2-to-w1}{\em 
		Suppose that $p$ is a three-dimensional saddle point, that $\nu$ is a local s-component and that $T$ is the image of an injective analytic curve contained in $\nu$ such that $\overline{T}\cap W^2_p$ is a single point $b$ with $b\ne p$ (a trace mark on $\nu$, for instance, although more generally we may have $b\in Sk(\wt{D}_p)$). By means of the Hartman-Grobman's Theorem, if $b$ is close enough to $p$ and $V$ is a  sufficiently small neighborhood of $p$ containing $b$, we have:
		$$(\overline{\Sat_{V}(T)}\cap V)\setminus\Sat_V(T)=(\ell_b \cup (W^1_p\cap\overline{\nu}))\cap V$$
		where $\ell_b$ is the leaf of $\LL$ at $b$ (so that $\ell_b\cap V$ is contained in $W_p^2$ if $V$ is small enough). Notice that, if $p$ is a $D$-saddle point, there is a unique edge $\sigma$ adjacent to $p$ such that $\sigma\cap\nu\ne\emptyset$ and such $\sigma$ contains $W^1_p\cap\nu$. On the contrary, if $p$ is a $D$-node, then $W^1_p\cap\nu$ does not intersect the divisor $D$.
	}
\end{remark}

\subsection{Angle marks} 
Next, we describe the curves attached to points in the (extended) skeleton that we will consider. They have some explicit ``tangency order'' with respect to the components of the divisor.
We start with the following definition, that adapts the one that already appears in Section $6$ of  \cite{Alo-C-R}.

\begin{definition}\label{def:quasi-order}
Let $c:[0, \infty) \rightarrow \mathbb{R}^2_{>0}$ be an injective analytic parameterized curve such that $\lim_{t \to \infty}c(t)=(0,0)$. Write $c(t)=(c_1(t),c_2(t))$ in the cartesian coordinates. We say that\emph{ $c$ has a quasi-order} if there exist $\rho>0$, $t_0>0$ and constants $k_1,k_2$ with $0<k_1<k_2< \infty$ such that
$$k_1c_1(t)^{\rho}< c_2(t) < k_2 c_1(t)^{\rho} \qquad \textrm{for any } t\ge t_0.$$
In this case, the univocally determined value $\rho$ is called the {\em quasi-order} of $c$ (with respect to the $x$-axis). A subset $C \subset \mathbb{R}^2_{>0}$ is said to be a {\em  curve with quasi-order $\rho$} if $C$ is the support of a parameterized analytic real curve $c$ with quasi-order $\rho$.
\end{definition}
Notice that if $C\subset\R^2_{>0}$ has quasi-order $\rho$ then $\rho$ does not depend on the parametrization $c$. Moreover, given $\psi:\R^2_{\ge 0}\to\R^2_{\ge 0}$ an analytic isomorphism, a subset $C\subset\R^2_{>0}$ has quasi-order if, and only if, $\psi(C)$ also does. If $\rho$ is the quasi-order of $C$, then the quasi-order of $\psi(C)$ is $\rho$, in case $\psi$ preserves each coordinate axis, or $1/\rho$, when $\psi$ inverts them. In particular, the quasi-orders of a curve with respect to the $x$ and the $y$-axis are mutually inverse.

 Let $a\in D$ be a non-singular point of $\LL$. A {\em planar section at $a$} is a germ of two-dimensional analytic submanifold $\Delta$ of $M$ passing through $a$ and everywhere transversal to the foliation $\LL$. We recall that the word ``submanifold'' is taken with the meaning of analytic manifold with boundary and corners. In particular, $\partial\Delta=\Delta\cap D$ and $(\Delta,a)$ is either isomorphic to a closed half space $(\R\times\R_{\ge 0},0)$ if $e(a)=1$, or to a quadrant $(\R^2_{\ge 0},0)$, if $e(a)=2$.

\begin{definition}\label{def:angle-mark}
Let $\sigma$ be a skeleton edge and let $a\in\sigma$. An {\em angle mark at $a$} is a subset $\Sigma$ of a planar section $\Delta$ at $a$ such that there exists an analytic chart $\psi:\Delta\to\R^2_{\ge 0}$ so that $\psi(\Sigma)$ is a curve with quasi-order. In this case, if $D_1,D_2$ are the two components of $D$ such that $\sigma\subset D_1\cap D_2$ then, for $i=1,2$, the {\em quasi-order of $\Sigma$ with respect to $D_i$} (or {\em $D_i$-quasi-order}) is the quasi-order $\rho_i$ of $\psi(\Sigma)$  with respect to the coordinate axis $\psi(\Delta\cap D_i)$ (and thus $\rho_1\rho_2=1$).
\end{definition}
Notice that if $\Delta'$ is a planar section at a different point $a'\in\sigma$ and $\phi:\Delta\to\Delta'$ is the analytic isomorphism induced by the foliation $\LL$, then $\Sigma\subset\Delta$ is an angle mark at $a$ iff $\phi(\Sigma)\subset\Delta'$ is an angle mark at $a'$ with the same quasi-order. Any such curve $\phi(\Sigma)$ will be said to be an angle mark {\em associated to $\sigma$} without explicit mention of the point $a'$.
\subsection{Transition of trace and angle marks}
Denote by $S'\subset S$ the set of $D$-saddle points for which the one-dimensional invariant manifold is contained in the skeleton. In other words, the set $S'$ is composed of the tangential saddle corner points and  the transversal saddle angle points.
In the following two propositions, we adapt a couple of results that appear in \cite{Alo-C-C1}, in order to describe the transition of trace and angle marks through a saddle point $p \in S'$. Roughly speaking, the saturation of a trace mark near a point $p \in S'$ produces an angle mark associated to the skeleton edge that supports $W^1_p$.

 Given $p \in S'$, if $D_i,D_j$ are the two components of $D$ such that $W^1_p\subset D_i\cap D_j$ and $\lambda_k$ is the eigenvalue of a local generator of $\LL$ at $p$ associated to the invariant curve $D_k\cap W^2_p$, for $k=i,j$, we define the {\em $D_i$-weight of $p$} (respectively the {\em $D_j$-weight of $p$}) as the value
 $$
 w_{D_i}(p)=\lambda_j/\lambda_i\,\mbox{(respectively } w_{D_j}(p)=\lambda_i/\lambda_j\mbox{)}.
 $$
Observe that the weights are independent of the chosen local generator and that they are mutually inverse positive real numbers.

\begin{proposition}\label{pro:trace-to-angle}
Let $p\in S'$ and let $\nu$ be a local s-component at $p$. Denote by $\sigma$ the edge containing $W^{1}_p\cap\nu$ and by $D_i,D_j$ the components of $D$ such that $\sigma\subset D_i\cap D_j$. Let $T$ be a trace mark on $\nu$. Then, in a sufficiently small neighborhood $V$ of $p$, if $\Delta$ is a planar section at some $a\in\sigma\cap V$, the curve $\Sigma=\Delta\cap\Sat_{V}(T)$ is an angle mark at $a$ with $D_k$-quasi-order equal to $w_{D_k}(p)$, for $k=i,j$.
\end{proposition}
\begin{proof} Let $(V,(x,y,z))$ be a chart at $p$ such that $D_i\cap V=\{z=0\}$, $D_j\cap V=\{y=0\}$, $W^2_p=\{x=0\}$ and $\nu=\{x>0\}$.  For any given $\rho>0$, the {\em $\rho$-weighted blow-up with center $Y=D_i\cap D_j\cap V$}  is the map $\pi_\rho:M\to V$, where $M$ is the analytic manifold with boundary and corners constructed from two charts $(M_i,(x_i,y_i,z_i))$, $i=1,2$, by identifying points in $M_1\setminus\{z_1=0\}$ with points in $M_2\setminus\{y_2=0\}$ so that $\pi_\rho$ is well defined by the respective expressions
	$$
	 x=x_1,\,y=y_1,\,z=z_1y_1^{\rho};\;\;\;\;x=x_2,\,y=y_2z_2^{1/\rho},\,z=z_2
	$$
(see \cite{Pan} or \cite{Mar-R-S} for intrinsic definitions). 
The map $\pi_\rho$ is a continuous proper surjection and restricts to an analytic isomorphism outside the exceptional divisor $H=\pi_\rho^{-1}(Y)$ (given by equations $\{y_1=0\}$ and $\{z_2=0\}$ in $M_1$ and $M_2$, respectively).
Notice that, given the planar section $\Delta_a=\{x=x(a)\}$ at some $a\in\sigma$ and the image $\Sigma\subset\Delta_a\setminus\{a\}$ of a parameterized analytic injective curve, $\Sigma$ is an angle mark with $D_i$-quasi-order equal to $\rho$ if, and only if, $\pi^{-1}_\rho(\Sigma)$ accumulates on the fiber $\pi^{-1}_\rho(a)$ along a subset contained in a segment inside $M_1\cap M_2$ of the form $\{y_1=0,x=x(a),k_1\leq z_1\leq k_2\}$, where $0<k_1<k_2$.

The transformed foliation $\wt{\LL}=\pi_\rho^*\,\LL|_V$ on $M\setminus H$ extends continuously to $H$ leaving $H$ invariant. More precisely, consider an analytic local generator $\xi$ of $\LL$ at $p$ written as
$$
\xi=x(\alpha+A)\partial_x+y(\lambda_i+B)\partial_y+
z(\lambda_j+C)\partial_z
$$  with $\alpha\lambda_i<0$, $\lambda_i\lambda_j>0$ and $A(0)=B(0)=C(0)=0$. The transformed vector field $\wt{\xi}=\pi_\rho^*\xi$ is a generator of $\wt{\LL}$ on $M\setminus H$. We have that $\xi_1=\wt{\xi}|_{M_1}$ is written as:
\begin{equation}\label{eq:xi1}
\xi_1=x_1(\alpha+A\circ\pi_\rho)\partial_{x_1}+
y_1(\lambda_i+B\circ\pi_\rho)\partial_{y_1}+
z_1\left((\lambda_j-\rho\lambda_i)+(C-\rho B)\circ\pi_\rho\right)\partial_{z_1}
\end{equation}
(and a similar expression for $\xi_2=\wt{\xi}|_{M_2}$). The two vector fields $\xi_1,\xi_2$ extend continuously to $H$ and, in fact, these extensions are of class $C^1$ at any point of $M_1\cap M_2\cap H$, independently of $\rho$. When $\rho<1$ (respectively $\rho>1$),  it is possible that $\xi_1$ (respectively $\xi_2$) is not of class $C^1$ at the origin of $M_1$ (respectively of $M_2$). If this is the case for instance for $\xi_1$, we can make a change of variables $y_1\leadsto y_1^{1/\rho}$ (valid in $M_1\setminus H$) so that $\xi_1$ transforms into a vector field written similarly as in equation (\ref{eq:xi1}) but replacing $\lambda_i$ with $\lambda_i\rho$ and $A\circ\pi_{\rho}$ with $\rho A(x_1,y_1^{1/\rho},z_1y_1)$, etc. This new vector field is then of class $C^1$ also at the origin of $M_1$.  

Consider the special weight $\rho=w_{D_i}(p)$. Taking into account the expression (\ref{eq:xi1}), we have the following properties:
\begin{enumerate}[(a)]
\item The leaves of $\wt{\LL}|_H$ are given by the lines $\{y_1=0,z_1=c\}_{c\in\R_{\ge 0}}$ in $M_1\cap H$ (or by $\{z_2=0,y_2=c\}_{c\in\R_{\ge 0}}$ in $M_2\cap H$).
\item The fiber $\pi_\rho^{-1}(0)=\pi_\rho^{-1}(W_p^2)\cap H$ is composed of singular points of $\wt{\LL}$.
\item At any point $t\in\pi_\rho^{-1}(0)\cap M_1\cap M_2$, the linear part of $\xi_1$ (or $\xi_2$) has one positive, one negative and one null eigenvalue. Moreover, the fiber $\pi_\rho^{-1}(0)$ is the (unique) center manifold $W^c(t)$ of $\wt{\LL}$ at $t$, whereas the center-unstable and center-stable manifolds satisfy $\{W^{cu}(t),W^{cs}(t)\}=\{H,\pi_\rho^{-1}(W^2_p)\}$.
\item Up to making a ramification of the variable $y_1$ (or of the variable $z_2$) as mentioned above, we have the same conclusion as in item (c) for $\xi_1$ at the origin of $M_1$ (or for $\xi_2$ at the origin of $M_2$). 
\end{enumerate}

\begin{figure}[h]
	\begin{center}
		\includegraphics[scale=0.70]{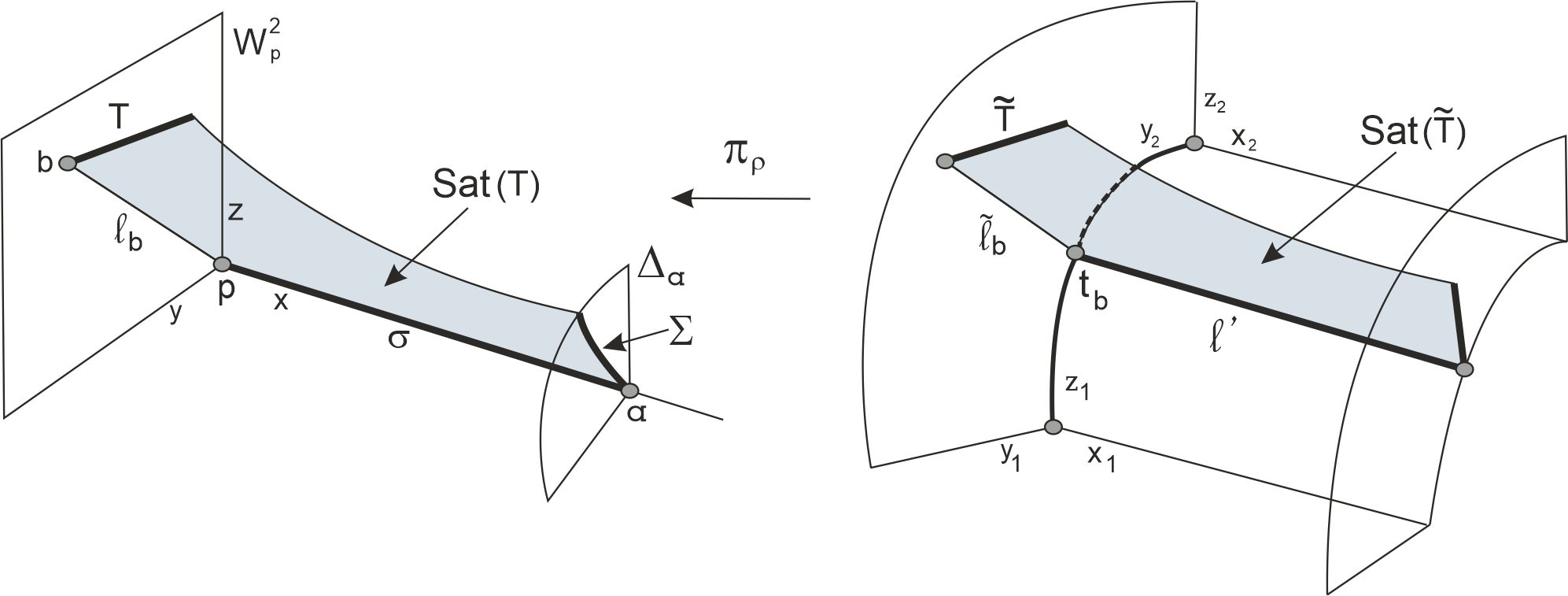}
	\end{center}
	\caption{Transition of a trace mark.}
	\label{Fig:TransitionTraceMark}
\end{figure}

Using the properties (b-d) above and applying the Theorem of Reduction to the Center Manifold of vector fields of class $C^1$ (see \cite{Hir-P-S}, for instance), we show that for a sufficiently small neighborhood $\wt{V}$ of $\pi_\rho^{-1}(0)$ and for any point  $t\in\pi_\rho^{-1}(0)$, there is exactly one non-singular leave of $\wt{\LL}|_{\wt{V}}$ contained in $\pi_\rho^{-1}(W^2_p)$ and accumulating to $t$.  

Now, suppose that $T$ is a trace mark attached to a point $b$. Notice that by definition of a trace mark we have $b\in W^2_p\setminus (D_i\cup D_j)$. Let $\ell_b\cap V\subset W^2_p$ be the leaf of $\LL|_V$ through $b$ and $\widetilde{\ell_b}=\pi_{\rho}^{-1}(\ell_b\cap V)$. If $V$ is sufficiently small, as we have said, $\widetilde{\ell_b}$ accumulates to a single point $t_b$ of the fiber $\pi^{-1}_\rho(0)$. Necessarily, $t_b$ is not the origin of none of the charts $M_1, M_2$ (because $b\not\in D_i\cup D_j$ and only the transforms by $\pi_\rho$ of the two edges intersecting $W^2_p\cap D_i$ or $W^2_p\cap D_j$  accumulate to those origin points). Again by the Theorem of Reduction to the Center Manifold applied to $\xi_1$ at $t_b$, if we put $\wt{T}=\pi^{-1}_\rho(T)$ and $\wt{V}=\pi_\rho^{-1}(V)$, we have that, $$\left(\overline{\Sat_\wt{V}(\wt{T})}\setminus\Sat_\wt{V}(\wt{T})\right)\cap\wt{V}=
\widetilde{\ell_b}\cup\{t_b\}\cup\ell',
$$
 where $\ell'$ is a non-singular leave of $\wt{\LL}|_{H\cap\wt{V}}$ that accumulates to $t_b$ (see Figure \ref{Fig:TransitionTraceMark}). By the property (a), $\ell'$ cuts any transversal section of the form $\pi^{-1}_\rho(\{x=x(a)\})$ with $a\in\sigma$ in a unique point with coordinates $(x_1,y_1,z_1)=(a,0,z_1(t_b))$. Since $z_1(t_b)\ne 0$, the observation above proves that $\Sigma=\Sat_V(T)\cap\{x=x(a)\}$ is an angle mark. This ends the proof.
 \end{proof}

The converse of Proposition~\ref{pro:trace-to-angle} is not necessarily true: from the proof just discussed, one can construct an angle mark $\Sigma$ whose saturation may accumulate to a set with non-empty interior inside $W^2_p$. However, if the angle mark $\Sigma$ has $D_k$-quasi-order different from $w_{D_k}(p)$, its saturation will accumulate to an edge adjacent to $p$ contained in $W^2_p$ and produce angle marks associated to that edge. In the following proposition we make this result precise.

 With the previous notation, take $p\in S'$ and $\sigma$ an edge intersecting $W^1_p$ with $\sigma\subset D_i\cap D_j$. Let $\tau_i,\tau_j$ be the two edges adjacent to $p$ and containing $D_i\cap W^2_p$, $D_j\cap W^2_p$, respectively. Finally, let $\alpha,\lambda_i,\lambda_j$ be the eigenvalues of a generator of $\LL$ at $p$ associated to the directions of $\sigma,\tau_i,\tau_j$, respectively.
\begin{proposition}\label{pro:angle-to-angle}
 Let $\Sigma$ be an angle mark associated to $\sigma$ with $D_k$-quasi-order equal to $\rho_k$ for $k=i,j$. Assume that $\rho_i\ne w_{D_i}(p)$ (thus also $\rho_j\ne w_{D_j}(p)$). Define $\epsilon\in\{i,j\}$ as the unique index satisfying $\rho_\epsilon>w_{D_\epsilon}(p)$. Then, in a sufficiently small neighborhood $V$ of $p$, we have
\begin{equation}\label{eq:sat-angle-mark}
\overline{\Sat_V(\Sigma)}\cap D\cap V=(\sigma\cup\{p\}\cup\tau_\epsilon)\cap V.
\end{equation}
 Moreover, if $p$ is a corner saddle point (so that $\tau_i,\tau_j$ are skeleton edges) then, for any planar section $\Delta$ at some point in $\tau_\epsilon\cap V$, we have that $\wt{\Sigma}=\Sat_V(\Sigma)\cap\Delta$ is an angle mark associated to $\tau_\epsilon$ with $D_\epsilon$-quasi-order equal to
\begin{equation}\label{eq:transition-explicit}
  \wt{\rho}_\epsilon=
    \frac{\lambda_{\epsilon'}-\lambda_{\epsilon}{\rho_\epsilon}}
   {\alpha},\;\;
  \mbox{ where }\{\epsilon,\epsilon'\}=\{i,j\}.
\end{equation}
Conversely, when $p$ is a corner saddle point, if $\wt{\Sigma}$ is an angle mark associated to $\tau_i$ (respectively to $\tau_j)$ with $D_i$-quasi-order (respectively $D_j$-quasi-order) equal to $\wt{\rho}$, then, for any planar section $\Delta$ at some point in $\sigma\cap V$, $\Sat_V(\wt{\Sigma})\cap\Delta$ is an angle mark with $D_i$-quasi-order equal to $\frac{\lambda_j}{\lambda_i}-\wt{\rho}\frac{\alpha}{\lambda_i}$ (respectively with $D_j$-quasi-order equal to $\frac{\lambda_i}{\lambda_j}-\wt{\rho}\frac{\alpha}{\lambda_j}$).
\end{proposition}

\begin{proof}
Assume that $\rho_i>\lambda_j/\lambda_i$ (i.e., $\epsilon=i$). Let $(V,(x,y,z))$ be a chart at $p$ with the same properties as in Proposition~\ref{pro:trace-to-angle}, and consider the $\rho_i$-weighted blow-up $\pi=\pi_{\rho_i}:M\to V$ with center $Y=\sigma\cap V$. Keeping the same notations as in the proof of Propostion \ref{pro:trace-to-angle}, we have that $Y'=\pi^{-1}(\tau_i)$ is the $y_1$-axis of the first chart $M_1$ of $\pi$. The transformed foliation $\wt{\LL}=\pi^*\LL$ has a local generator $\xi_1$ in $M_1$ given by the expression in (\ref{eq:xi1}), where $\rho$ is replaced with $\rho_i$. As we have already pointed out, this vector field is not necessarily of class $C^1$ at the origin of $M_1$. However, up to performing a ramification of the variable $y_1$, we obtain a vector field $\bar{\xi}_1$ of class $C^1$ on $M_1$ which is topologically equivalent to $\xi_1$ and has a similar expression as (\ref{eq:xi1}), although the eigenvalue $\lambda_i$ may be replaced with $\rho_i\lambda_i$.
In particular, taking into account that $\rho_i>w_{D_i}(p)$, the point $p_1$ is a hyperbolic saddle point of $\bar{\xi}_1$, for which the divisor $H=\pi^{-1}(Y)$ is the two-dimensional invariant manifold, and the $y_1$-axis is the one-dimensional invariant manifold.

Assume that the angle mark $\Sigma$ is contained in a transversal section at some $a\in\sigma$ of the form $\Delta_a=\{x=x(a)\}$. Since $\Sigma$ has $D_i$ quasi-order equal to $\rho_i$, we have that $\pi^{-1}(\Sigma)$ is contained in a compact subset of $M_1$ of the form $$K=\{x_1=x(a),\,k_1\le z_1\le k_2,\,0\le y_1\le\varepsilon\},$$
where $0<k_1<k_2$. The saturation of $K$ in $M$ by  $\wt{\LL}$ can be computed by using the vector field $\bar{\xi}_1$, since $K\subset M_1$ and $M_1$ is saturated for $\wt{\LL}$. Applying Hartman-Grobman's Theorem to $\bar{\xi}_1$ at $p_1$ as in Remark~\ref{rk:mark-from-w2-to-w1}, we conclude that $\overline{\Sat_{M}(K)}\cap\pi^{-1}(D)$ is contained in $H\cup Y'$ and contains $Y'$ (see Figure \ref{Fig:TransitionAngleMark}). This proves the required property (\ref{eq:sat-angle-mark}).

\begin{figure}[h]
	\begin{center}
		\includegraphics[scale=0.78]{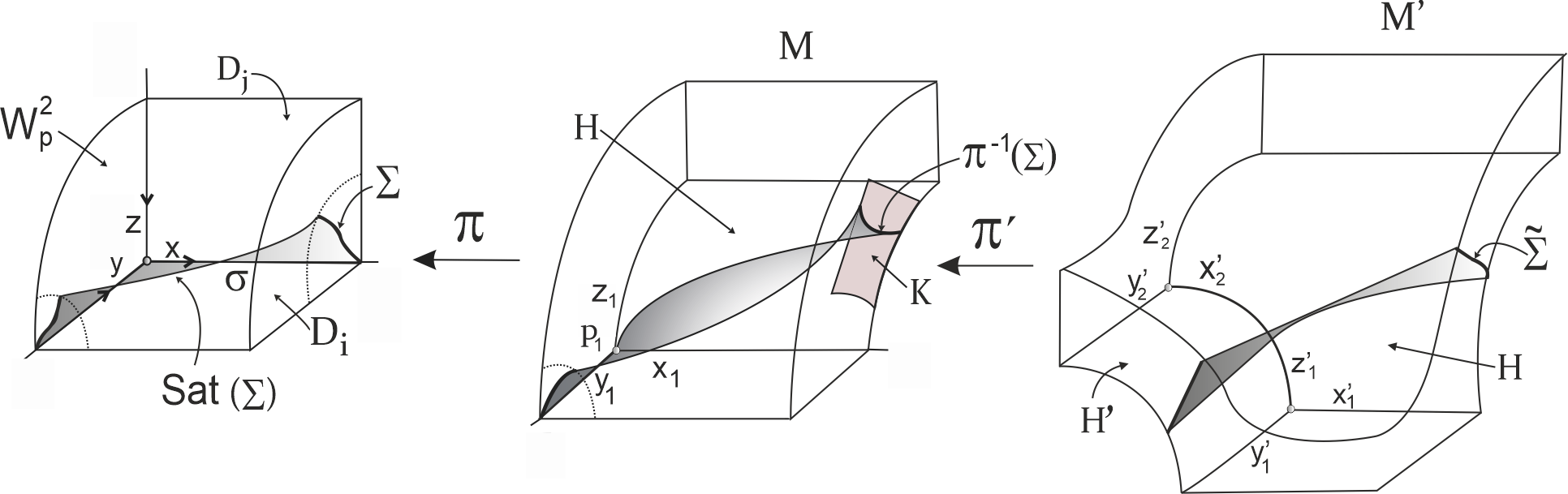}
	\end{center}
	\caption{Transition of an angle mark.}
	\label{Fig:TransitionAngleMark}
\end{figure}

In order to prove (\ref{eq:transition-explicit}), we notice first that the eigenvalues of the linear part of $\bar{\xi}_1|_H$ at $p_1$ are $\alpha,\lambda_j-\rho_i\lambda_i$ (both of the same sign), associated to the directions of the $x_1$ and $z_1$-axis, respectively. Thus, any trajectory of $\bar{\xi}_1|_H$ in a neighborhood of $p_1$, except for those two axis, has quasi-order equal to $\wt{\rho}_i$ with respect to the $x_1$-axis (identifying $H\cap M_1=\{y_1=0,x_1\ge 0,z_1\ge 0\}$ with the quadrant $\R^2_{\ge 0}$ by means of the coordinates $(x_1,z_1)$). In particular, $\Sat_M(K)\cap H$ is enclosed by two such trajectories.
 We perform a new $\wt{\rho}_i$-weighted blow-up $\pi'=\pi_{\wt{\rho}_i}:M'\to M_1$ with center $Y'=\{x_1=z_1=0\}$, written in the two charts $(M'_1,(x'_1,y'_1,z'_1))$, $(M'_2,(x'_2,y'_2,z'_2))$ defining $M'$ as
$$
x_1=x'_1,\,y_1=y'_1,\,z_1=z'_1(x'_{1})^{\wt{\rho}_i};\;\;\;\;
x_1=x'_2(z'_{2})^{1/\wt{\rho}_i},\,y_1=y'_2,\,z_1=z'_2.
$$
A computation as in the proof of Proposition~\ref{pro:trace-to-angle} shows that $\xi'_1=\pi'^{\ast}\bar{\xi}_1$ is a vector field of class $C^1$ on $M'_1\cap M'_2$ for which any point $q\in\pi'^{-1}(p_1)\cap M'_1\cap M'_2$ is a singular point with eigenvalues $\alpha,\lambda_i,0$ (independent of the point $r$). At any such point $q$, the fiber $\pi'^{-1}(p_1)$ is the center manifold of $\xi'_1$ (unique in this case), whereas the center-unstable and center-stable manifolds satisfy $\{W^{cu}(q),W^{cs}(q)\}=\{\pi'^{-1}(H),H'\}$, where $H'=\pi'^{-1}(Y')$ is the divisor of $\pi'$. We obtain that trajectories of $\bar{\xi}_1|_H$ outside $\{x_1z_1=0\}$ are lifted by $\pi'$ to curves accumulating to a single point of $\pi'^{-1}(p_1)\cap M'_1\cap M'_2$. Also, the leaves of $\xi'_1|_{H'}$ are the lines $\{x'_1=0,z'_1=c\}_{c\ge 0}$. Using these observations and the same arguments as in the proof of Proposition~\ref{pro:trace-to-angle}, we conclude that the saturation $\wt{\Sigma}$ of $(\pi\circ\pi')^{-1}(\Sigma)$ by $\xi'_1$ is contained in a region of the form $\{k'_1\le z'_1\le k'_2\}$ for some constants $0<k'_1<k'_2$. Thus, for $u>0$ sufficiently small, we have $\pi\circ\pi'(\wt{\Sigma}\cap\{y'_1=u\})$ is an angle mark associated to $\tau_i$ inside the transversal section $\Delta_u=\{y=u\}$ with quasi-order equal to $\wt{\rho}_i$ with respect to $D_i=\{z=0\}$. This proves the second assertion of the proposition.
The rest of the statement concerning the transition of angle marks associated to $\tau_i$ or $\tau_j$ to an angle mark associated to $\sigma$ can be proved analogously.
\end{proof}

With the notations of Proposition~\ref{pro:angle-to-angle}, when $p$ is a corner saddle point, the map $\rho_\epsilon\mapsto\wt{\rho}_\epsilon$ defined by (\ref{eq:transition-explicit}) for $\epsilon\in\{i,j\}$ will be called the {\em transition from $\sigma$ to $\tau_\epsilon$} and denoted by $T_{\sigma,\tau_\epsilon}$, whereas its inverse will be called the {\em  transition from $\tau_\epsilon$ to $\sigma$} and denoted by $T_{\tau_\epsilon,\sigma}$. Notice that $T_{\sigma,\tau_\epsilon}$ defines an affine bijection from the interval $(w_{D_\epsilon}(p),+\infty)$ into the interval $(0,+\infty)$.

\subsection{Saddle-resonance} 
Now we can introduce the main definitions of this section.
\begin{definition}{(Resonant s-connections)}\label{def:infinitesimal-connections}
Let $p,q\in S'$ with $p\ne q$ and let $\g=(\sigma_0,\sigma_1,...,\sigma_{n})$ be a multiple saddle connection from $p$ to $q$. We say that $\g$ is {\em saddle-resonant} (or {\em s-resonant} for short) if the following conditions are satisfied:
\begin{itemize}
  \item Each edge $\sigma_i$ is a skeleton edge.
  \item The first edge $\sigma_0$ (respectively the last edge $\sigma_{n}$) intersects $W^1_p$ (respectively $W^1_q$).
  \item For $i=0,\dots,n-1$, let $D_i$ be the component of $D$ which contains $\sigma_i$ and $\sigma_{i+1}$. Take $\rho_0=w_{D_0}(p)$. Then $\rho_0$ belongs to the domain of $T_{\sigma_0,\sigma_1}$. By defining recursively, for $i=1,\dots,n-1$, the value
      $$\rho_i=
      \left\{
        \begin{array}{ll}
          T_{\sigma_{i-1},\sigma_{i}}(\rho_{i-1}), & \hbox{if }D_{i-1}=D_{i} \\
          \\
          {1}/{T_{\sigma_{i-1},\sigma_{i}}(\rho_{i-1})}, & \hbox{if }D_{i-1}\ne D_{i},
        \end{array}
      \right.
      $$
      we have that $\rho_{i}$ belongs to the domain of $T_{\sigma_{i},\sigma_{i+1}}$ when $i<n$.
      \item It holds $T_{\sigma_{n-1},\sigma_n}(\rho_{n-1})=w_{D_{n-1}}(q)$.
\end{itemize}
The HAFVSD $\mathcal{M}$ will be called {\em s-resonant} if it has any s-resonant multiple saddle connection.
\end{definition}
\subsection{Saturation of trace marks}
Now we state the main result in this section. Roughly speaking, under the conditions of Morse-Smale and not s-resonance, the saturation of a given trace mark near a point $p\in S$ will be a surface accumulating along the support of a path of edges that does not depend on the trace mark. As a consequence, the saturation of the two-dimensional invariant manifold at any $p\in S_{tr}$ shares an analogous property. 

In the following statement, we fix realizations of the invariant manifolds $W^2_p$ for any $p\in S_{tr}$ and denote 
$$
\wt{D}=D\cup\bigcup_{p\in S_{tr}}W^2_p.
$$
Notice that, according to the notation already introduced before Definition~\ref{def:trace-mark}, the germ of $\wt{D}$ at any $p\in S_{tr}$ is equal to $\wt{D}_p$.

\begin{theorem}\label{th:path of a trace mark}
Let $\mathcal{M}=(M,D,\LL)$ be a HAFVSD and assume that it is not s-resonant and of Morse-Smale type. Consider a $D$-saddle point $p\in S$ and a local s-component $\nu$ at $p$. Then there exists a unique path of edges $\Theta(\nu)$ with extremities at $p$ and another point $q=q(\nu)$, which is a $D$-node, satisfying the following property:
for any sufficiently small neighborhood $U$ of $D$ in $M$, and for any trace mark $T$ on $\nu$ contained in $U$, we have
\begin{equation}\label{eq:Sat-T}
\overline{\Sat^\epsilon_U(T)}\cap\wt{D}=\ell^\epsilon_b\cup|\Theta(\nu)|,
\end{equation}
where $\epsilon=+$ (respectively $\epsilon=-$) in case $W^1_p$ is the unstable (respectively stable) manifold at $p$ and $b\in W^2_p$ is the point where $T$ is attached.
\end{theorem}
\begin{proof}
Assume for instance that $W^1_p=W^{u}(p)$. We proceed by induction on the length $l(p)$ (cf. Section~\ref{sec:SHNR-foliations}). Let $\sigma_1=[p,p_1]$ be the edge starting at $p$ which contains $W^1_p\cap\nu$. Take $T$ a trace mark on $\nu$ attached to a point $b\in W^2_p\setminus Sk(\widetilde{D}_p)$ sufficiently near to $p$. Using Remark~\ref{rk:mark-from-w2-to-w1}, in a small open neighborhood $V_0$ of $p$ we have 
\begin{equation}\label{eq:Sat-T-0}
\overline{\Sat^+_{V_0}(T)}\cap\wt{D}\cap V_0=(\ell^+_{b}\cup\{p\}\cup\sigma_1)\cap V_0
\end{equation}
Put $T_1=\Sat^+(T)\cap\Delta_1$, where $\Delta_1$ is a transversal section at some point $a_1\in\sigma_1\cap V_0$. From (\ref{eq:Sat-T-0}), if $U$ is a sufficiently small neighborhood of $D$, we obtain that
\begin{equation}\label{eq:saturation of marks}
\overline{\Sat^+_U(T)}\cap\wt{D}=\ell^+_b\cup\{p\}\cup\sigma_1\cup\left(\overline{\Sat^+_{U}(T_1)}\cap\wt{D}\right).
\end{equation}
Now, we have several possibilities:

(a-1) The point $p_1$ is a $D$-node. Hence, necessarily $p_1\in N^a$. We put $\Theta(\nu)=(\sigma_1)$. As in Remark~\ref{rk:mark-from-w2-to-w1}, the result follows from equation (\ref{eq:saturation of marks}) by applying Hartman-Grobman's Theorem at the point $p_1$. This situation occurs when $l(p)=1$, so we can start our induction argument with this case (a-1).

(b-1) The point $p_1$ is a $D$-saddle point and the edge $\sigma_1$ is a trace edge. In this case, by the Morse-Smale condition, the component of $D$ which contains $\sigma_1$ coincides with the two-dimensional invariant manifold $W^2_{p_1}$ locally at $p_1$, that is, $p_1\in S_{tg}$. Thus, by using the flow, we can see the curve $T_1$ as a trace mark on the (unique) local s-component at $p_1$, denoted by $\nu_1$. Using the induction hypothesis, we consider the path of edges $\Theta(\nu_1)$ starting at $p_1$ and satisfying (\ref{eq:Sat-T}) for $\nu_1$ and $T_1$. By means of (\ref{eq:saturation of marks}), the path $\Theta(\nu)=(\sigma_1,\Theta(\nu_1))$ satisfies the requirements of the theorem.

(c-1) The point $p_1$ is a $D$-saddle point and the edge $\sigma_1$ is a skeleton edge. Let $D_1,D_2$ be the two components of $D$ such that $\sigma_1\subset D_1\cap D_2$. By Proposition~\ref{pro:trace-to-angle}, the curve $T_1$ is an angle mark associated to $\sigma_1$ with $D_i$-quasi-order equal to $w_{D_i}(p)$, for $i=1,2$. It turns out that there exists a unique edge $\sigma_2=[p_1,p_2]$ such that, in a small enough neighborhood $V_1$ of $p_1$, and assuming that $a_1$ is sufficiently close to $p_1$, we have
\begin{equation}\label{eq:sigma1-sigma2}
\overline{\text{Sat}_{V_1}(T_1)}\cap\wt{D}\cap V_1=\overline{\text{Sat}_{V_1}(T_1)}\cap D\cap V_1=(\sigma_1\cup\{p_1\}\cup\sigma_2)\cap V_1.
\end{equation}
Indeed, if $\sigma_1$ is contained in $W^2_{p_1}$ locally at $p_1$ (that is, either $W^2_{p_1}\subset D_1$ or $W^2_{p_1}\subset D_2$) then (\ref{eq:sigma1-sigma2}) holds with $\sigma_2$ the edge intersecting $W^1_{p_1}$, by Remark~\ref{rk:mark-from-w2-to-w1}. On the other hand, if $\sigma_1$ intersects $W^1_{p_1}$ then, by the non-resonance hypothesis, we have that $w_{D_i}(p)\ne w_{D_i}(p_1)$, for $i=1,2$, and equation (\ref{eq:sigma1-sigma2}) holds as a consequence of  Proposition~\ref{pro:angle-to-angle}. 
Put $T_2=\Sat^+(T_1)\cap\Delta_2$, where $\Delta_2$ is a transversal section at some point $a_2\in\sigma_2$.
Using equations (\ref{eq:saturation of marks}) and (\ref{eq:sigma1-sigma2}), we obtain that if $U$ is a small enough neighborhood of $D$ then
$$
\overline{\Sat^+_U(T)}\cap\wt{D}=\ell^+_b\cup\{p\}\cup\sigma_1\cup\sigma_2\cup\left(\overline{\Sat^+_{U}(T_2)}\cap\wt{D}\right).
$$
Take $\sigma_2=[p_1,p_2]$. 
Now, we have the same three possibilities for the extremity $p_2$ as those we had for $p_1$, namely:

(a-2) The point $p_2\in N^a$. We finish by taking $\Theta=(\sigma_1,\sigma_2)$.

(b-2) The point $p_2$ belongs to $S_{tg}$ and $\sigma_2$ is a trace edge (thus intersecting $W^2_{p_2}\subset D$). Taking $a_2$ sufficiently near to $p_2$, we have that $T_2$ is a trace mark at $p_2$. By induction hypothesis on $l(p_2)<l(p)$, we finish by taking $\Theta=(\sigma_1,\sigma_2,\Theta(\nu_2))$ where $\nu_2$ is the local s-component at $p_2$ containing $T_2$.

(c-2) The point $p_2$ belongs to $S$ and  $\sigma_2$ is a skeleton edge. Applying Proposition~\ref{pro:angle-to-angle} to $p_1$ and $T_1$, we have that $T_2$ is an angle mark associated to $\sigma_2$. 

Since the graph $\Omega$ is finite and has no cycles, by repeating the arguments, we obtain a path $(\sigma_1,\dots,\sigma_r)$ with the following properties:
\begin{itemize}
\item The path $(\sigma_1,\dots,\sigma_{r-1})$ is a multiple saddle connection.
\item For $j\in\{2,\dots,r\}$, define 
$T_j=\Sat^+(T_{j-1})\cap\Delta_j$, where $\Delta_j$ is a transversal section at some point $a_j\in\sigma_j$. Then $T_j$ is an angle mark for $j<r$ and, if $U$ is a sufficiently small neighborhood of $D$, we have
$$
\overline{\Sat^+_U(T)}\cap\wt{D}=\ell^+_b\cup\{p\}\cup\sigma_1\cup\sigma_2\cup\cdots\cup\sigma_r\cup
\left(\overline{\Sat^+_{U}(T_r)}\cap\wt{D}\right).
$$
\item The extremity $p_r=\omega(\sigma_r)$ is either in the situation (a-$r$), i.e. $p_r\in N^a$, or in the situation (b-$r$), i.e., $p_r\in S_{tg}$ and $\sigma_r$ is a trace edge intersecting $W^2_{p_r}\subset D$.
\end{itemize}
If $p_r\in N^a$, we put $\Theta(\nu)=(\sigma_1,...,\sigma_r)$ and the result follows. In the situation (b-r), we may assume that $T_r$ is a trace mark on $\nu_r$, the (unique) local s-component at $p_r$. By induction hypothesis, we get a path $\Theta(\nu_r)$ satisfying (\ref{eq:Sat-T}) for $\nu_r$ and $T_r$. Finally, the path $\Theta(\nu)=(\sigma_1,\dots,\sigma_r,\Theta(\nu_r))$ satisfies the required property for $\nu$ and $T$.
This ends the proof.
\end{proof}
\begin{remark}\label{rk:path-theta}
{\em
 From the construction of the path $\Theta(\nu)$ in Theorem~\ref{th:path of a trace mark}, it follows that there is no trace edge in $\Theta(\nu)$ which ends (resp. starts) at a transversal saddle point if $W^1_p$ if the unstable manifold (resp. the stable manifold) at $p$.
}
\end{remark}

\begin{corollary}\label{cor:sat-w2}
With the same conditions as above, given a transversal saddle point $p\in S_{tr}$, there are two paths of edges $\Pi_p^1,\Pi_p^2$ starting (resp. ending) at $p$ and ending (resp. starting) at some $D$-node point when $W^2_p$ is unstable (resp. stable), such that for any sufficiently small neighborhood $U$ of $D$ in $M$ we have
	\begin{equation}\label{eq:paths-w2p}
	\overline{\Sat_U(W^2_p)}\cap\wt{D}=\overline{W^2_p\cap U}\cup|\Pi_p^1| \cup |\Pi_p^2|.
\end{equation}%
As a consequence, two different elements in the family $\{\Sat_U(W^2_p)\}_{p\in S_{tr}}$ are always mutually disjoint, for $U$ small enough.
\end{corollary}
\begin{proof}
	Assume, for instance, that $W^2_p$ is the unstable manifold. In light of Lemma~\ref{lm:graph}, there are exactly two edges starting at $p$, say $\sigma_i=[p,p_i]$, $i=1,2$. Moreover, they are trace edges and satisfy $W^2_p\cap D\subset\sigma_1\cup\sigma_2\cup\{p\}$. By the Morse-Smale condition, the extremity $p_i$, for $i=1,2$, is either a $D$-node or a tangential saddle point,  and hence, there is a unique local s-component $\nu_i$ at $p_i$. Consider a neighborhood $U_i$ of $\overline{\sigma_i}$ and a transversal section $\Delta_i$ at some point in $\sigma_i\cap\nu_i$, both sufficiently small to guarantee that  $T_i=\Sat^+_{U_i}(W^2_p)\cap\Delta_i$ is a trace mark on $\nu_i$. 
	Put $\Pi_p^i=(\sigma_i, \Theta(\nu_i))$, $i=1,2$. Property (\ref{eq:paths-w2p}) follows by applying  Theorem~\ref{th:path of a trace mark} to the local s-components $\nu_1,\nu_2$ and the trace marks $T_1, T_2$, respectively. Now, notice that if $p'\in S_{tr}$ and $p'\ne p$, then (\ref{eq:paths-w2p}) implies that $\Sat_U(W^2_p)\cap W^2_{p'}\subset D$. On the other hand, we have $\Sat_U(W^2_p)\cap D=\{p\}\cup\sigma_1\cup\sigma_2$, where $\sigma_1,\sigma_2$ are as above. These two observations prove the last claim in the corollary, taking into account that none of those two edges cuts $W^2_{p'}$, by virtue of the Morse-Smale condition. 
\end{proof}

Note that the paths in Corollary~\ref{cor:sat-w2} associated to different transversal saddle points may share common edges. This means that, although the saturations of the two-dimensional invariant manifolds at different transversal saddle do not intersect, their closures in $M$ could.
\begin{remark}\label{rk:sat-empty-intersection}
	{\em 
The disjointness property of the family of sets in the last part of Corollary~\ref{cor:sat-w2} can be extended to saturations of trace marks. More in precise: asume that for each $\nu\in\VV(\Omega)$ we take a (possibly empty) family $\{T^1_\nu,\dots,T^{s_\nu}_\nu\}$ of mutually disjoint trace marks in $\nu$, attached to points not in $|\Omega|$.
Then, for a sufficiently small neighborhood $U$ of $D$ in $M$, the elements of the family
$$
\FF=\{\Sat_U(W^2_p)\}_{p\in S_{tr}}\cup\{\Sat_U(T^j_\nu)\}_{\nu\in\VV(\Omega),\,j=1,...,s_\nu}
$$
are pairwise disjoint. To see this, notice that if $\nu$ is associated to a transversal saddle, then equation (\ref{eq:Sat-T}) is true for any $T^j_\nu$ (deleting the exponent $\epsilon\in\{+,-\}$ that indicates the sense of the flow). On the contrary, if $\nu$ is not associated to a transversal saddle and $T^j_\nu$ is attached to $a^j$, then $a^j\in D$ for $j \in \{1,\dots ,s_\nu\}$. By pushing $T^j_\nu$ by the flow, we get two trace marks $T^j_\nu(+)$, $T^j_\nu(-)$ on the local s-components $\tilde{\omega}(\ell_{a^j})$ and $\tilde{\alpha}(\ell_{a^j})$, respectively, and we have
$$
\overline{\Sat_U(T^j_\nu)}\cap\wt{D}=
\left(\overline{\Sat^+_U(T^j_\nu(+))}\cap\wt{D}\right)\cup
\ell_{a^j}\cup
\left(\overline{\Sat^-_U(T^j_\nu(-))}\cap\wt{D}\right).
$$
The fact that two different elements of $\FF$ do not intersect follows, as in Corollary~\ref{cor:sat-w2}, from these observations along with the corresponding properties (\ref{eq:Sat-T}) and (\ref{eq:paths-w2p}) applied to those elements. 
}
\end{remark}

\section{Distinguished neighborhoods from chimneys}\label{sec:distinguished}

Let $\mathcal{M}=(M,\LL,D)$ be a non s-resonant of Morse-Smale type HAFVSD.
In this section we construct a convenient base of neighborhoods of the support $|\Omega|$ by using the so-called {\em distinguished fattenings} that we define below. These fattenings will constitute the first pieces in order to get the fitting domains announced in Theorem~\ref{th:main}. Each one of them will be constructed by considering small local neighborhoods at points of $\Sing(\LL)$ of {\em chimney-shape} connected along the edges of the graph in such a way we control perfectly the transversal frontier of the resulting neighborhood. The constructions in this section are inspired in those ones presented by Alonso-Gonz\'{a}lez, Cano and Camacho in \cite{Alo-C-C1,Alo-C-C2}.

For ease of its reading, since we introduce a considerable amount of notation and intermediate technical results, we organize this section in several subsections.

{\em Notation.-} In what follows, a subset $I\subset M$ will be called an {\em (open or closed) interval} if $I$ is homeomorphic to an (open or closed) interval of $\R$ and its closure $\overline{I}$ is homeomorphic to a compact interval $[a,b]\subset\R$. In this case, the points $a',b'$ of $\overline{I}$ corresponding to $a,b$ by an homeomorphism are called the {\em extremities} of $I$, while we use the notation $\dot{I}=I\setminus\{a',b'\}$. An interval $I$ will be called {\em non-trivial} if it is non-empty and not reduced to a single point.
On the other hand, a subset of $M$ homeomorphic to a (open or closed) two-dimensional disc will be also called a {\em (open or closed) disc}. In this situation, expressions of the type ``boundary'' of a disc $T\subset M$, etc., refer to the subset of $T$ which corresponds to the boundary of the genuine disc by a homeomorphism.

\subsection{Chimney-shape neighborhoods at singular points}
Let $p\in\Sing(\LL)$ be a singular point and assume that $p$ is a three-dimensional saddle.  Let $(V,\xx=(x,y,z))$ be an analytic
chart of $M$ at $p$ such that the local invariant manifolds at $p$ are analytic submanifolds of $V$ given by 
$$
W^2_p\cap V=\{z=0\},\;\; W^1_p\cap V=\{x=y=0\}.
$$
The coordinate $z$ can take values either on $\R_+$ (when $p\in N\cup S_{tg}$) or on $\R$ (when $p\in S_{tr}$). For convenience, we use the notation $\epsilon=+$ in the first case and $\epsilon\in\{+,-\}$ in the second one. According to the definition in Section \ref{sec:SHNR-foliations}, the local s-components at $p$ are the germs at $p$ of the open sets $\{\epsilon z>0\}$, denoted by $\nu_p^\epsilon$, where $\epsilon$ is chosen with that convention.
For any such $\epsilon$ and any $\delta>0$, we
consider the closed cylinder
$B=B^\epsilon_{\delta}=\{x^2+y^2\leq\delta,0\leq\epsilon z\leq\delta\}$ and distinguish the followings three components on its boundary:
$$
\begin{array}{l}
b(B)=\{x^2+y^2\leq\delta,z=0\},\\
t(B)=\{x^2+y^2\leq\delta,z=\epsilon\delta\},\\
w(B)=\{x^2+y^2=\delta,0\leq\epsilon z\leq\delta\},
\end{array}
$$
called, respectively, the {\em base}, the {\em top} and the {\em wall} of $B$. Notice that $b(B)=B\cap W^2_p$ and that $t(B)$ cuts transversally $W^1_p$ at a unique point. The base and the top are closed discs, whereas the wall is either a disc or homeomorphic to the cylinder $\SSS^1\times[0,1]$, the latter case occurring if and only if  $p\in N$. When $w(B)$ is a disc, its boundary is the union of four closed intervals: $w(B)\cap t(B)$, $w(B)\cap b(B)$ and the two connected components of $D^*_p\cap w(B)$, where $D^*_p$ is the union of components of $D$ containing $W^1_p$.

We will always assume that $\delta$ is sufficiently small so that $p$ is the only singular point of $\LL$ in $B$ and $\LL$ is transversal to the top and to the wall. By means of the Hartman-Grobman Theorem, this implies that, for any $a\in w(B)\setminus b(B)$, the leaf $\ell_a$ of $\LL|_{B}$ through $a$ cuts $t(B) \setminus W_p^1 $ at a single point, thus establishing a homeomorphism $w(B)\setminus b(B)\simeq t(B)\setminus W^1_p$. Moreover, if $a \in b(B) \setminus\{p\}$, then $\ell_a\cap B$ is a $B$-leaf, completely contained in $W^2_p$, that cuts once $w(B)$ and accumulates to $p$.

\begin{definition}\label{def:fence}
Given $B=B^\epsilon_\delta$ as above, a {\em fence} in $B$ is a semi-analytic subset $F\subset w(B)$ which is homeomorphic to the wall $w(B)$ and satisfies (see Figure \ref{Fig:Chimney})
\begin{enumerate}
  \item $F\cap t(B)=\emptyset$.
  \item The set $b(F)=F\cap W^2_p$, called the  {\em base} of $F$, is equal to $b(B)\cap w(B)$.
  \item If $w(B)$ is a disc and $\tilde{J}$ is a connected component of $D^*_p\cap w(B)$, then $J=\tilde{J}\cap F$ is a non trivial closed interval. Any such $J$ is called a {\em doorjamb} of the fence.
\end{enumerate}
\end{definition}

Recall that either $F$ is homeomorphic to $\SSS^1\times[0,1]$ or $F$ is a closed disc. In the first case $F$ has no doorjambs, while in the second case it has two doorjambs, denoted by $J_1,J_2$. In both cases, 
the fence $F$ is a topological surface with boundary, and its boundary $\partial F$ contains the base $b(F)$ and the doorjambs of $F$ (if they exist). The {\em handrail} of the fence $F$ is defined as
$
h(F):=\overline{\partial F\setminus(b(F)\cup J_1\cup J_2)},
$
(and hence $\partial F=b(F)\cup J_1\cup J_2\cup h(F)$, where $J_1=J_2=\emptyset$ if $F$ has no doorjambs). The handrail of $F$ does not intersect the base $b(F)$. It is homeomorphic to $\SSS^1$ if $F\simeq\SSS^1\times[0,1]$ or a closed interval if $F$ is a disc. Moreover, in this latter case, the handrail $h(F)$ cuts each doorjamb just in a single point, which is a common extremity of both.

\begin{definition}\label{def:chimney}
Let $\nu$ be a local s-component at $p$. A {\em chimney neighborhood} (or a {\em c-nbhd} for short) in $\nu$, is a set of the form
$
\CC=\CC_{B,F}=\overline{\Sat_{B}(F)}\cap B
$, where $F$ is a fence in a cylinder $B=B^\epsilon_{\delta}$ such that $\nu\subset B$ (see Figure \ref{Fig:Chimney}). We say that $F$ is the {\em fence} of $\CC$, and denote it by $F_\CC$, if we want to emphasize that it is associated to $\CC$. A {\em refinement} of $\CC=\CC_{B,F}$ is another c-nbhd of the form $\CC'=\CC_{B,F'}$, where $F'$ is a fence in $B$ contained in $F_\CC$. We will use the notation $\CC'<\CC$. (Notice that the cylinder frame $B$ is the same for $\CC$ and for any of its refinements).
\end{definition}

\begin{figure}[h]
	\begin{center}
		\includegraphics[scale=0.60]{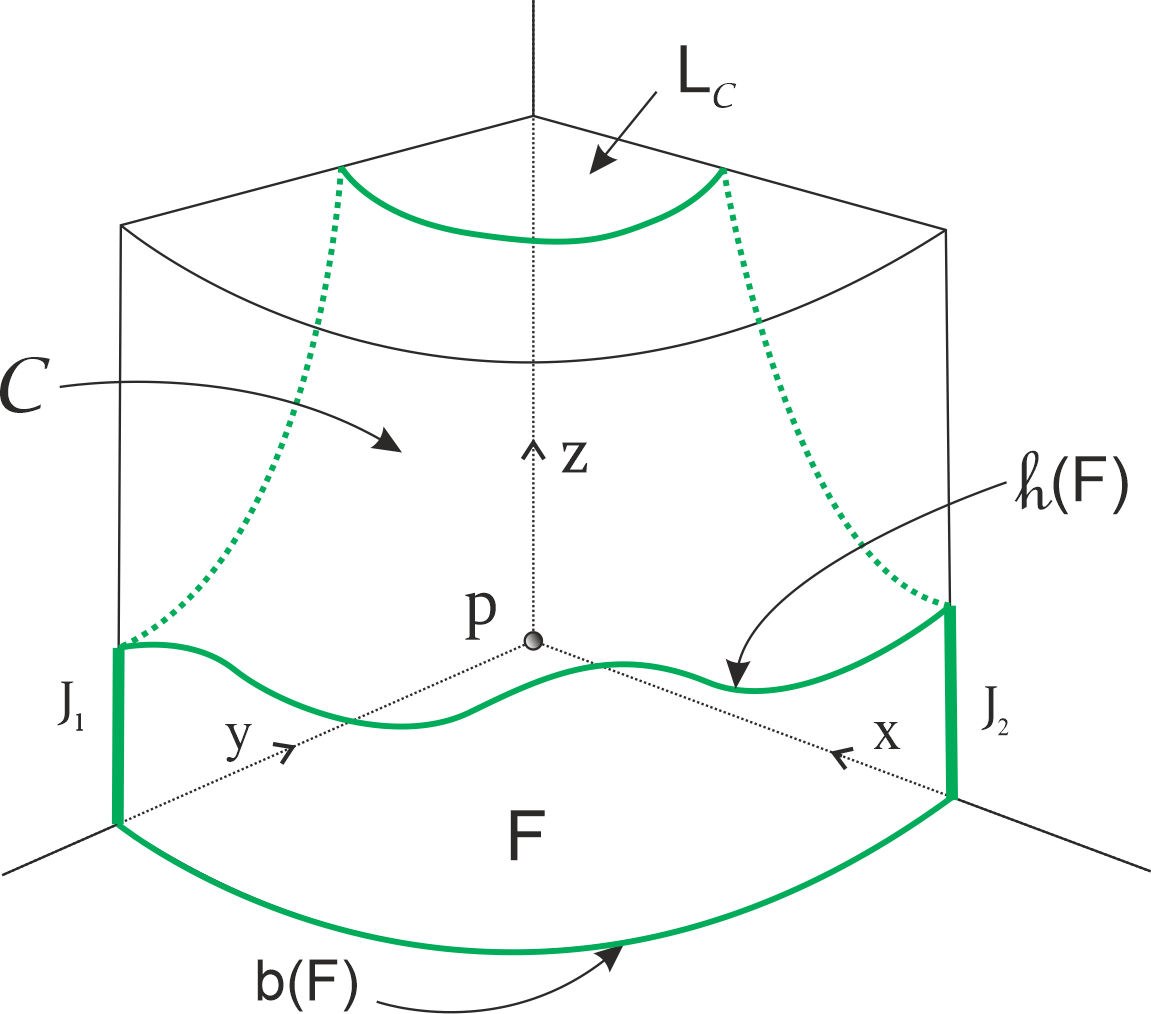}
	\end{center}
	\caption{Chimney with elements.}
	\label{Fig:Chimney}
\end{figure} 

A c-nbhd $\CC$ is a compact three dimensional topological manifold with boundary and a semi-analytic set of $M$. Notice that its boundary $\partial\CC$ contains the topological frontier $Fr(\CC)$ inside $M$, but that the equality does not hold (unless $p\in S_{tr}$) since we are working in $M$, having itself a boundary. More precisely, $\partial\CC\setminus Fr(\CC)$ is equal to the interior inside $\partial\CC$ of $\partial\CC\cap D$, a set which is locally invariant by $\LL$. Therefore, there is no ambiguity in adopting, from now on, the notation $\partial\CC^\pitchfork:=Fr(\CC)^\pitchfork$ for the transversal frontier of $\CC$ (according to the definitions introduced in Section~\ref{sec:SHNR-foliations}). Observe that $\partial\CC^\pitchfork$ consists of two connected components:
 the fence $F=F_\CC$ of $\CC$ and the set $L=L_\CC:=\CC \cap t(B)$, called the {\em lid of $\CC$}. The lid is a closed disc, independently of the topology of the fence. More precisely, we have $\{\partial\CC^{in},\partial\CC^{out}\}=\{F_\CC,L_\CC\}$, where $\partial\CC^{in}=F_\CC$ (resp. $\partial\CC^{in}=L_\CC$) if, and only if, $W^2_p$ is the stable (resp. unstable) manifold.  Considered as manifolds with boundary, points in the interior of $\partial\CC^{in}$ (resp. $\partial\CC^{out}$) are of type e-i (resp. i-e),  relatively to $\CC$, while points in the boundary of $\partial\CC^{in}$ (resp. $\partial\CC^{out}$) are of type e-t (resp. t-e).

Notice that, if $p$ is a tangential saddle or a $D$-node point, and $\nu_p$ is the unique local s-component at $p$, then a c-nbhd in $\nu_p$ is actually a neighborhood of $p$ in $M$. On the contrary, if $p$ is a transversal saddle point, the union of the respective chimney neighborhoods $\CC^{+},\CC^{-}$ in the two local $s$-components $\nu_p^{+},\nu_p^{-}$ at $p$, is a neighborhood of $p$. 

\strut

So far, we have constructed chimney neighborhoods for three-dimensional saddles.
In order to unify notation, if $p\in\Sing(\LL)$ is not a three-dimensional saddle point and $\nu=\nu_p$ is the unique local s-component at $p$, a {\em
chimney neighborhood} (or a c-nbhd for short) in $\nu$ is a pair $(\CC,F)$ of sets where:

- $\CC=\{x^2+y^2+z^2\leq\delta\}\subset\overline{\nu}$ in some analytic
coordinates $\xx=(x,y,z)$ at $p$ such that ${\partial\CC\setminus D}$ is everywhere
transversal to $\LL$;  

- $F$ is a subset of $\partial\CC^\pitchfork(:=Fr(\CC)^\pitchfork=\overline{\partial\CC\setminus D})$ which is the image of a one-to-one continuous map $\vp:(\partial\CC\cap D)\times[0,1]\to\partial\CC^\pitchfork$ such that $\vp(t,0)=t$ for any $t\in\partial\CC\cap D$. 

For the sake of simplicty, we just say that $\CC$ is a c-nbhd and that $F=F_\CC$ is the {\em fence} of $\CC$. The set $b(F):=\partial \CC \cap D$ is called the {\em base} of the fence $F$ and the set $h(F):=\vp((\partial\CC\cap D)\times\{1\})$   is called the {\em handrail} of the fence $F$; both are semi-analytic curves homeomorphic to $\SSS^1$. As in the case of a three-dimensional saddle $D$-node, there are no doorjambs of $F$. Finally, the set $L=\overline{\partial\CC^\pitchfork\setminus F}$ is called the {\em lid} of the c-nbhd $\CC$. It is a semi-analytic closed disc, satisfying $\partial L=h(F)$.

\subsection{Doors and pre-doors of a chimney neighborhood}

\begin{definition}\label{def:pre-door}
Let $\CC$ be a c-nbhd in some $\nu\in\VV(\Omega)$ with fence $F=F_{\CC}$ and lid $L=L_{\CC}$. A {\em pre-door of}
$\CC$ is a semi-analytic disc $\DD\subset\partial\CC^\pitchfork$, either contained in $F$ or in $L$, such that:
\begin{enumerate}[(i)]
  \item $\DD$ contains at most one point of the set $\partial\CC\cap|\Omega|$. If such a point exists, it is called the {\em center} of $\DD$.
  \item If $\DD\subset L$, then $\overline{\Sat_{\CC}(\DD)}$ is a refinement of $\CC$. The set $b(\DD):=D \cap \DD$ is called the {\em base} of $\DD$ in this case.
  \item If $\DD\subset F$, the set $b(\DD):=\DD\cap b(F)$, called the {\em base} of $\DD$, is a non-trivial closed interval (necessarily contained in $\partial\DD$).
  \item If $\DD\subset F$ and $F$ has doorjambs $J_1,J_2$, then at most one of the two sets $J_1\cap\DD, J_2\cap\DD$ is non empty and, if $J_i\cap\DD\ne\emptyset$ for $i\in\{1,2\}$, then $J_i\cap\DD$ is a non-trivial closed interval with one extremity in $b(\DD)$. If $J_i\cap\DD\ne\emptyset$ then $J_i\cap\DD$ is called the {\em fixed doorjamb} of $\DD$.
  \item If $a$ is the center of $\DD$, then $a$ is an interior point of $b(\DD)\cup J$, where $J$ is the fixed doorjamb of $\DD$ or $J=\emptyset$ if $\DD$ has no fixed doorjamb (for instance, when $\DD \subset L$).
\end{enumerate}
\end{definition}

Given a pre-door $\DD$ in $\CC$, a {\em framing} of $\DD$ is the choice of a non-trivial closed interval $h(\DD)\subset\partial\DD$, called a {\em handrail} of $\DD$, which satisfies the following rules (see Figure \ref{Fig:Predoor_door}):
\begin{itemize}
\item The interior $\dot{h(\DD)}$ of $h(\DD)$ does not intersect $D\cup W^2_p$.
\item If $\DD$ is contained in the lid of $\CC$ then $h(\DD)=\overline{\partial\DD\setminus b(\DD)}$.
\item If $\DD$ is contained in the fence $F_\CC$ then $h(\DD)$ contains $\DD\cap h(F_\CC)$ and $\partial\DD\setminus(\dot{b(\DD)}\cup\dot{h(\DD)})=J\cup J'$, where $J,J'$ are two nontrivial closed disjoint intervals such that one of them is the fixed doorjamb of $\DD$ if $\DD$ has a fixed doorjamb.
\end{itemize}

\begin{figure}[h]
	\begin{center}
		\includegraphics[scale=0.65]{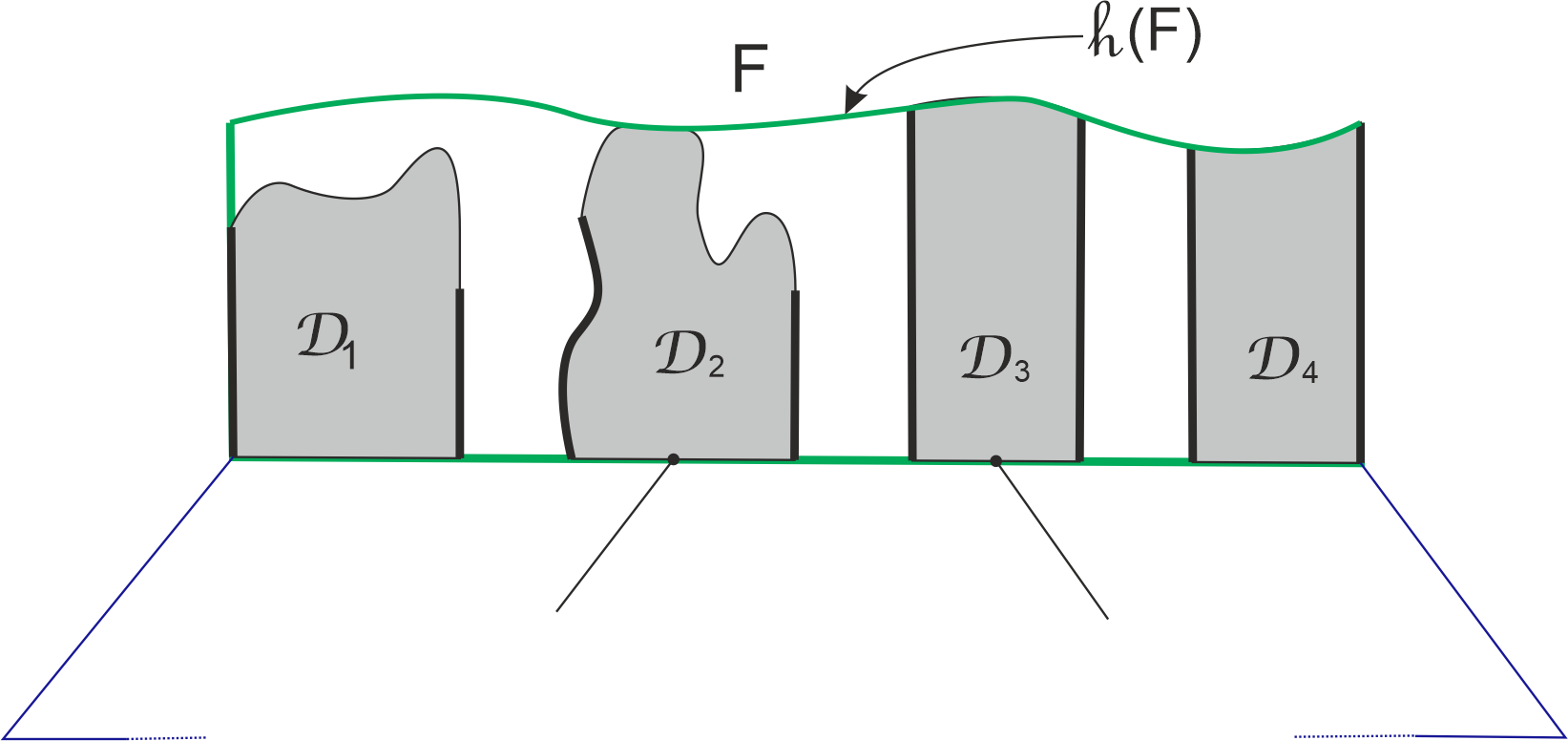}
	\end{center}
	\caption{Predoors $\mathcal{D}_1$, $\mathcal{D}_2$ and doors  $\mathcal{D}_3$, $\mathcal{D}_4$, all of them with their doorjams in bold line.}
	\label{Fig:Predoor_door}
\end{figure}

Once we have assigned a handrail to $\DD$, we will say that $\DD$ is a {\em framed pre-door}. Notice that there is only a possible frame in  case $\DD$ is contained in the lid. Moreover, in this case, $\Sat_{\CC}(h(\DD))\cap F_\CC$ is the  handrail of the fence of the refinement $\overline{\Sat_{\CC}(\DD)}<\CC$.  On the contrary, if $\DD\subset F_\CC$, the two intervals $J,J'$ in the definition of frame are 
 called the {\em doorjambs} of the framed pre-door $\DD$. One of them coincides with the fixed doorjamb of $\DD$, in case it exists (hence, it does not depend on the chosen frame). A doorjamb of $\DD$ that is not equal to the fixed doorjamb is called an {\em unfixed doorjamb} of $\DD$.

Two pre-doors $\DD_1,\DD_2$ in $\CC$ will be said {\em compatible} if either $\DD_1\cap\DD_2=\emptyset$ or if there exist respective frames such that $\DD_1\cap\DD_2$ is a common doorjamb of both of them (necessarily a common unfixed doorjamb).

\begin{definition}
 A {\em door of} the c-nbdd $\CC$ with  fence $F$ and lid $L$ is a framed pre-door $\DD$ in $\CC$ such that:
\begin{enumerate}[(i)]
\item If $\DD\subset L$, then $\DD=L$.
 \item If $\DD\subset F$, then the handrail of $\DD$ is equal to $h(\DD)=\DD\cap h(F)$.
 \item If $\DD \subset F$ and $\DD$ has a fixed doorjamb $J$, then $J$ is also a doorjamb of $F$.
\end{enumerate} 
We will also say that $\DD$ is an {\em in-door} or an {\em out-door} if it is contained either in $\partial\CC^{in}$ or in $\partial\CC^{out}$, respectively.
\end{definition}

\subsection{Distinguished fattenings of the graph}

We use the notations introduced in Section~\ref{sec:SHNR-foliations}. In particular, recall that if $G<\Omega$, then $\mathcal{V}(G)$ denotes the family of local s-components at vertices of $G$.

\begin{definition}\label{def:fattening}
Let $G<\Omega$ be a subgraph of $\Omega$. A {\em fattening of $G$}
 is a map
$\KK$ from $\VV(G)\cup E(G)$ into the family of compact subsets of $M$ satisfying:

\begin{enumerate}[(i)]
  \item For each $\nu\in\VV(G)$, the set $\KK(\nu)$ is a c-nbdh in $\nu$.
  \item For each edge $\sigma=[p,q] \in E(G)$, there are points $a_1, a_2 \in \sigma$ and respective transversal sections $\Sigma_i$ at $a_i$, for $i=1,2$, both compact discs, such that $\mathcal{K}(\sigma)$ is a flow box, with respect to the foliation $\mathcal{L}$, from $\Sigma_1$ to $\Sigma_2$.
  \item The set
  $$
  |\KK|=\bigcup_{\nu\in\VV(G)}\KK(\nu)\;\;\cup\;
  \bigcup_{\sigma\in E(G)}\KK(\sigma),
  $$
  is a  neighborhood of $|G|$ in $M$. It is called the {\em support of $\KK$}.
\end{enumerate}
\end{definition}
The elements $\KK(\sigma)$ for $\sigma\in E(G)$ are also called {\em tubes}. 
Notice that the transversal sections $\Sigma_1,\Sigma_2$ in Definition~\ref{def:fattening} are the inner and outer frontier of the tube $\KK(\sigma)$, denoted by $\partial\KK(\sigma)^{in}$, $\partial\KK(\sigma)^{out}$, respectively.

Given a fattening $\KK$ of $G$, a {\em refinement of $\KK$} is a fattening $\wt{\KK}$ of $G$ such that $\wt{\KK}(\nu)$ is a refinement of c-nbhd of $\KK(\nu)$ for any $\nu\in V(G)$, and such that $\wt{\KK}(\sigma)$ satisfies $\partial\wt{\KK}(\sigma)^{in}\subset\partial\KK(\sigma)^{in}$ and $\partial\wt{\KK}(\sigma)^{out}\subset\partial\KK(\sigma)^{out}$, for any $\sigma\in E(G)$ (we just say that $\wt{\KK}(\sigma)$ is a refinement of the tube $\KK(\sigma)$). Notice that a refinement $\wt{\KK}(\sigma)$ is completely determined by either its inner or its outer frontier.
On the other hand, if $\KK$ is a fattening of $G$ and $G'<G$ is a subgraph, we denote by $\KK|_{G'}$ the fattening of $G'$ given by the restriction of $\KK$ to $\VV(G') \cup E(G')$.

\begin{definition}\label{def:distinguished}
Let $\KK$ be a fattening of a subgraph $G<\Omega$. Given $p\in V(G)$, we will say that $\KK$ is {\em pre-distinguished at $p$} if the following properties hold:
 \begin{enumerate}[(a)]
 \item If $\sigma,\tau$ are different edges of $G$ adjacent to $p$, then $\KK(\sigma)\cap\KK(\tau)=\emptyset$.
   \item If $p\in S_{tr}$ and $\nu_p^+,\nu_p^-$ are the two local s-components at $p$, then the c-nbhs $\KK(\nu_p^+)$ and $\KK(\nu_p^-)$ have equal base.
   \item For any edge $\sigma\in E(G)$ starting at $p$ and for any $\nu\in\tilde{\alpha}(\sigma)$,    
   the set $\partial\KK(\sigma)^{in}\cap\KK(\nu)$ is a door of $\KK(\nu)$, and we have $\partial\KK(\sigma)^{in}=\bigcup_{\nu\in\tilde{\alpha}(\sigma)}\partial\KK(\sigma)^{in}\cap\KK(\nu)$. Moreover, if $p\in S_{tr}$ and $\sigma\cap W^2_p\ne\emptyset$, then the two doors $\partial\KK(\sigma)^{in}\cap\KK(\nu^\epsilon_p)$, $\epsilon\in\{+,-\}$, have equal base.
  \item  For any edge $\sigma\in E(G)$ ending at $p$ and for any $\nu\in\omega(\sigma)$, the set $\partial\KK(\sigma)^{out}\cap\KK(\nu)$ is a pre-door of $\KK(\nu)$, and we have $\partial\KK(\sigma)^{out}=\bigcup_{\nu\in
  	\tilde{\omega}(\sigma)}\partial\KK(\sigma)^{out}\cap\KK(\nu)$. Moreover, if $p\in S_{tr}$ and $\sigma\cap W^2_p\ne\emptyset$ then the two pre-doors $\partial\KK(\sigma)^{out}\cap\KK(\nu^\epsilon_p)$, $\epsilon\in\{+,-\}$, have equal base.

 \end{enumerate}

The fattening $\KK$ will be called {\em distinguished} at $p$ (or also {\em distinguished at $\nu$} if $\nu$ is a local s-component at $p$) if $\KK$ is pre-distinguished at $p$ and in item (d) we may replace ``pre-door'' with ``door''. 
Finally, $\KK$ is said to be {\em pre-distinguished} (resp. {\em distinguished}) if it is pre-distinguished (resp. distinguished) at every $p\in V(G)$.
\end{definition}

Notice that being distinguised is the same thing as being pre-distinguished both for $\LL$ and for the reverse foliation $-\LL$.

\begin{figure}[h]
	\begin{center}
		\includegraphics[scale=0.65]{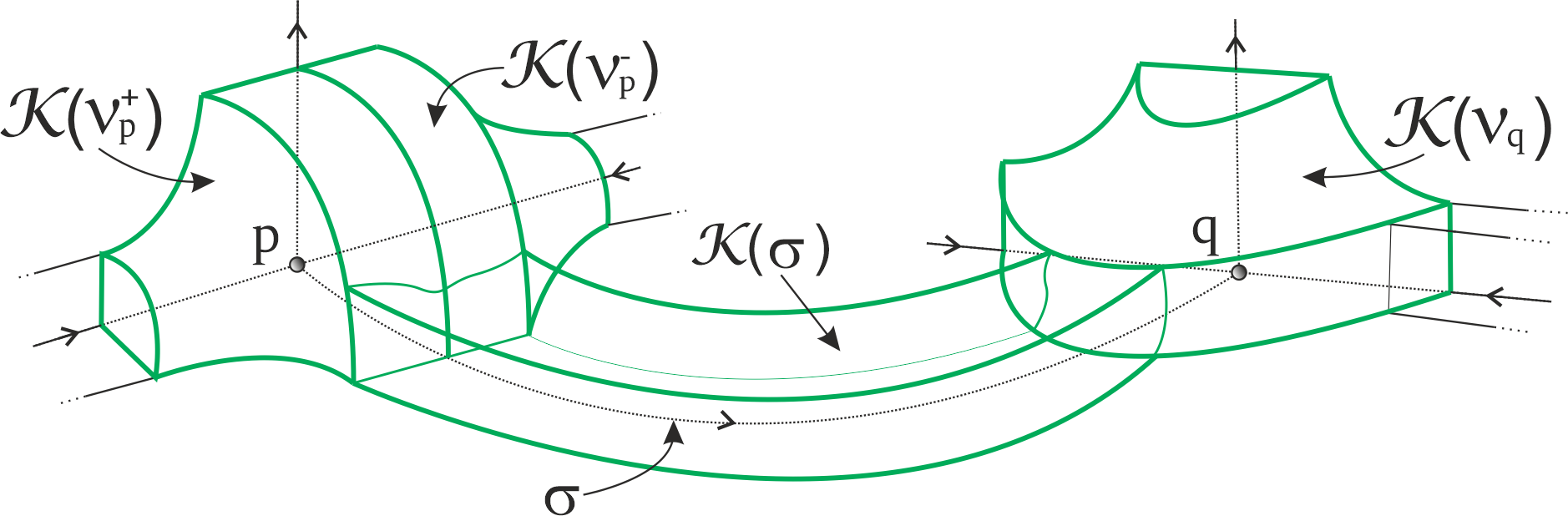}
	\end{center}
	\caption{Distinguished fattening.}
	\label{Fig:Distinguidos1}
\end{figure}

\begin{definition}\label{def:free-doors}
Let $\KK$ be a pre-distinguished fattening of the whole graph $\Omega$ and let $\nu$ be a local s-component at some $p\in V(\Omega)$. Assume that $\KK$ is distinguished at $p\in V(\Omega)$ or that $\partial\KK(\nu)^{out}=F_{\KK(\nu)}$. In this case, an {\em unfree $\KK$-door at $\nu$} (or {\em in $\KK(\nu)$}) is any of the sets of the form $F_{\KK(\nu)}\cap\KK(\sigma)$, where $\sigma$ is an edge satisfying $\sigma\cap\overline{\nu}\ne\emptyset$. Also, a {\em free $\KK$-door at $\nu$} (or {\em in $\KK(\nu)$}) is the closure of a connected component of $F_{\KK(\nu)}\setminus H$, where $H$ is the union of the unfree $\KK$-doors at $\nu$.   
\end{definition}

Notice that both the unfree or free $\KK$-doors at $\nu$ are actually doors of the c-nbhd $\KK(\nu)$. Besides, a free $\KK$-door does not cut the support of the graph $\Omega$. When $\KK$ is distinguished, given a face $\G$ of $\Omega$, there are exactly two free $\KK$-doors intersecting $\G$ and we say that they are {\em associated to} $\G$. One of them, denoted by $\DD^{out}_{\KK,\G}$, is an out-door of $\KK(\tilde{\alpha}(\G))$ and the other one, denoted by $\DD^{in}_{\KK,\G}$, is an in-door at $\KK(\tilde{\omega}(\G))$.   
On the other hand, if $p$ is a transversal saddle and $\KK$ is distinguished at $p$ then there is a unique free $\KK$-door at any of the two local s-components $\nu_p^+,\nu_p^-$ and none of them cuts the divisor. We say also that they are {\em associated to $p$}.
Notice that any free $\KK$-door is either associated to a face or associated to a transversal saddle point.

\strut

In the following statement we summarize the structure of the points on the frontier of a distinguished fattening support according to the definitions in Section~\ref{sec:SHNR-foliations} (see Figure~\ref{Fig:Distinguidos2}). The proof follows straightforwardly by construction. 

\begin{proposition}\label{pro:trans-distinguished}
	Let $\KK$ be a distinguished fattening of $\Omega$ and $\FF$ the family of all free $\KK$-doors.  	
	Then, the transversal frontier of $|\KK|$ is equal to
	$$
	Fr(|\KK|)^\pitchfork=\bigcup_{\DD\in\FF}\DD\;\cup\;\bigcup_{q\in N} L_{\KK(\nu_q)},
	$$
		and this set does not contain points of type e-e nor i-i (relatively to $|\KK|$). Consider now $p\in V(\Omega)$, $\nu$ a local s-component at $p$ and $\DD$ a  free $\KK$-door  in $\KK(\nu)$. Assume for instance that  $\DD\subset\partial\KK(\nu)^{in}$. Let $J_1,J_2$ be the doorjambs of $\DD$ with extremities $h(J_i):=J_i\cap h(\DD)$, $b(J_i):=J_i\cap b(\DD)$, for $i=1,2$. Then we have the following types, relatively to $|\KK|$, of points in $Fr(|\KK|)^\pitchfork$: 
	\begin{enumerate}[(i)]
		\item If $p\in S$, then  $$int(\DD)\cup\dot{b(\DD)},\;\; \dot{h(\DD)},\;\;\dot{J_1}\cup\dot{J_2}\cup\{b(J_1),b(J_2)\},\;\;\{h(J_1),h(J_2)\}
		$$ are, respectively, the sets of points in $\DD$ of type e-i, e-t, t-i and t-t. 
		\item If $p\in N$ and $p$ is a three-dimensional saddle, then 
		the interior (resp. the boundary) of $ L_{\KK(\nu)}$ is the set of points in $ L_{\KK(\nu)}$ of type i-e (resp. t-e). 
		\item If $p\in N$ and $p$ is not a three-dimensional saddle, then any point of $int(L_{\KK(\nu)})$ is of type i-e whereas the sets
		$$
		\DD\setminus(J_1\cup J_2),\;\;J_1\cup J_2
		$$
		are, respectively, the sets of points in $\DD$ of type e-i and t-i.
	\end{enumerate}
The situation in completely analogous in case $\DD\subset\partial\KK(\nu)^{out}$.
\end{proposition}

 \begin{figure}[h]
	\begin{center}
		\includegraphics[scale=0.65]{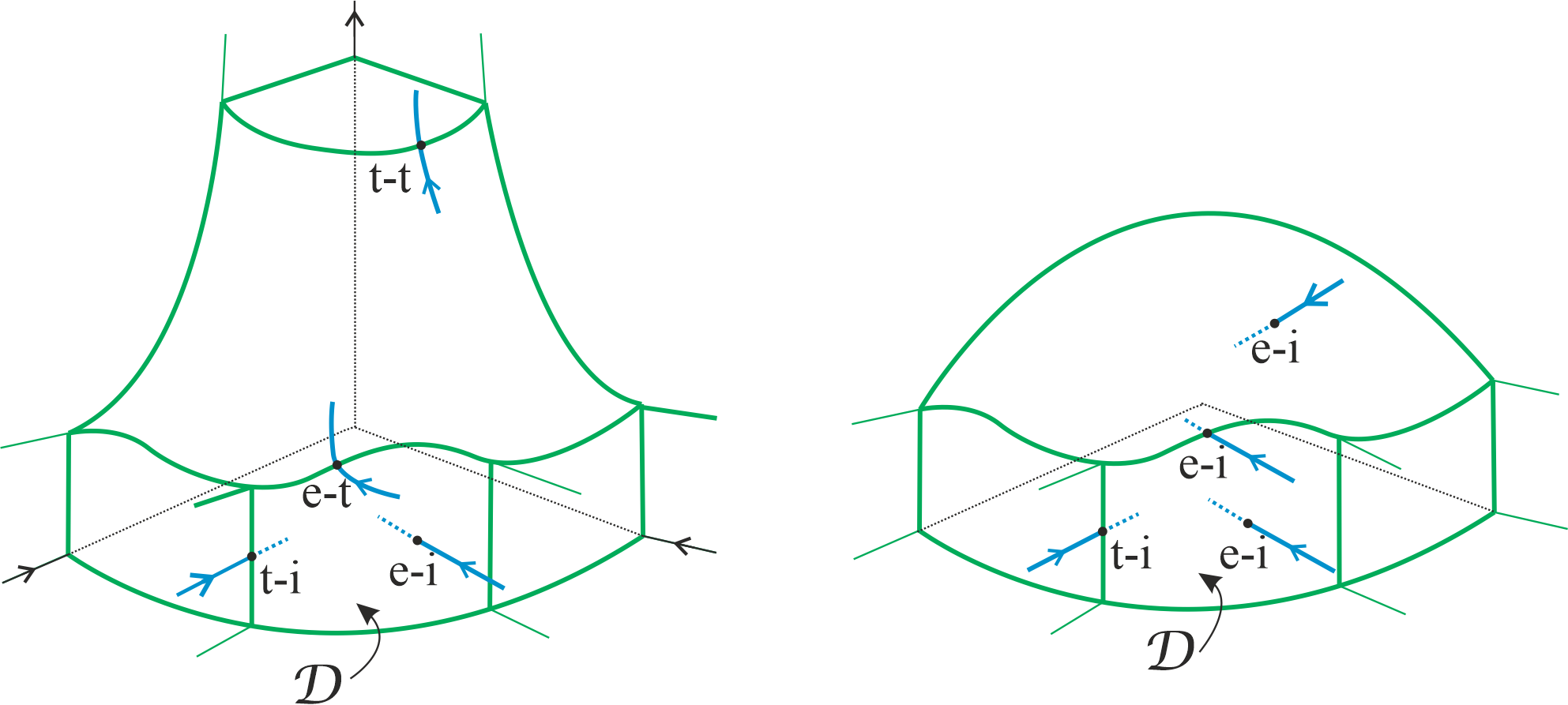}
	\end{center}
	\caption{Different type points.}
	\label{Fig:Distinguidos2}
\end{figure}

\subsection{Obtaining distinguished fattenings}

It is clear that the family of supports of fattenings of $\Omega$ forms a base of neighborhoods of $|\Omega|$. Our main objective in this subsection is to prove that this is also true for the family of supports of distinguished fattenings. Let us first introduce a definition and a lemma whose proof is direct.
\begin{definition}
	Let $\nu\in\VV(\Omega)$ and let $\CC$ be a c-nbhd in $\nu$. A {\em system of entrances (resp. sources)} of $\CC$ is a finite collection $\mathcal{E}=\{\DD_1,\dots,\DD_r\}$ of mutually compatible framed pre-doors in $\partial\CC^{in}$ (resp. $\partial\CC^{out}$). The system $\mathcal{E}$ will be called {\em complete} if the set  $\CC'=\Sat_\CC(\bigcup_{i=1}^r\DD_i)$ is a refinement of $\CC$. In this last case, we say that $\CC'$ is the c-nbhd {\em associated} to the complete system $\EE$.
\end{definition}

Notice that, if some $\DD\in\EE$ is contained in $L_\CC$, then $\EE=\{\DD\}$ and $\EE$ is complete. On the other hand, if  $\EE=\{\DD_1,\dots,\DD_r\}$ and any $\DD_i\subset F_\CC$, then $\EE$ is complete if, and only if, $\bigcup_{i=1}^rb(\DD_i)=b(\CC)$. Finally, if $\EE$ is a complete system of entrances and $\CC'$ is the associated c-nbhd, then each element of $\EE$ is a door of $\CC'$.

\begin{lemma}\label{lm:cled}
	Any system of entrances $\EE=\{\DD_1,\dots,\DD_r\}$ of a c-nbhd $\CC$ can be extended to a complete one $\wt{\EE}\supset\EE$. As a consequence, there exists a refinement $\CC'<\CC$ for which any $\DD_i$ is a door of $\CC'$.
\end{lemma}
\begin{theorem}\label{th:distinguished}
Let $\mathcal{M}=
(M,D,\LL)$ be a non s-resonant of Morse-Smale type HAFVSD and  $\Omega$ its associated graph.
 Then given any fattening $\KK$ of $\Omega$, there exists a distinguished fattening $\wt{\KK}$ of $\Omega$ such that $|\wt{\KK}|\subset|\KK|$.
\end{theorem}

In the proof of Theorem~\ref{th:distinguished}, we proceed recursively on the components of the filtration (\ref{eq:filtration-graph}) of the graph $\Omega$ defined in Section~\ref{sec:SHNR-foliations}. Let us notice that we may assume that $\KK$ satisfies that for any $p\in S_{tr}$, the two c-nbds $\KK(\nu^+_p)$ and $\KK(\nu^-_p)$ have the same base.

{\em First step.- Getting pre-distinguished fattenings.} We first show that the set of pre-distinguished fattenings is a base of neighborhoods of $|\Omega|$. More precisely, given a fattening  $\KK$ of $\Omega$, there exists a pre-distinguished fattening $\KK'$ such that $|\KK'|\subset|\KK|$.
To do this, we use the filtration (\ref{eq:filtration-graph}), where $l:=l(\Omega)$ is the length of $\Omega$, and we construct, by recurrence on $j=0,\dots,l$, a pre-distinguished fattening $\KK^j$ of $\Omega^j$ such that $|\KK^j|\subset|\KK|_{\Omega^j}|$. At the end, the desired pre-distinguished fattening is given by $\KK':=\KK^l$. 

For $j=0$ we just take $\KK^0=\KK|_{\Omega^0}$. 

Assume  that $\KK^{j-1}$ is constructed for $j>0$. Given an edge $\sigma\in E(\Omega^j)\setminus E(\Omega^{j-1})$, we consider a tube $\TT_\sigma\subset\KK(\sigma)$ from a disc $A(\sigma)$ to a disc $B(\sigma)$ such that, if $\sigma=[p,q]$, they satisfy the following conditions: 
\begin{enumerate}[(i)]
	\item If $\wt{\alpha}(\sigma)=\{\nu\}$, then $A(\sigma)$ is a pre-door of $\KK(\nu)$ contained in $\partial\KK(\nu)^{out}$ with center the point $\partial\KK(\nu)^{out}\cap \sigma$.
	\item If $p\in S_{tr}$ and $\wt{\alpha}(\sigma)=\{\nu^+_p,\nu^-_p\}$, then $A(\sigma)$ is the union of two pre-doors with equal base, one of them in $\partial \KK(\nu_{p}^+)^{out}$ and the other one in $\partial \KK(\nu_{p}^-)^{out}$, both with equal center (at $\partial \KK(\nu_{p}^+)^{out}\cap\sigma$).
	\item If $\tilde{\omega}(\sigma)=\mu$ then $B(\sigma)$ is a pre-door of $\KK(\mu)$ contained in $\partial\KK(\mu)^{in}$.
	\item If $q\in S_{tr}$ and $\wt{\omega}(\sigma)=\{\nu^+_q,\nu^-_q\}$, then $B(\sigma)$ is the union of two pre-doors with equal base and equal center, one of them in $\partial \KK(\nu_{q}^+)^{in}$ and the other one in $\partial \KK(\nu_{q}^-)^{in}$.
\end{enumerate}
Notice that, in the situation (ii) (resp. in (iv)) we just have to guarantee that the boundary of $A(\sigma)$ (resp. $B(\sigma)$) cuts only once $W^2_p$ (resp. $W^2_q$). Observe also that both (ii), (iv) do not occur simultaneously by the Morse-Smale condition\footnote{In any case, there is no particular problem here to guarantee both conditions (ii) and (iv), because $W^2_p$, $W^2_q$ are analytic manifolds.}.
Moreover, the tubes $\TT_\sigma$ for $\sigma\in E(\Omega^j)\setminus E(\Omega^{j-1})$ should be chosen small enough to be mutually disjoint. 
We put $\KK^{j-1}(\sigma):=\TT_\sigma$ for any such $\sigma$.

  Now, if $\nu\in\VV(\Omega^j)\setminus\VV(\Omega^{j-1})$, the family $\mathcal{E}_{\nu}=\{A(\sigma)\cap\KK(\nu)\,:\,\nu\in\wt{\alpha}(\sigma)\}$  is a system of sources of $\KK(\nu)$. Applying Lemma~\ref{lm:cled}, there exists a refinement  $\CC_\nu<\KK(\nu)$ such that each element of $\EE_\nu$ is a door of $\CC_\nu$. Put $\KK^j(\nu):=\CC_\nu$. Summarizing, we have defined $\KK^j$ at the local s-components at vertices of $\Omega^j$ that are not in $\Omega^{j-1}$, and at the edges starting at them. We extend these values to $\Omega^{j-1}<\Omega^j$ by putting $\KK^j|_{\Omega^{j-1}}=\KK^{j-1}$ so that we get a fattening on $\Omega^j$, using (\ref{eq:edges-j}). By construction and recurrence, $\KK^j$ is pre-distinguished and $|\KK^j|\subset|\KK|_{\Omega^j}|$, as wanted.

\strut

{\em Second step.- Getting distinguished fattenings.} To finish the proof of Theorem~\ref{th:distinguished}, it is enough to prove that any pre-distinguished fattening has a distinguished refinement.
This will be a consequence of the following more general result, which will be useful to us later. It is stated for subgraphs of $\Omega$ and asserts that the refinement $\wt{\KK}$ can be chosen to preserve some prescribed systems of entrances.

\begin{proposition}\label{pro:getting-distinguished}
Let $G$ be a subgraph of $\Omega$ and let $\KK$ be a pre-distinguished fattening of $G$. Suppose that for any $\nu\in\VV(G)$ there exists a (possible empty) system of entrances $\mathcal{E}_\nu$ of $\KK(\nu)$ satisfying the following property
	\begin{equation}\label{eq:condition-E-nu}
	\DD\cap\partial\KK(\sigma)^{out}=\emptyset\;\;\mbox{ for any }\DD\in\mathcal{E}_\nu\mbox{ and for any }\sigma\in E(G)\cap\tilde{\omega}^{-1}(\nu).
	\end{equation}
	Then there exists a distinguished refinement $\wt{\KK}$ of $\KK$  such that, for any $\nu\in \mathcal{V}(G)$ and for any $\DD\in\mathcal{E}_\nu$, $\DD$ is a door of $\wt{\KK}(\nu)$.
\end{proposition}
\begin{proof}
	Let $g=g(G)$ be the maximal length  of the vertices of $G$. Notice that the filtration (\ref{eq:filtration-graph}) induces a filtration
	$$
	G^0<G^1<\cdots<G^g=G^{g+1}=\cdots=G^l=G,
	$$
	where, for any $j$, the graph $G^j$ is the intersection of $\Omega^j$ and $G$, that is, $V(G^j)=V(\Omega^j)\cap V(G)$ and $E(G^j)=E(\Omega^j)\cap E(G)$.
	The proof goes by induction under $g$.
	
	If $g=0$, then $G=G^0$ consists only of finitely many $D$-nodes (all of them attractors, except, possibly, for the exceptional case described in Lemme~\ref{lm:graph}, (iv)) and contains no edges. The proposition follows by applying Lemma~\ref{lm:cled} to every element of $\VV(G)$.
	
	Suppose that $g>0$ and assume that the result holds for any subgraph $G'<\Omega$ with $g(G')<g$. Let $\nu\in\VV(G)\setminus \VV(G^{g-1})$. First, we apply Lemma~\ref{lm:cled} to the system $\EE_\nu$, so that we obtain a refinement $\CC_\nu<\KK(\nu)$ for which any $\DD\in\mathcal{E}_{\nu}$ is a door of $\CC_\nu$. Now, for any given $\sigma\in\tilde{\alpha}^{-1}(\nu)$, we can take a refinement  $\mathcal{T}_\sigma<\KK(\sigma)$ determined by an inner transversal part $\partial\mathcal{T}_\sigma^{in}\subset\partial\KK(\sigma)^{in}$
	in such a way that, being $\{c_\sigma\}=\sigma\cap\partial\KK(\nu)$, the following conditions are satisfied:
	
	(a) If $\wt{\alpha}(\sigma)=\{\nu\}$ and $\wt{\omega}(\sigma)=\{\eta\}$ are both singletons, then we choose $\partial\mathcal{T}_\sigma^{in}$ to be a door of $\CC_\nu$ contained in $\partial\CC_\nu^{out}$ and centered at $c_\sigma$ (hence, $\partial\mathcal{T}_\sigma^{out}$ is a pre-door in $\partial\KK(\eta)^{in}$).
	
	(b) If $\wt{\alpha}(\sigma)=\{\nu\}$ and  $\wt{\omega}(\sigma)=\{\nu^+_q,\nu^-_q\}$ with $q\in S_{tr}$, then we choose $\partial\mathcal{T}_\sigma^{in}$ to be a door in $\partial\CC_\nu^{out}$ centered at $c_\sigma$ and such that the resulting outer part $\partial\mathcal{T}_\sigma^{out}$ is the union of a pre-door in $\partial\KK(\nu_q^+)^{in}$ and a pre-door in $\partial\KK(\nu_q^-)^{in}$, both sharing their base.
	
	(c) If $\nu$ is associated to some $p\in S_{tr}$ and $\wt{\alpha}(\sigma)=\{\nu^+_p,\nu^-_p\}=\{\nu,\nu'\}$, then we choose $\partial\mathcal{T}_\sigma^{in}=A\cup A'$, where $A$ is a door in $\partial\CC_\nu^{out}$ and $A'$ is a door in $\partial\CC_{\nu'}^{out}$, both centered at $c_\sigma$ with equal base. Notice that the Morse-Smale hypothesis implies that in this case $\wt{\omega}(\sigma)$ is a singleton. We conclude, as in case (a), that $\partial\mathcal{T}_\sigma^{out}$ is a pre-door in $\partial\KK(\nu_q)^{in}$.
	
	To see that there exists such a tube $\mathcal{T}_\sigma$ (determined by its inner frontier) with the above properties, only the case (b) deserves a comment: in this case, the Morse-Smale hypothesis implies that $\sigma$ crosses the boundary of  $\CC_\nu$  necessarily through the fence, thus equal to $\partial\CC_\nu^{out}$. This is enough to guarantee that the door $\partial\mathcal{T}_\sigma^{in}$ can be chosen so that its boundary cuts the analytic curve $\Sat_{|\KK|}(W^2_q)\cap\partial\KK(\nu)^{out}$ only at $c_\sigma$  and at another point belonging to the handrail $h(\partial\mathcal{T}_\sigma^{in})\subset h(F_{\CC_\nu})$ (see Figure \ref{Fig:DemosDistinguidos}). As a consequence, $\partial\mathcal{T}_\sigma^{out}$ intersects $W^2_q$ along a closed interval and hence it holds the required property stated in (b).
	
\begin{figure}[h]
	\begin{center}
		\includegraphics[scale=0.65]{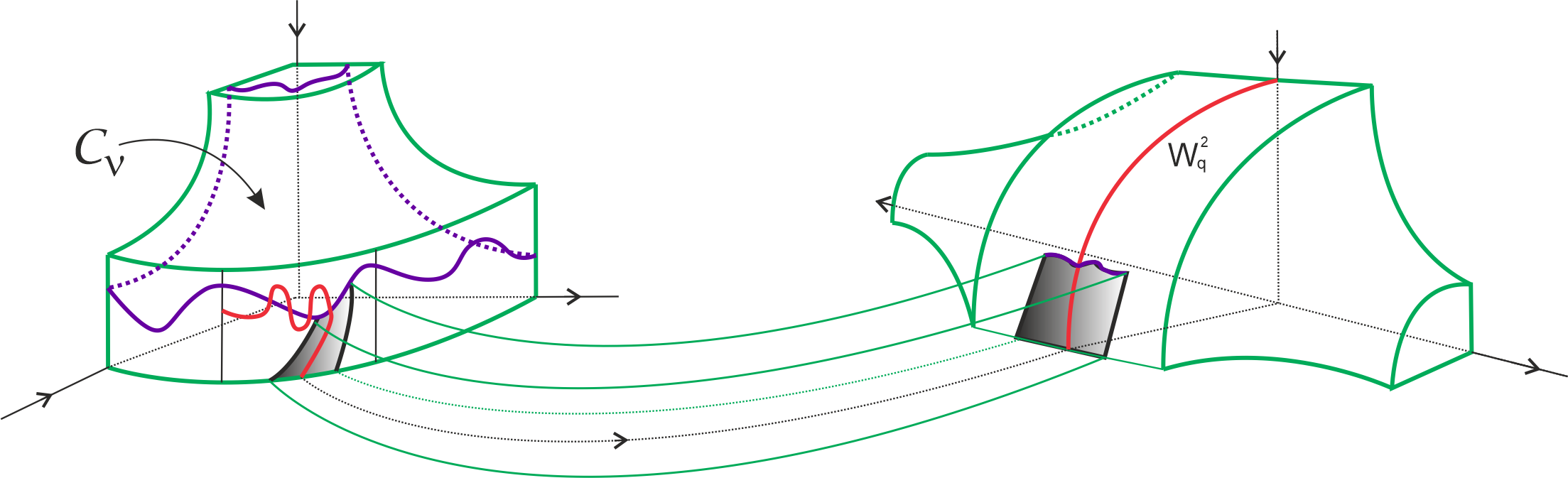}
	\end{center}
	\caption{Getting distinguished neigborhoods.}
	\label{Fig:DemosDistinguidos}
\end{figure}
	
	Now we consider for any $\mu\in \mathcal{V}(G^{g-1})$ the following family of pre-doors of $\KK(\mu)$ in $\partial\KK(\mu)^{in}$:
	\begin{equation}\label{eq:tilde-E-mu}
	\wt{\mathcal{E}}_{\mu}\!=\!\mathcal{E}_{\mu}
	\bigcup
	\{\partial\mathcal{T}_\sigma^{out}\!\cap\!\partial\KK(\mu)^{in}:
\sigma\!\in\!E(G)\mbox{ with }\mu\!\in\!\tilde{\omega}(\sigma),\,
\tilde{\alpha}(\sigma)\cap\mathcal{V}(G^{g-1})=\!\emptyset\}.
	\end{equation}
	Using (\ref{eq:condition-E-nu}) for $\EE_\mu$ and the construction of the tubes $\TT_\sigma$ above, we have that the family $\wt{\mathcal{E}}_{\mu}$ is a system of entrances of the c-nbhd $\KK(\mu)$ for any $\mu\in\mathcal{V}(G^{g-1})$. Moreover, the collection $\{\wt{\EE}_\mu\}$ still satisfies (\ref{eq:condition-E-nu}) when we replace $G$ by $G^{g-1}$ (notice that $\wt{\EE}_\mu$ only differs from $\EE_\mu$ if $\mu\in\VV(G^{g-1})\setminus\VV(G^{g-2})$, and there are no edges of the graph $G^{g-1}$ ending at those local s-components).  By the induction hypothesis, there exists a distinguished refinement $\wt{\KK}^1$ of $\KK|_{G^{g-1}}$ such that for any $\mu\in\mathcal{V}(G^{g-1})$ and for any $\DD\in\wt{\mathcal{E}}_{\mu}$, $\DD$ is a door of $\wt{\KK}^1(\mu)$. Let $\wt{\KK}$  be the fattening of $G=G^g$ defined by $\wt{\KK}|_{G^{g-1}}=\wt{\KK}^1$ and also by putting $\wt{\KK}(\nu)=\CC_\nu$ and $\wt{\KK}(\sigma)=\mathcal{T}_\sigma$, for any $\nu\in\VV(G)\setminus\VV(G^{g-1})$ and any $\sigma\in E(G)\setminus E(G^{g-1})$. By construction, we have that $\wt{\KK}$ is a distinguished refinement of $\KK$ satisfying the desired requirements.
\end{proof}

{\em Proof of Theorem~\ref{th:distinguished}.-}
In light of the first step above described, we may assume that the initial fattening $\KK$ is pre-distinguished. Applying Proposition~\ref{pro:getting-distinguished} to $G=\Omega$ and taking $\EE_\nu=\emptyset$ for any $\nu\in\VV(\Omega)$, we have that there is a distinguished refinement $\wt{\KK}<\KK$, and we are done.
$\hfill{\square}$

\begin{scholium}\label{sch-prop-distinguished}
	{\em
		It is clear that the distinguished refinements construction given in the proof of Proposition~\ref{pro:getting-distinguished} is by no means unique, so it can be adapted to many different situations. We take advantage of this adaptability along the rest of the paper in order to construct distinguished fattenings with additional properties. To systematize the arguments and notations in such constructions, let us summarize the proof of  Proposition~\ref{pro:getting-distinguished} in the following way:

The resulting distinguished refinement $\wt{\KK}<\KK$ is constructed recursively as a final step in a sequence
\begin{equation}\label{eq:process-dist}
		\KK=\NN_{g+1}>\NN_{g}>\cdots>\NN_{1}>\NN_{0}=\wt{\KK}
\end{equation}
		of predistinguished fattenings of $G$
		such that, for $j\in\{0,1,\dots,g\}$, $\NN_j$ only differs from $\NN_{j+1}$ on the local s-components at points of $V(G^{j})\setminus V(G^{j-1})$ and also on the edges starting at them (with $V(G^{-1})=\emptyset$). In other words, for any such index $j$ (and putting $G^{g+1}=G$), we have
\begin{equation}\label{eq:prop-scholium}
\begin{array}{c}
  \NN_{j}|_{G^{j-1}}=\KK|_{G^{j-1}}, \\
  \NN_j|_{(G^{j+1})^c\cap G}=\NN_{j+1}|_{(G^{j+1})^c\cap G},
\end{array}
\end{equation}
where $(G^k)^c\cap G$ is the complement of $G^k$ in $G$, i.e., the
subgraph of $G$ generated by the edges of $E(G)\setminus E(G^k)$ (cf. Section~\ref{sec:SHNR-foliations}). In particular, $\NN_j$ is distinguished at any $\nu\in\VV(G)\setminus\VV(G^{j-1})$, for any $j$.
}
\end{scholium}

\section{Good saturations}\label{sec:good-saturations} 

As mentioned in the introduction, the fitting neighborhoods we are looking for are obtained by extending the support of an appropriate distinguished fattening. To this end, we adapt the construction in Proposition~\ref{pro:getting-distinguished} in order to get distinguished fattenings with controlled free doors saturations  (the {\em good saturations} property). Recall that these free doors form essentially the transversal frontier of the fattening (Proposition~\ref{pro:trans-distinguished}) and those associated to faces of the graph must be ``closed'' to the purpose of getting fitting domains. 

We fix again a non s-resonant of Morse-Smale type HAFVSD $\MM=(M,D,\LL)$.
\subsection{Definitions and statements}

Let $\KK$ be a pre-distinguished fattening of $\Omega$. Given $p\in S_{tr}$, a {\em fixed mark (of $\KK$) at $p$} is the intersection of $W^2_p$ with the boundary of a tube $\KK(\sigma)$, where $\sigma$ is adjacent to $p$ and intersects $W^2_p$. If $\nu$ is a s-component at $p$, we say also that we have a fixed mark (of $\KK$) at $\nu$. Since there are two of those edges $\sigma$, there are also two fixed marks at $p$. If we need to distinguish them, we just say that the fixed mark is {\em associated} to the edge $\sigma$. Notice that a fixed mark associated to $\sigma$ provides, by saturation, a trace mark in the local s-component at the extremity of $\sigma$ that is not $p$ (cf. proof of Corollary~\ref{cor:sat-w2}).

\begin{definition}\label{def:good-saturations}
	Let $\KK$ be a distinguished fattening of $\Omega$. We say that $\KK$ has
	\begin{itemize}
\item  {\em Good saturations for fixed marks (gsfm)} if given $p,q\in S_{tr}$ with $p\ne q$ and fixed marks $I_p,I_q$ of $\KK$ at $p$ and $q$, respectively, we have
$$
\Sat_{|\KK|}(I_p)\cap\Sat_{|\KK|}(I_q)=\emptyset.
$$
\item {\em Good saturations for free doors (gsfd)} if for any pair of different free $\KK$-doors $\DD,\DD'$ which are not associated to the same face of $\Omega$, we have
		$$
		\Sat_{|\KK|}(\DD)\cap\Sat_{|\KK|}(\DD')=\emptyset.
		$$
\item  {\em Good saturations for fixed marks with respect to free doors (gsfmfd)} if for any fixed mark $I_p$ at some $p\in S_{tr}$ and any free $\KK$-door $\DD$ at $\nu\in\VV(\Omega)$ such that $\nu\not\in\{\nu^+_p,\nu^-_p\}$, we have
$$
\Sat_{|\KK|}(I_p)\cap\DD=\emptyset. 
$$
\end{itemize}
	We say that $\KK$ has {\em good saturations} if the three conditions above hold.
\end{definition}

\strut

The main result in this section is that any distinguished fattening has a refinement with the property of good saturations. More in precise:

\begin{theorem}\label{th:good-sat}
Assume that $\mathcal{M}=(M,\LL,D)$ is non s-resonant and of of Morse-Smale type. Given a distinguished fattening $\KK$ over $\Omega$, there exists a refinement $\wt{\KK}$ of $\KK$ which is distinguished and has good saturations.
\end{theorem}
\begin{remark}\label{rm:siempre gsfm}
	{\em 
Notice that the germ along the divisor of a fixed mark does not depend on the considered fattening and that, using Corollary~\ref{cor:sat-w2}, we could assume that our given distinguished fattening $\KK$ already has good saturations for fixed marks (gsfm). Moreover, such property is preserved by refinements. On the contrary, this is not the case for conditions (gsfd) and (gsfmfd) given that they involve free doors and these elements strongly depend on $\KK$. This fact explains the technical difficulties to prove the existence of neighborhoods with the good saturations property: reasoning by recurrence on the length of subgraphs in the filtration (\ref{eq:filtration-graph}), as done in Theorem~\ref{th:distinguished}, eventually at some stage we are forced to impose proper refinements of tubes that, in turn, provoque modifications of the free doors where the good saturations condition must be reconsidered.
}
\end{remark}

\subsection{The stains}

 In a first step 
 towards the proof of Theorem~\ref{th:good-sat}, 
 we show that good saturations can be obtained by dealing just with fixed marks and free doors doorjambs saturations. Let us introduce some related notation.

\begin{definition}\label{def:stains}
	Let $\RR$ be a pre-distinguished fattening over $\Omega$. Take $\nu\in\VV(\Omega)$ and $\CC$ a refinement of $\RR(\nu)$. 
	\begin{itemize}
		\item A {\em fixed stain (of $\RR$) in $\CC$} is a non-empty subset of $F_\CC$ of the form
		$
		A=\Sat_{|\RR|}(I\setminus\{b(I)\})\cap F_\CC,
		$
		where $I$ is a fixed mark of $\RR$ at some $p\in S_{tr}$ such that $\nu\not\in\{\nu^+_p,\nu^-_p\}$ and $b(I)=I\cap D$. We say that $A$ is {\em generated at $p$} (or also {\em at} $\nu^+_p$ or {\em at} $\nu^-_p$) and that $I$ is the {\em generating mark} of $A$. Denote by $\Upsilon^\RR_\CC$ the family of fixed stains of $\RR$ in $\CC$.
\item
	A {\em mobile stain (of $\RR$) in $\CC$} is a non-empty subset of $F_\CC$ of the form
	$
	A=\Sat_{|\RR|}(J\setminus\{b(J)\})\cap F_\CC,
	$
	where $J$ is a doorjamb of a free $\RR$-door at some $\mu\in\VV(\Omega)$ at which $\RR$ is distinguished (see Definition~\ref{def:free-doors}) and $b(J)=J\cap b(\RR(\mu))$. We will say that $A$ is {\em generated at $\mu$} and that $J$ is the {\em generating mark} of $A$.
	Denote by $\Lambda^\RR_\CC$ the family of mobile stains of $\RR$ in $\CC$. 
		\end{itemize}
\end{definition}
\begin{remark}\label{rk:stains1}
	{\em
Notice that, if there exists a fixed or mobile stain $A$ in $\CC<\RR(\nu)$ generated at $\mu\in\VV(\Omega)$, then $\nu$ is related to $\mu$ with respect to the partial ordering on $\VV(\Omega)$ established in Section~\ref{sec:SHNR-foliations}. Moreover, if for instance $\mu\le\nu$ and $J$ is the generating mark of $A$, then we have
$
A=
\Sat^+_{|\RR|}(J\setminus b(J))\cap F_{\RR(\nu)}
$ 
(that is, only the saturation in one sense suffices to create a stain).
}
\end{remark}
For convenience, if $\kappa\in\{<,\le,>,\ge, \neq\}$, we will denote by $\Lambda^{\RR(\kappa)}_\CC\subset\Lambda^\RR_\CC$ the family of mobile stains in $\CC$ generated at a local s-component $\mu$ satisfying $\mu\kappa\nu$. We use also the notation $\Upsilon^{\RR(>)}_\CC$ (resp. $\Upsilon^{\RR(<)}_\CC$) for the family of fixed stains in $\CC$ generated at some $\mu$ with $\mu<\nu$ (resp. $\mu>\nu$).
 
\strut

Let us start with a lemma that provides a first description of the stains in a general fattening.

\begin{lemma}\label{lm:stains}
 Assume that $\RR$ is distinguished. Fix $\nu\in\VV(\Omega)$ and consider $A$ a fixed or mobile stain of $\RR$ in $\RR(\nu)$. Then the closure $\overline{A}$ has a finite number of connected components and intersects $D\cup b(\RR(\nu))$ along a finite (possibly empty) set of points. Moreover, if $a\in\overline{A}\cap(D\cup b(\RR(\nu))$, then the germ of $\overline{A}\setminus\{a\}$ at $a$ is non-empty and one of the following possibilities holds (See Figure \ref{Fig:Salpicaduras}):
\begin{enumerate}[(a)]
  \item  The point $a$ belongs to $F_{\RR(\nu)}\cap|\Omega|$. In this case, some representant of the germ of $\overline{A}\setminus\{a\}$ at $a$ has finitely many connected components, all of them being intervals. Moreover, given such an interval $Y$, there is a local s-component $\mu_Y$ with $\mu_Y\le\nu$ and a trace mark $T_Y$ in $\mu_Y$ such that $a\in|\Theta(\mu_Y)|$ and $Y$ is contained in $\Sat^+(T_Y)$. 
 
  \item The point $a$ belongs to $b(\RR(\nu))\setminus|\Omega|$. In this case,  $\overline{A}$ is semi-analytic at $a$. In case $\nu$ is not associated to a transversal saddle (i.e., $b(\RR(\nu))\subset D$) then the germ of $\overline{A}$ at $a$ coincides with the germ of a doorjamb of a free $\RR$-door in $\RR(\nu)$.
  \item The point $a\in D\setminus b(\RR(\nu))$. In this case $a$ is an extremity of the handrail $h(F_{\RR(\nu)})$ and the germ of $\overline{A}$ at $a$ coincides with the germ of $h(F_{\RR(\nu)})$.
\end{enumerate}
\end{lemma}

\begin{proof}
If $A$ is a stain generated at $\nu$ itself, then it is a doorjamb of a free door, that is, a mobile stain, and it satisfies (b) straightforwardly. Also, if $A$ is a fixed stain generated at some $\nu'$ immediately preceding or immediately succeeding $\nu$ then $A$ is a semianalytic interval satisfying (a). From these two starting situations, by a natural recurrence, using Remark~\ref{rk:stains1} (and up to replace $\LL$ with $-\LL$) it will be sufficient to prove the following claim:

\vspace{.2cm}

{\bf Claim.-} Let $\mu,\mu'\in\VV(\Omega)$ be two local s-components connected by an edge $\sigma$ going from $\mu$ to $\mu'$. Let $A\in\Upsilon^{\RR(<)}_{\RR(\nu)}\cup\Lambda^{\RR(<)}_{\RR(\nu)}$ and assume that $A$ satisfies all the statement properties for $\nu=\mu$. Then the subset of $F_{\RR(\mu')}$ given by
$$
A'=\Sat^+_{\RR(\mu)\cup\RR(\sigma)\cup\RR(\mu')}(A)\cap F_{\RR(\mu')}
$$ also satisfies the statement properties for $\nu=\mu'$.

\vspace{.2cm}

To be convinced why this claim ends the proof of Lemma~\ref{lm:stains}, we point out two facts about the set $A'$. In the one hand, $A'$ is contained in a stain $\wt{A}$ (with the same generating mark as $A$), but it may occur that $A'\ne\wt{A}$; in fact, the stain $\wt{A}$ is the union of several subsets of the form $A'$ corresponding to different edges ending at $\mu'$. On the other hand, $\overline{A'}$ is semi-analytic at any point of $\overline{A'}$ except possibly at points of $\overline{A'}\cap|\Omega|$, where $A'$ and the whole stain $\wt{A}$ locally coincide (using (b) or (c) for $A$). This will show that the closure of $\wt{A}$ has finitely many connected components. The rest of the statement for $\wt{A}$ will be deduced by the claim.

\begin{figure}[h]
	\begin{center}
		\includegraphics[scale=0.65]{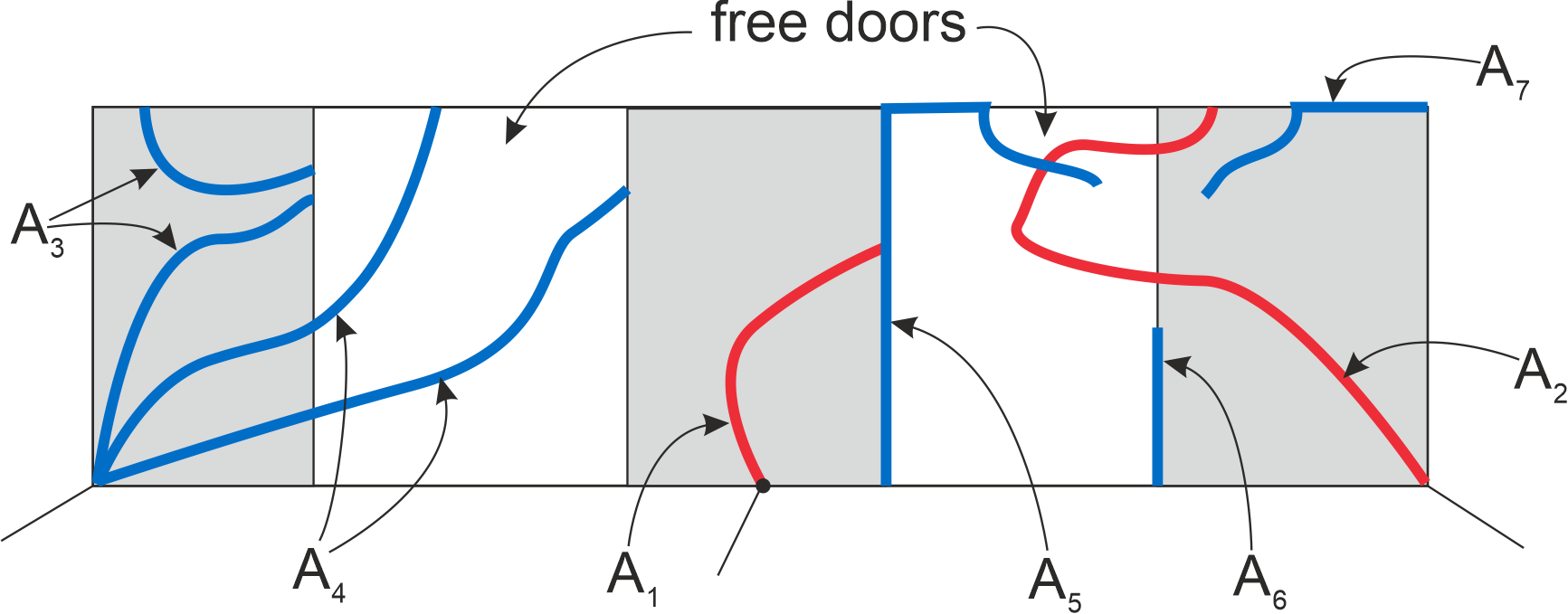}
	\end{center}
	\caption{Stains in a fence. $A_1, A_2, A_3$ and $A_4$ are in situation $a)$. $A_5, A_6$ are in situaction $b)$ and $A_7$ in situation $c)$.}
	\label{Fig:Salpicaduras}
\end{figure}

\strut To prove the claim we distinguish two cases:

\vspace{.1cm}

 {\bf Case I: $\partial\RR(\mu)^{out}=L_{\RR(\mu)}$.}  Being $\RR$ distinguished, we could have $\partial\RR(\sigma)^{out}=L_{\RR(\mu')}$, only if $\sigma$ is skeletical, by the Morse-Smale condition. Otherwise, $\partial\RR(\sigma)^{out}$ must be an unfree $\RR$-door contained in $F_{\RR(\mu')}$. Putting $H=F_{\RR(\mu')}\setminus b(\RR(\mu'))$ in the first situation and  $H=\partial\RR(\sigma)^{out}\setminus  b(\partial\RR(\sigma)^{out})$ in the second one, the flow defines an analytic isomorphism $$
\psi:F_{\RR(\mu)}\setminus b(\RR(\mu))\to H 
$$
sending $A$ to $A'$. We conclude that $A'$ has finitely many connected components. We need to show that $\overline{A'}\cap\big(D\cup b(\RR(\mu'))\big)$ is finite and that the germ of $\overline{A'}$ at each point of that set is in one of the situations (a)-(c). To do this, if $a\in\overline{A}\cap\big(D\cup b(\RR(\mu))\big)$ and $B$ is a sufficiently small representant of the germ of $\overline{A}\setminus\{a\}$ at $a$, it suffices to prove that $B'=\psi(B)$ accumulates to a single point in $D\cup b(\RR(\mu'))$ where (a), (b) or (c) holds. 

- Suppose that $B$ satisfies (a) at $a$. Let $Y$ be a connected component of $B$ and $\mu_Y\le\nu$, $T_Y$ as stated in item (a). Observe that $\nu$ cannot be associated to a transversal saddle point in this case; otherwise the edge containing the point $a$ must be an element of the path $\Theta(\mu_Y)$ and also end g at $\nu$, contradicting Remark~\ref{rk:path-theta}). In particular, $Y$ is either a trace mark (if $e(a)=1$) or an angle mark (if $e(a)=2$). Denote $Y'=\psi(Y)$, a connected component of $B'$. If $\partial\RR(\sigma)^{out}\subset F_{\RR(\mu')}$, then $Y'$ accumulates to $a'=\sigma\cap F_{\RR(\mu')}$ (using Remark~\ref{rk:mark-from-w2-to-w1}). On the contrary, if $\partial\RR(\sigma)^{out}=L_{\RR(\mu')}$, by the observation above asserting that $Y$ is a trace or an angle mark and taking into account that $\sigma\subset Sk(D)$, we have that $Y'$ accumulates to a single point $a'\in|\Omega|\cap F_{\RR(\mu')}$
(using the non s-resonant condition along with Propositions ~\ref{pro:trace-to-angle} and \ref{pro:angle-to-angle}). In both cases, the hypothesis $Y\subset\Sat^+(T_Y)$ implies that $a'\in\sigma\subset|\Theta(\mu_Y)|$ and $Y'\subset\Sat^+(T_Y)$ too. This proves that $B'$ satisfies (a) at the point $a'$.

- If $B$ satisfies (b) at $a$, then $B$ is connected and it is itself a trace mark in $\nu$. We prove as in the case above that $B'$ accumulates to a point $a'\in F_{\RR(\mu')}\cap|\Omega|$. Since $B'\subset\Sat^+(B)$ by definition, $B'$ satisfies again (a) at the point $a'$.  

- If $B$ satisfies (c) at $a$, then $a$ belongs to the domain of $\psi$ and hence, being $a'=\psi(a)$ (a point in the leaf through $a$, thus contained in $D$), we have that $\overline{B'}\setminus B'=\{a'\}$. By the definition of distinguished fattening, we have that, if $\partial\RR(\sigma)^{out}\subset F_{\RR(\mu')}$, then $B'$ is contained either in an unfixed doorjamb or in the handrail of the door $\partial\RR(\sigma)^{out}$. Hence, $B'$ is in one of the situations (b) or (c) at $a'$. Otherwise, if $\partial\RR(\sigma)^{out}=L_{\RR(\mu')}$, then $B'$ is in the situation (c) at $a'$.

\vspace{.1cm}

{\bf Case II: $\partial\RR(\mu)^{out}=F_{\RR(\mu)}$.} Denote $\PPP=\partial\RR(\sigma)^{in}$, $\PPP'=\partial\RR(\sigma)^{out}$ and $\phi:\PPP\to\PPP'$ the analytic isomorphism given by the flow of $\LL$. Notice that $\overline{A}\cap\PPP$ is semi-analytic at any of its points, except possibly at the center $c_\PPP=\PPP\cap\sigma$, if $c_\PPP\in\overline{A}$. In this last case, $A$  satisfies the properties of item (a) at $c_{\PPP}$. In any case, the set $\overline{A}\cap\PPP$ has finitely many connected components and we have
$$
A'=\Sat^+_{\RR(\sigma)\cup\RR(\mu')}\big
(A\cap\PPP\big)\cap F_{\RR(\mu')}.
$$  We distinguish several subcases:

- If $\partial\RR(\mu')^{in}=F_{\RR(\mu')}$ and $\mu'$ is not associated to a transversal saddle, then $\PPP'\subset F_{\RR(\mu')}$ and $\phi$ sends $\overline{A}\cap\PPP$ to $\overline{A'}$. The claim follows easily given that $\phi$ sends $c_\PPP$ to the center of $\PPP'$ as well as the germs of the handrail and the unfixed doorjambs of $\PPP$ (at the points where they cut $D$) to the germs of the handrail and the unfixed doorjambs of $\PPP'$, respectively. 

- If $\partial\RR(\mu')^{in}=F_{\RR(\mu')}$ and $\mu'$ is associated to some $q\in S_{tr}$, then $\PPP'\ne\PPP'\cap F_{\RR(\mu')}$ and $\overline{A'}$ is the image by $\phi$ of $\overline{A}\cap\PPP\cap\phi^{-1}(\PPP'\cap F_{\RR(\mu')})$. We conclude as in the previous subcase provided that we can ensure that there are representants of the germs of $A$ and  $\Sat_{\RR(\sigma)}(W^2_q)\cap F_{\RR(\mu)}$ at $c_\PPP$ that are mutually disjoint. Since $A$ is contained either in the saturation of $W^2_p$, with some $p\in S_{tr}$ different from $q$, or in the saturation of a trace mark attached to a point not in the support of $\Omega$, this is guaranteed by Remark~\ref{rk:sat-empty-intersection}.

- If $\PPP'=L_{\RR(\mu')}$, then $A'=\Sat^+_{\RR(\mu')}\big
(\phi(A\cap\PPP)\big)\cap F_{\RR(\mu')}$ and the claim can be proved as in Case I. Notice that in this case, $\mu$ is not associated to a transversal saddle point (by the Morse-Smale condition). Moreover, if $\sigma$ is a trace edge, then $A\cap\PPP$ is not in the situation (a) at the point $c_\PPP$ (since $\sigma$ cannot belong to the path of some $\mu_1$ with $\mu_1<\mu$).
\end{proof}

\begin{remarks}\label{rk:stains}{\em
Going over the proof of Lemma~\ref{lm:stains}, one can point out the following:

\vspace{.2cm}

 (\ref{rk:stains}-1) If $A$ is a stain in $\RR(\nu)$ in situation (a), it is posible that the point  $a\in\overline{A}\cap(b(\RR(\nu))\cap|\Omega|)$ is not unique or that there is more than one local component $Y$ at such $a$. However, if $\Sat_{|\RR|}(A)$ does not cut any fixed mark of $\RR$ (except the one that generates $A$, in case $A$ is a fixed stain), then $A=Y$ has a unique local connected component at $a$ and $\mu_A$ and the trace mark $T_A$ stated in item (a) does not depend on $\nu$, but only on the generating mark of $A$. In fact, if $A$ is a fixed stain generated at some $\nu_p^+$, then $\mu_A$ is a local s-component immediately connected to $\nu^+_p$ and $T_A=\Sat(W^2_p)\cap F_{\RR(\mu_A)}$; if $A$ is a mobile stain generated at some $\nu_1$ with generating mark $J$, then either $\nu_1=\mu_A$ and $T_A=J$, or $\{\nu_1,\mu_A\}=\{\tilde{\alpha}(\G),
 \tilde{\omega}(\G)\}$ for  $\G$ a face of $\Omega$. In this case $T_A$ is the image of $J$ by the flow of $\LL$. 

  (\ref{rk:stains}-2) If $A$ is either a stain in $\RR(\nu)$ in situation (b) with $b(\RR(\nu))\subset D$ or in situation (c), then $A$ is a mobile stain generated at some $\nu'$ such that $\nu,\nu'$ are associated to points in the boundary subgraph $\partial\G$ of a face $\G$ of $\Omega$. In fact, in the situation (b), either $\nu=\nu'$ or $\{\nu,\nu'\}=\{\tilde{\alpha}(\G),\tilde{\omega}(\G)\}$.

  (\ref{rk:stains}-3) Note that Lemma~\ref{lm:stains} also holds for a pre-distinguished fattening $\RR$ if, being $\nu_1$ a local s-component where the stain is generated, we assume that $\RR$ is distinguished at any vertex belonging to any path of edges with extremities at $\nu$ and $\nu_1$. This slightly weakening of the statement hypothesis will be useful in the next paragraph.
}
\end{remarks}

According to Lemma~\ref{lm:stains}, the family of stains of a general distinguished fattening may present a really complicated behavior, especially those whose saturation cuts either some fixed mark or some free door doorjamb, different from the one that generates them. The following definition captures a situation where we have a nicer ``picture'' for the stains behavior.

\begin{definition}\label{def:disjoint stains}
	Let $\RR$ be a distinguished fattening. We say that $\RR$ has the property of {\em disjoint stains} ((DS) for short) if given $\nu\in\VV(\Omega)$ and given different stains $A,A'\in\Upsilon^\RR_{\RR(\nu)}\cup\Lambda^\RR_{\RR(\nu)}$, we have $A\cap A'=\emptyset$.
\end{definition}
\begin{remark}\label{rk:DS}
	{\em
 By means of Remark~\ref{rk:sat-empty-intersection}, in general (with or without the (DS) condition), two stains $A,A'$ of $\RR$ in $\RR(\nu)$ with different generating marks can only intersect in a finite number of points, except in the case where $A,A'$ are mobile stains with respective generating doorjambs $J,J'$ of free $\RR$-doors associated to the same face of $\Omega$, and such that their germs at the base points $b(J), b(J')$ are connected by the flow. In this exceptional case, if $A\cap A'$ is infinite, then $A,A'$ accumulate to the same point of $b(\RR(\nu))\cup D$ and they share a common germ at that point. Notice that, in this situation, the condition (DS) would imply that $A=A'$, even if $J\ne J'$. 
}
\end{remark}

Our purpose is to prove that to get good saturations it is enough to have the (DS) property. To this end, we introduce the following definition and state a very helpful lemma.

\begin{definition}\label{def:essential-as}
	Let $\RR$ be a distinguished fattening at some  $\nu\in\VV(\Omega)$. An interval $Y$ contained in $F_{\RR(\nu)}$ will be called a {\em well-positioned curve (relatively to $\RR$ at $\nu$)} if there exists an unfree $\RR$-door $\PPP$ in $F_{\RR(\nu)}$ such that
		$\overline{Y}$ is a closed non-trivial interval contained in $\PPP$ with:
		\begin{itemize}
			\item $Y$ does not intersect neither the base nor the doorjambs of $\PPP$.
			\item The set of extremities of $\overline{Y}$ is $\{c_\PPP,h(Y)\}$, where $c_\PPP$ is the center of $\PPP$ and $h(Y)$ is a point in the interior $\dot{h(\PPP)}$ of the handrail of $\PPP$.
			\item  $Y=\overline{Y}\setminus\{c_\PPP\}$ (thus $Y$ is a half-open interval containig the extremity $h(Y)$).
		\end{itemize}  We also say that $Y$ is well-positioned {\em inside} $\PPP$.
\end{definition}
\begin{lemma}\label{lm:getting-es-adm}
	Assume that $\RR$ is a distinguished fattening with the  (DS) property. Let $A$ be a stain of $\RR$ in $\RR(\mu)$ for some $\mu\in\VV(\Omega)$.
	\begin{enumerate}[(a)]
		\item If $A$ is a fixed stain generated at $p\in S_{tr}$, then $\mu$ belongs to the path $\Pi^1_p\cup\Pi^2_p$ (cf. Corollary~\ref{cor:sat-w2})) and $A$ is a  well-positioned curve.
		\item  If $A$ is a mobile stain generated at $\nu$, where $\nu<\mu$ and $F_{\RR(\nu)}\subset\partial\RR(\nu)^{in}$, then  $\mu$ belongs to the path $\Theta(\nu)$ (cf. Theorem~\ref{th:path of a trace mark}) and $A$ is a well-positioned curve.
		\item  If $A$ is a mobile stain generated at $\nu$ with  $\nu<\mu$ and $F_{\RR(\nu)}\subset\partial\RR(\nu)^{out}$, then $A$ is contained in an unfree $\RR$-door $\PPP_A$ in $\RR(\mu)$. Moreover, either $A$ is a well-positioned curve, in case it intersects the interior of $\PPP_A$, or $A$ is the unfixed doorjamb of $\PPP_A$, in case it cuts an unfixed doorjamb of $\PPP_A$, or $A$ is contained in $h(\PPP_A)$, in other case.
	\end{enumerate}
Items (b) and (c) also hold if we replace $\nu<\mu$ with $\nu>\mu$ and we interchange the superscripts ``in'' and ``out''. 
\end{lemma}
\begin{proof}
	Put $A=\Sat_{|\RR|}(J)\cap F_{\RR(\mu)}$, where $J$ is either the generating fixed mark at $p$, in case (a), or a free door doorjamb  in $\RR(\nu)$, in cases (b) or (c). 
	
	Let us prove (a) and (b) jointly, since the proof is the same. Notice that the hypothesis in (b) implies that $A=\Sat^+_{|\RR|}(J)\cap F_{\RR(\mu)}$. Without lost of generality, we suppose that in case (a) we have that $W^2_p$ is the unstable manifold at $p$.
	
	First, we show that if  $\nu'>\nu$ and $\PPP$ is a $\RR$-door (free or unfree) at $\nu'$, then the set
	$$
	B_{\PPP}=\{x\in\dot{J}\,:\,\ell_x^+\cap|\RR|\mbox{ cuts }\PPP\}
	$$
	is open and closed in $\dot{J}$, thus either empty or equal to $\dot{J}$. To see that $B_{\PPP}$ is closed, suppose that $x\in\dot{J}$ is the limit point of a sequence $\{x_n\}\subset B_{\PPP}$. If $y_n\in\ell^+_{x_n}\cap|\RR|\cap\PPP$, by compactness of $\PPP$, we may assume that there exists $y=\lim_ny_n\in\PPP$. By continuity of the flow, we have $y\in\ell^+_x$, which shows that $x\in B_\PPP$. On the other side, to show that $B_{\PPP}$ is open, let us take $x\in B_{\PPP}$ and put $\ell^+_x\cap\PPP=\{y\}$. Now, given that points in an interval of $\ell^+_x$ to the right of $x$ are of type i-i ($\dot{J}$ is contained in the interior of $\RR$ in the case (a), and by the hypothesis  $F_{\RR(\nu)}\subset\partial\RR(\nu)^{in}$ in the case (b)), we have that the point $y$ must be of type i-i or i-e relatively to $|\RR|$. This is consequence of the fact that,  by the (DS) property, the type can not switch to i-t along the open segment between $x$ and $y$ in $\ell^+_x$ (recall that, by Proposition~\ref{pro:trans-distinguished}, the points of type i-t are contained in doorjambs of free $\RR$-doors). We conclude that $y$ belongs to $\PPP\setminus\partial\PPP$, an open subset of $F_{\RR(\nu')}$. This proves that $x$ is an interior point of $B_{\PPP}$ using again the continuity of the flow. Notice that we have shown that $\Sat^+_{|\RR|}(\dot{J})\cap\PPP\subset(\PPP\setminus\partial\PPP)$.
	
	The proof of statements (a) and (b) can be deduced now from Theorem~\ref{th:path of a trace mark} as follows. In case (b), we have that $J$ is a trace mark in $\nu$ at its base point $b(J)$. In case (a), the positive saturation of $J$ produces, in turn, a trace mark in an immediate successor of $p$ in $\Pi^1_p\cup\Pi^2_p$ (cf. proof of Corollary~\ref{cor:sat-w2}). In both cases, for $\nu'\in\VV(\Omega)$, Theorem~\ref{th:path of a trace mark} guarantees that  there is a subinterval $\tilde{J}\subset\dot{J}$ with extremity $b(J)$ satisfying $\Sat^+_{|\RR|}(\tilde{J})\cap F_{\RR(\nu')}\ne\emptyset$ if, and only if, $\nu'$ belongs to the path $\Theta(\nu)$, in case (b), or to the path $\Pi^1_p\cup\Pi^2_p$, in case (a). Moreover, in this situation, $\Sat^+_{|\RR|}(\tilde{J})\cap F_{\RR(\nu')}$ is an interval accumulating at the center of an $\RR$-unfree door $\PPP_{\nu'}$ at $\nu'$ and, by the property just proved, we must have $B_{\PPP_{\nu'}}=\dot{J}$. In this way, we put $\tilde{J}=\dot{J}$ and  $\Sat^+_{|\RR|}(\dot{J})\cap F_{\RR(\nu')}$ is an interval contained in $\PPP_{\nu'}\setminus\partial\PPP_{\nu'}$. Applied to $\nu'=\mu$, we deduce that $A$ is a well positioned curve inside $\PPP_\mu$: the extremity of $\overline{A}$ different from the center of $\PPP_\mu$ cannot belong to an unfixed doorjamb of $\PPP_\mu$ (by means of the (DS) property since this doorjamb is itself a mobile stain), nor to a fixed doorjamb contained in $D$. This proves (a) and (b).
	
	Let us prove (c). First, we show that $A$ is contained in the union of unfree $\RR$-doors in $\RR(\mu)$. Reasoning by contradiction, suppose that there exists some $x\in J$ such that $\ell^+_x\cap|\RR|$ cuts the fence $F_{\RR(\mu)}$ at some $y\in A$ which does not belong to any unfree $\RR$-door. Using that the union of unfree $\RR$-doors at $\mu$ is closed in $F_{\RR(\mu)}$, we may assume that $x\in\dot{J}$. By hypothesis, the point $x$ is of type i-t relatively to $|\RR|$, and all the points in an open subinterval of $\ell^+_x$ to the right of $x$ are of type t-t. Also, the point $y$ can not be of type t-t, because $y$ does not belong to any unfree door (the set of points of type t-t in $F_{\RR(\mu)}$ is equal to the union of the handrails of unfree doors, by Proposition~\ref{pro:trans-distinguished})). Consequently, there is a first point $z\ne x$ in the segment of the leaf $\ell^+_x$ between $x$ and $y$ where the type has switched to t-i or to t-e. Let us see that this is not possible.

	Suppose that the point $z$ is of type t-i. Then $z$ is in $\dot{J}'$, where $J'$ is an unfixed doorjamb of a free $\RR$-door at some $\nu'\in\VV(\Omega)$. Moreover, $z\ne y$ since $J'$ is also a doorjamb of an unfree $\RR$-door at $\nu'$. Hence $\nu<\nu'<\mu$. On the other hand, necessarily $F_{\RR(\nu')}\subset\partial\RR(\nu')^{in}$, thus we are in the situation of item (b) for the mobile stain $\Sat_{|\RR|}(J'\setminus\{b(J')\})\cap F_{\RR(\mu)}$ generated at $\nu'$. As a consequence, the leaf $\ell^+_{z}\cap|\RR|$ can only intersect $F_{\RR(\mu)}$ inside an unfree door. This contradicts the existence of $y$. 
	
Assume now that the point $z$ is of type t-e. This implies that $z=y$ and that $\ell^+_x\cap|\RR|$ is just the segment from $x$ to $y$ in the leaf $\ell_x$. Notice that $y$ is an interior point (not an extremity) of the handrail $h(\DD)$ of a free $\RR$-door $\DD$ at $\mu$. In other words, $x$ belongs to the set
$$
	K=\{a\in\dot{J}\,:\,\ell^+_a\cap|\RR|\mbox{ cuts }\dot{h(\DD)}\},
$$
	which is open in $\dot{J}$ (using similar arguments as those we have used for the set $B_\PPP$ in items (a), (b) above).
	If $K=\dot{J}$ then the base point $b(J)$ of $J$ is an accumulation point of $K$ and hence the positive leaf $\ell^+_{b(J)}$ either ends in a singular point (placed between $\nu$ and $\mu$) or it 	
	cuts $h(\DD)$. In the first case we would find some $x'\in\dot{J}$ whose positive $\ell^+_{x'}$ has points of type t-i, contrary to what we have proved above. In the second case we have also a contradiction: if $\nu$ is a local s-component associated to a transversal saddle $q\in S_{tr}$, then $b(J)$ is in a fixed mark at $q$ and then $h(\DD)$ (hence $\DD$) will cut the saturation of a fixed stain, i.e., some fixed stain, against statement (a); if $\nu$ is not associated to a transversal saddle, then $b(J)\in D$, so that $\ell_{b(J)}\subset D$, whereas $h(\DD)\cap D=\emptyset$.
	On the other hand, if $K$ is a proper subset of $\dot{J}$, there is some $w\in\dot{J}\setminus K$ in the frontier of $K$ in $\dot{J}$. By the flow continuity, the leaf $\ell^+_w$ cuts the handrail $h(\DD)$ in an extremity point $w'$, since $w$ is not in $K$. Hence, $w'$ belongs to the intersection of $A$ with a doorjamb of $\DD$, which gives also a contradiction with the (DS) hypothesis.
	
	Thus, we can assure that $A\subset\bigcup_{\PPP\in Q}\PPP$, where $Q$ is the family of unfree $\RR$-doors at $\mu$. Let us finish proving that there is one single $\PPP_A\in Q$ such that $A\subset\PPP_A$ and that the last sentence in (c) holds. We have different cases.
	
	- If there exists some $\PPP\in Q$ and some $y\in A\cap int(\PPP)$ with $y\in\ell^+_x$, $x\in \dot{J}$, then, being $x$ of type i-t and $y$ of type i-i with respect to $|\RR|$, there must be some $z$ of type t-i in the piece of  leaf between $x$ and $y$. That is, $z$ belongs to a doorjamb $J'$ of some free $\RR$-door in $\RR(\nu')$ for some $\nu'\in\VV(\Omega)$ with $\nu<\nu'<\mu$. By the (DS) property, we have that $\Sat_{|\RR|}(J\setminus b(J))\cap F_{\RR(\nu')}=J'\setminus b(J')$ and hence $A$ is a mobile stain in $\RR(\mu)$ generated by $J'$. We conclude that $A$ is well-positioned inside $\PPP$ by item (b). Put $\PPP_A:=\PPP$ and we conclude (c) in this case. 
	
	- If $A$ intersects the unfixed doorjamb $L$ of some $\PPP\in Q$, being $L$ also a stain in $\RR(\mu)$, then $A$ coincides with $L\setminus b(L)$, by the (DS) property. We conclude again by putting $\PPP_A:=\PPP$.
	
	- In the remaining case, we have $A\subset\bigcup_{\PPP\in Q}\dot{h(\PPP)}$. By using similar arguments as above, for each $\PPP\in Q$, the set $\{x\in\dot{J}\,:\,\ell^+_x\cap|\RR|\mbox{ cuts }h(\PPP)\}$ is open and closed in $\dot{J}$. We deduce that there is a single $\PPP\in Q$ such that $A\cap\PPP\ne\emptyset$ and, for this $\PPP(=:\PPP_A)$ we have $A\subset\dot{h(\PPP_A)}$, concluding also (c) in this case. 
\end{proof}

In light of the Lemma~\ref{lm:getting-es-adm} we prove that, to get  good saturations, it is enough to have the (DS) property.

\begin{proposition}\label{pro:DS implies GS}
	Let $\RR$ be a distinguished fattening. If $\RR$ has the property of disjoint stains, then $\RR$ has good saturations. 	
\end{proposition}
\begin{proof}
To prove we have the condition (gsfm), suppose that there are fixed marks $I_p,I_q$ at different transversal saddle points $p,q\in S_{tr}$, and that there exists some $x\in\Sat_{|\RR|}(I_p)\cap\Sat_{|\RR|}(I_q)$. Being $x$ in $|\RR|$, the $|\RR|$-leaf through $x$ cuts the c-nbhd $\RR(\nu)$ for some $\nu\in\VV(\Omega)$, then it cuts also the fence $F_{\RR(\nu)}$. Hence, the fixed stains at $\nu$ with generating marks $I_p$ and $I_q$ share some point, in contradiction with the (DS) property.

The condition (gsfmfd) is a direct consequence of item (a) in Lemma~\ref{lm:getting-es-adm}, which asserts, in particular, that a fixed stain cannot cut a free $\RR$-door. 

To finish, let us prove  that $\RR$ satisfies the condition (gsfd). Suppose, by contradiction, that there are free $\RR$-doors $\DD \subset F_{\RR(\nu)}$ and $\DD'\subset F_{\RR(\nu')}$ not associated to the same face of $\Omega$ together with some $x\in\DD$ such that $\ell_x\cap|\RR|$ cuts $\DD'$ at some point $z$. Assume, for instance, that $\nu'>\nu$, so the orientation of $\LL$ in the leaf $\ell_x$ goes from $x$ to $z$. We have several possibilities:

$\bullet$  $\DD\subset\partial\RR(\nu)^{in}$ and $\DD'\subset\partial\RR(\nu')^{in}$. We have that $x$ cannot belong to any of the doorjambs of $\DD$, by Lemma~\ref{lm:getting-es-adm}, (b). Also, 	
	$z$ cannot belong to any of the doorjambs of $\DD'$; otherwise, $x$ would be in an unfree $\RR$-door by the last sentence in Lemma~\ref{lm:getting-es-adm} concerning item (c) with $\nu:=\nu'$ and $\mu:=\nu$  (imcompatible with the fact that $x$ does not belong to a doorjamb of $\DD$). As a consequence, $z$ is of type e-i or e-t relatively to $|\RR|$, so that the negative leaf $\ell^-_z$ scapes from $|\RR|$, contradicting the existence of $z$.
	
	$\bullet$  $\DD\subset\partial\RR(\nu)^{in}$ and $\DD'\subset\partial\RR(\nu')^{out}$. As above, $x$ does not belong to any of the doorjambs of $\DD$ and $z$ does not belong to any of the doorjambs of $\DD'$. Using similar arguments as in the proof of Lemma~\ref{lm:getting-es-adm}, one can see that, if $x\in\dot{h(\DD)}$ (resp. if $x\in(\DD\setminus\partial\DD)$), then $\ell^+_{x'}\cap|\RR|$ cuts $\DD'$ for any $x'\in\dot{h(\DD)}$ (resp. for any $x'\in(\DD\setminus\partial\DD)$). By continuity, this last property must be also true for points in the closure of $\dot{h(\DD)}$ (resp. $\DD\setminus\partial \DD$), thus for some points in the doorjambs of $\DD$. This is a contradiction with Lemma~\ref{lm:getting-es-adm}, (b).  
	
	$\bullet$ $\DD \subset\partial\RR(\nu)^{out}$ and $\DD'\subset\partial\RR(\nu')^{in}$. The point $x$ belongs to one of the doorjambs of $\DD$; otherwise, $x$ would be of type t-e or i-e relatively to $|\RR|$ and the positive leaf $\ell^+_x$ would scape from $|\RR|$. By the same reason, $z$ belongs to one of the doorjambs of $\DD'$. Notice, moreover, that $x\in D$ if, and only if, $z\in D$. This is the forbiden situation in which $\DD$ and $\DD'$ are associated to the same face of $\Omega$. The case where $x$ and $z$ belong to doorjambs of the respective free doors $\DD$ and $\DD'$, but $x,z\not\in D$, gives also a contradiction with the (DS) property. 
	
	$\bullet$ $\DD\subset\partial\RR(\nu)^{out}$ and $\DD'\subset\partial\RR(\nu')^{out}$. As in the precedent case, $x$ belongs to one of the doorjambs of $\DD$ and hence, by Lemma~\ref{lm:getting-es-adm}, (c), $z$ belongs to some unfree $\RR$-door of $F_{\RR(\nu')}$, thus in one of the doorjambs of $\DD'$. This gives again a contradiction with the last sentence of Lemma~\ref{lm:getting-es-adm} concerning (b) with $\nu:=\nu'$, $\mu:=\nu$. 
\end{proof}

\subsection{Proof of the good saturations property}
The following result along with Proposition~\ref{pro:DS implies GS}, gives the proof of Theorem~\ref{th:good-sat}.

\vspace{.2cm}
\noindent {\bf Theorem~\ref{th:good-sat}'}.
{\em Assume that $\mathcal{M}=(M,\LL,D)$ is not s-resonant and of Morse-Smale type. Given a distinguished fattening $\KK$ over $\Omega$, there exists a distinguished refinement $\wt{\KK}<\KK$ that satisfies the (DS) property.
}

The strategy is to obtain the refinement $\wt{\KK}$ after applying Proposition~\ref{pro:getting-distinguished} to the initial fattening $\KK$ by imposing convenient systems of entrances at the graph sources, i.e., the repeller $D$-node points.
The systems of entrances will be chosen so that the stains at those points
are disjoint, and so that this property is susceptible to be propagated ``along'' the graph.
To do this, we introduce an auxiliar definition that highlights
the intermediate steps in the recursive procedure towards the (DS) property.

\begin{definition}\label{def:qds}
	Let $\RR$ be a predistinguished fattening of $\Omega$ and let $\nu\in\VV(G)$. Assume that $\RR$ is distinguished at the vertex corresponding to $\nu$.
	\begin{itemize}
	\item[(DS)] We say that $\RR$ has the {\em Property $(DS)^{\le\nu}$} if any pair of distinct stains $A,A'\in\Upsilon^\RR_{\RR(\nu)}\cup\Lambda^{\RR(\le)}_{\RR(\nu)}$ are mutually disjoint.
	\item[(qDS)] We say that $\RR$ has the {\em Property $(qDS)^{>\nu}$} if, given  $A\in
	\Upsilon^\RR_{\RR(\nu)} \cup \Lambda^{\RR(>)}_{\RR(\nu)}$, we have that:
	
	- If $J$ is a doorjamb of a free $\RR$-door at $\nu$, then either $A\cap J=\emptyset$ or the germs of $\overline{A}$ and $J$ at $b(J)$ coincide.
	
	 - If $\overline{A}\cap b(\RR(\nu))\ne\emptyset$ then any connected component $Y$ of $A$ satisfies also $\overline{Y}\cap b(\RR(\nu))\ne\emptyset$ and, either $Y$ does not intersect the interior of an unfree $\RR$-door or $Y$ is a well-positioned curve on it. Moreover, in this last situation, if  $A'$ is another stain $A'$ of $\RR$ in $\RR(\nu)$ and $Y\cap A'\ne\emptyset$, then $Y\subset A'$.
\end{itemize}
\end{definition}

\strut

{\bf Proof of Theorem~\ref{th:good-sat}'.}

We may assume that $\KK$, hence any pre-distinguished refinement of $\KK$, already satisfies the property of good saturations for fixed marks (gsfm) (cf. Remark \ref{rm:siempre gsfm}). Thus, any two different fixed stains of a refinement of $\KK$ do not intersect.

Consider now refinements $\CC_q<\KK(\nu_q)$ for any  $q\in N^r$ in such a way that, being
$$
\AAA_q=
\bigcup_{A\in\Lambda^\RR_{\CC_q}\cup\Upsilon^\RR_{\CC_q}}
\,A
$$
the union of all stains of $\RR$ in $\CC_q$, then it holds: 
\begin{enumerate}[(i)]
\item Any connected component $Y$ of $\AAA_q$ is an interval whose closure is a closed interval with extremities $b(Y),h(Y)$, where $b(Y)\in b(\CC_q)$ and $h(Y)\in h(F_{\CC_q})$. 
\item If $Y,Y'$ are two connected components of $\AAA_q$ then either the intersection $\overline{Y}\cap\overline{Y'}$ is empty or it contains just the common extremity $b(Y)=b(Y')$. 
\end{enumerate}
Such refinements $\CC_q$ (see Figure \ref{Fig:PruebaDS_1}) can be obtained by means of Lemma~\ref{lm:stains}, to guarantee item (i), and Remark~\ref{rk:sat-empty-intersection} to get item (ii).

\begin{figure}[h]
	\begin{center}
		\includegraphics[scale=0.60]{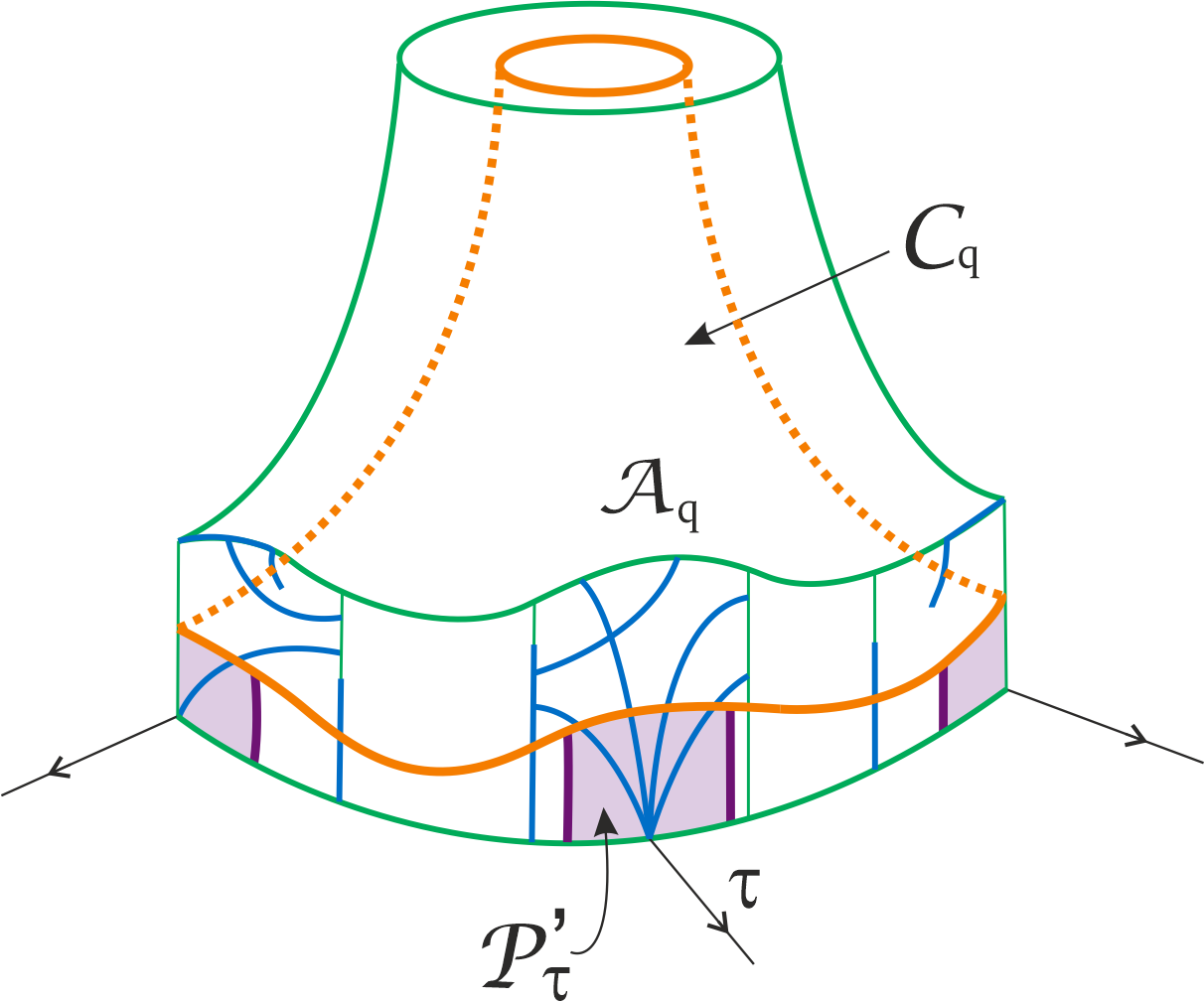}
	\end{center}
	\caption{Refinement of $\KK(\nu_q)$ to get the (DS) property.}
	\label{Fig:PruebaDS_1}
\end{figure}

The distinguished refinement $\wt{\KK}$ required in Theorem~\ref{th:good-sat}' is obtained by applying Proposition~\ref{pro:getting-distinguished} to the whole graph $G=\Omega$ together with the collection $\{\EE_\nu\}_{\nu\in\VV(\Omega)}$ of systems of entrances of  $\KK$, where $\EE_\nu=\emptyset$, if $\nu\not\in N^r$, and $\EE_{\nu_q}=\{L_{\CC_q}\}$, if $q\in N^r$ (thus $\EE_{\nu_q}$ is a complete system of entrances of $\KK(\nu_q)$ and the corresponding refinement obtained from it is precisely $\CC_q$). 

 As we know, if we apply Proposition~\ref{pro:getting-distinguished}, we can get many possible resulting refinements as application. To get one of them, say $\wt{\KK}$, having the (DS) property, we need to control the process according to what was discussed in Scholium~\ref{sch-prop-distinguished}. To be more precise, if $l=l(\Omega)$ is the length of the graph $\Omega$, the refinement $\wt{\KK}$ will be the last of a sequence of pre-distinguished refinements 
\begin{equation}\label{eq:sequence-dist-2}
\KK=\NN_{l+1}>\NN_l>\cdots>\NN_1>\NN_0=\wt{\KK},
\end{equation}
satisfying the conditions in (\ref{eq:prop-scholium}). Theorem~\ref{th:good-sat}' will be a consequence of the following claim.

\begin{quote}
{\bf Claim.-}
 Refinements in (\ref{eq:sequence-dist-2}) can be performed so that, given $j$, the fattening $\NN_j$ safisfies Property $(DS)^{\le\nu}$, for any  $\nu\in\VV(\Omega)\setminus\VV(\Omega^{j-1})$, and Property $(qDS)^{>\mu}$, for any $\mu\in\VV(\Omega^j)\setminus\VV(\Omega^{j-1})$.
\end{quote}

The claim proof goes by inverse recursion for $j=l,l-1,...,0$. Although a priori it seems that only condition $(DS)^{\le\nu}$ counts to get Theorem~\ref{th:good-sat}', we are led to consider also condition $(qDS)$ for the recurrence to work, as we will see next.

\strut

Let us start by constructing $\NN_l$. Notice that, if $\nu\in\VV(\Omega)\setminus\VV(\Omega^{l-1})$, then $\nu=\nu_q$ is the unique local s-component associated to a $D$-repeller point $q\in N^r$. For any such $q$, we have fixed the value $\NN_l(\nu_q):=\CC_{q}$, a refinement of $\KK(\nu_q)$. Let $\tau$ be an edge starting at $q$ and let $\PPP_\tau=\partial\KK(\tau)^{in}$ be the unfree $\KK$-door at $q$ with center at $c_\tau$, where $\{c_\tau\}=F_{\KK(\nu_q)}\cap\tau$. Taking into account the properties imposed to $\CC_q$, we may consider a refinement $\TT_\tau<\KK(\tau)$ defined by setting its inner part equal to
$\partial\TT_\tau^{in}=\PPP'_\tau$, where $\PPP'_\tau \subset \text{Int}_{F_{\KK(\nu_q)}}(\PPP_\tau)\cap F_{\CC_q}$ is a door in $F_{\CC_q}$  (in particular, the doorjambs of $\PPP'_\tau$ do not intersect the ones of $\PPP_\tau$) with center at $c_\tau$ and so that, if $Y$ is a connected component of $\AAA_q$, then $Y\cap \PPP'_\tau\ne\emptyset$ if, and only if, $Y$ has one extremity at $c_\tau$ and, in this case, $\overline{Y}$ is contained in $\PPP'_\tau$ and does not intersect the doorjambs of $\PPP'_\tau$ (see Figure \ref{Fig:PruebaDS_1} ).

Now, let $\NN_l$ be the fattening over $\Omega$ defined by $\NN_l(\nu_q):=\CC_q$ and $\NN_l(\tau):=\TT_\tau$, for any $q\in V(\Omega)\setminus V(\Omega^{l-1})$ and any $\tau\in\alpha^{-1}(q)$, and by putting $\NN_l=\KK$ elsewhere. By construction, $\NN_l$ is a pre-distinguished refinement of $\KK$ and it is distinguished at any $q\in V(\Omega)\setminus V(\Omega^{l-1})$. Moreover, the fattening $\NN_l$ has the property $(DS)^{\le\nu_q}$ for any $q\in V(\Omega)\setminus V(\Omega^{l-1})$ (observe that $\Upsilon^{\NN_l}_{\CC_q}=\Upsilon^{\KK}_{\CC_q}$ and that $\Lambda^{\NN_l(\le q)}_{\CC_q}$ consists only on the family of subsets of the form $J\setminus\{b(J)\}$, where $J$ is an unfixed doorjamb of a door $\PPP'_\tau$ with $\tau\in\alpha^{-1}(q)$). Finally, we check that $\NN_l$ satisfies the property $(qDS)^{>\nu_q}$ for any $q\in V(\Omega)\setminus V(\Omega^{l-1})$. Let $A\in\Upsilon^{\NN_l}_{\NN_l(\nu_q)}\cup\Lambda^{\NN_l(>)}_{\NN_l(\nu_q)}$. In the case that $A$ is a new stain not existing for $\RR=\NN_{l+1}$ (that is, $A$ is mobile generated at an immediate succesor $\nu'$ of $\nu_q$, only in case that $\NN_l$ is distinguished also at the point corresponding to $\nu'$), we have that $A$ is contained in the boundary of a door of the form $\PPP'_\tau$ and its germ at $\overline{A}\cap b(\NN_l(\nu_q))$ coincides with that of a doorjamb of $\PPP'_\tau$. Otherwise, $A$ coincides with some element of $\Upsilon^{\RR}_{\NN_l(\nu_q)}\cup\Lambda^{\RR(>)}_{\NN_l(\nu_q)}$ and it is contained in the union of the doors $\PPP'_\tau$, where $\tau$ runs over $\alpha^{-1}(q)$. By construction, we show the required properties for $(qDS)^{>\nu_q}$ in both cases.

\vspace{.2cm}

Assume that for some $k\le l$ we have already constructed a subsequence $\NN_l>\NN_{l-1}>\cdots>\NN_{k}$ of (\ref{eq:sequence-dist-2}) such that any of its terms satisfies the claim. We build the next fattening $\NN_{k-1}$ in three steps:
\begin{enumerate}
	\item We determine certain refinements $\CC_\mu<\NN_k(\mu)$ and $\TT_\tau<\NN_k(\tau)$, for any $\mu\in\VV(\Omega^{k-1})\setminus\VV(\Omega^{k-2})$ and any $\tau\in\tilde{\alpha}^{-1}(\mu)$.
	\item We define $\NN_{k-1}$ by setting $\NN_{k-1}(\mu):=\CC_\mu$ and $\NN_{k-1}(\tau):=\TT_\tau$ for $\mu,\tau$ as in $1)$, and $\NN_{k-1}:=\NN_k$ elsewhere.
	\item We check that $\NN_{k-1}$ is distinguished at any $\mu\in\VV(\Omega^{k-1})\setminus\VV(\Omega^{k-2})$ and satisfies the claim for index $j:=k-1$.
\end{enumerate} 
A pertinent remark should be made about the last step: in order to check that $\NN_{k-1}$ satisfies the claim, we just need to show that it satisfies properties $(DS)^{\le\mu}$ and $(qDS)^{>\mu}$ for any $\mu\in\VV(\Omega^{k-1})\setminus\VV(\Omega^{k-2})$. In turn, by virtue of Remark~\ref{rk:stains1}, these properties depend exclusively on the modifications of $\NN_k$ performed at $\mu$, that is, the values $\CC_{\mu}$ and $\TT_{\tau}$ for $\tau\in\tilde{\alpha}^{-1}(\mu)$. In other words, we check  properties $(DS)^{\le\mu}$ and $(qDS)^{>\mu}$ while we build the refinements $\CC_\mu$, $\TT_\tau$, even if we are not completely done with the first step yet.

\strut

We distinguish several situations according to the possibilities for $\mu\in\VV(\Omega^{k-1})\setminus\VV(\Omega^{k-2})$.

\vspace{.2cm}

(s0) {\em The local s-component $\mu=\nu_q$ is associated to a $D$-repeller $q\in N^r$.} In this case, we just put $\CC_\mu=\CC_q$, where this last c-nbhd is the refinement of $\KK(\nu_q)$ ($=\NN_k(\nu_q)$) fixed at the beginning of the proof. For any $\tau\in\tilde{\alpha}^{-1}(\mu)$, we determine a refinement $\TT_\tau<\NN_k(\tau)$ by setting $\partial\TT_\tau^{in}=\PPP'_\tau$, where $\PPP'_\tau$ is a door in $F_{\CC_q}$ with the same properties as stated for $j=l$. 
Analogously to that case, we conclude that $\NN_{k-1}$ is distinguished at $\mu$ and satisfies properties $(DS)^{\le\mu}$ and $(qDS)^{>\mu}$.

\vspace{.3cm}

(s1) {\em The lid $L_{\NN_k(\mu)}$ of $\NN_k(\mu)$ coincides with $\partial\NN_k(\mu)^{in}$, and we are not in the situation (s0).} In this case, there is a single edge $\sigma$ ending at $\mu$. Besides, by the Morse-Smale condition, there is a unique local s-component $\nu$ where $\sigma$ starts. Since $\NN_k$ is pre-distinguished, $\partial\NN_k(\sigma)^{out}$ is a predoor of $L_{\NN_k(\mu)}$ and we put $\CC_\mu$ to be the refinement of $\NN_k(\mu)$ for which $L_{\CC_\mu}=\partial\NN_k(\sigma)^{out}$ (in fact, this is the unique way to assign a value $\NN_{k-1}(\mu):=\CC_\mu$ in order that $\NN_{k-1}$ in the sequence (\ref{eq:sequence-dist-2}) is distinguished at $\mu$). Also, since $\NN_k$ is distinguished at $\nu$, the  inner part  $\PPP_\sigma=\partial\NN_{k}(\sigma)^{in}$ is either equal to the lid $L_{\NN_k(\nu)}$ or to an unfree $\NN_k$-door in the fence $F_{\NN_k(\nu)}$.
We show first that the following property is satisfied:
\begin{quote}
($P_1$) Given $A\in\Upsilon^{\NN_k}_{\CC_\mu}\cup\Lambda^{\NN_k}_{\CC_\mu}$ accumulating in the base $b(\CC_\mu)$, then $A$ is a union of finitely many mutually disjoint intervals whose closure has one extremity in $b(\CC_\mu)$ and the other one in  $h(F_{\CC_\mu})$.
\end{quote}
To do that, notice that $A_1=\Sat_{|\NN_k|}(A)\cap F_{\NN_k(\nu)}\in\Upsilon^{\NN_k}_{\NN_k(\nu)}\cup\Lambda^{\NN_k}_{\NN_k(\nu)}$. In case $\PPP_\sigma\subset F_{\NN_k(\nu)}$ we have that $A_1$ is contained in $\PPP_\sigma$, an unfree $\NN_k$-door, and accumulates to its center. Also, if $\{J,J'\}$ is the family of unfixed doorjambs of $\PPP_\sigma$ (it may hold $J=J'$), then $h(\PPP_\sigma)\cup J\cup J'$ is sent homeomorphically to $h(F_{\CC_\mu})$ by the flow. Property $(DS)^{\le\nu}$ for $\NN_k$ along Lemma~\ref{lm:getting-es-adm}, in case $A_1\in\Lambda^{\NN_k(\le)}_{\NN_k(\nu)}$, or Property $(qAS)^{>\nu}$ for $\NN_k$, in case $A_1\in\Upsilon^{\NN_k}_{\NN_k(\nu)}\cup\Lambda^{\NN_k(>)}_{\NN_k(\nu)}$, make that any connected component of $A_1$ is a well-positioned curve inside $\PPP_\sigma$ and hence property ($P_1$) follows for $A$. In the case $\PPP_\sigma=L_{\NN_k(\nu)}$, we have that $A_1$ accumulates to some point of $b(\NN_k(\nu))$. If $A_1\in\Lambda^{\NN_k(\le)}_{\NN_k(\nu)}$ then, again by $(DS)^{\le\nu}$, each connected component of $A_1$ is an interval with one extremity in $b(\NN_k(\nu))$ and the other one in $h(F_{\NN_k(\nu)})$. If $A_1\in\Upsilon^{\NN_k}_{\NN_k(\nu)}\cup\Lambda^{\NN_k(>)}_{\NN_k(\nu)}$, the same happens, in virtue of $(qDS)^{>\nu}$ and taking into account that in this situation each connected component of $A_1$ accumulates to some point in $b(F_{\NN_k(\nu)})\cap|\Omega|$. The required property ($P_1$) for $A$ follows as above since $h(F_{\NN_k(\nu)})$ is sent homeomorphically to $h(F_{\CC_\mu})$ in this case.

As a consequence of property ($P_1$), any doorjamb of  a free $\KK$-door in $F_{\RR(\mu)}$ cuts only once the handrail $h(F_{\CC_\mu})$. Hence, if $\tau$ is an edge starting at $\mu$, the set $\PPP_\tau:=\partial\KK(\tau)^{in}\cap F_{\CC_\mu}$ is a door in $F_{\CC_\mu}$. More generally, any connected component of a stain $B$ of $\NN_k$ in $\CC_\mu$ that accumulates to $b(\CC_\mu)\setminus|\Omega|$ also cuts once the handrail $h(F_{\CC_\mu})$. Notice that such stain $B$ different from a doorjamb of some door $\PPP_\tau$ may appear only if $\mu$ is associated to a transversal saddle and $B$ is generated at some $\mu'>\mu$. In particular, such $B$ cannot cut any other stain $A'$ of $\NN_k$ in $\CC_\mu$ (if such intersection occurs, $\Sat_{|\NN_k|}(B)\cap F_{\NN_k(\nu)}$ would be contained in a stain in $\NN_k(\nu)$ generated at some $\nu'>\nu$, contradicting the hypothesis $(qDS)^{>\nu}$ for $\NN_k$.  

Taking these remarks into account, we can consider a refined door $\PPP'_\tau=\PPP'_\tau(\mu)<\PPP_\tau$ (see Figure \ref{Fig:PruebaDS_2}) such that:

\begin{figure}[h]
	\begin{center}
		\includegraphics[scale=0.65]{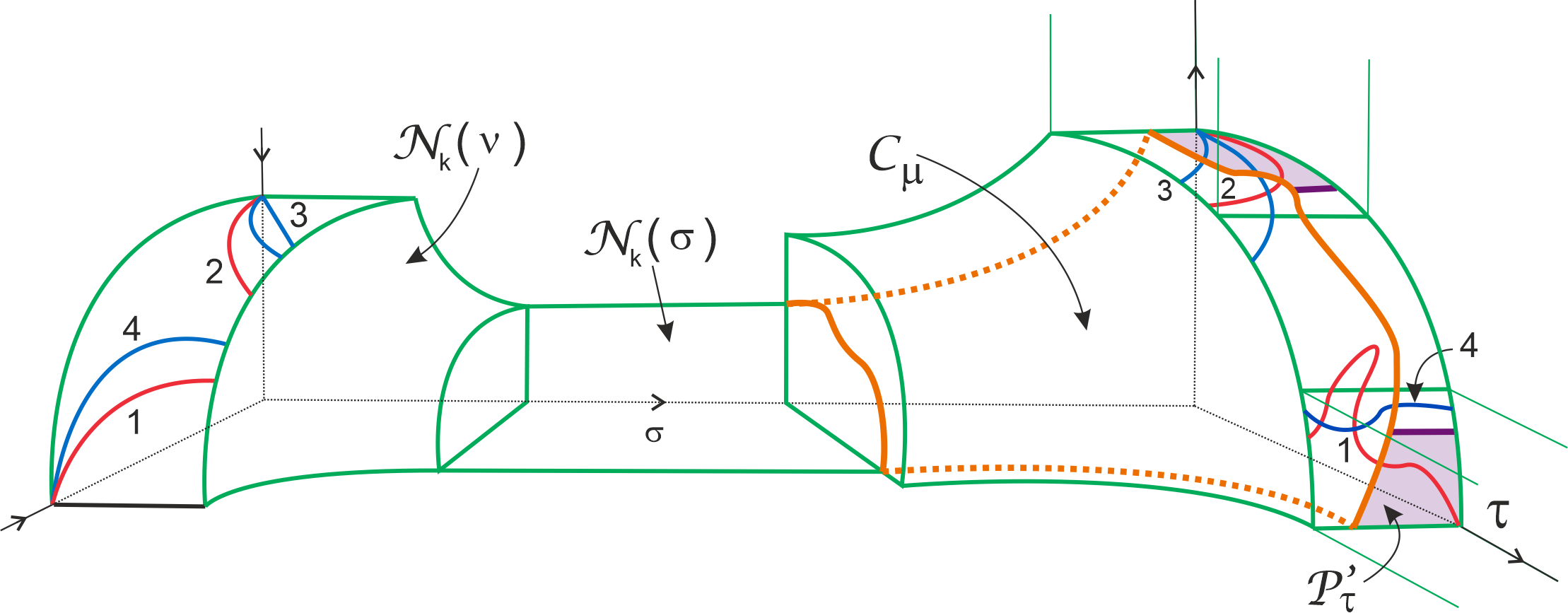}
	\end{center}
	\caption{Refinement of unfree doors in $\mathcal{N}_k(\mu)$ in situation (s1).}
	\label{Fig:PruebaDS_2}
\end{figure}

- For any $A\in\Upsilon^{\NN_k}_{\CC_\mu}\cup
\Lambda^{\NN_k}_{\CC_\mu}$ that accumulates in $b(\CC_\mu)$, either $\overline{A}\cap \PPP'_\tau=\emptyset$ for any $\tau\in\tilde{\alpha}^{-1}(\mu)$ or, if $\overline{A}\cap \PPP'_\tau\ne\emptyset$, then $A\cap\PPP'_\tau$ is a finite union of mutually disjoint well-positioned curves inside $\PPP'_\tau$. In particular, $\PPP'_\tau$ does not cut any unfixed doorjamb of $\PPP_\tau$.

- The unfixed doorjamb of $\PPP'_\tau$ does not intersect any $A\in\Upsilon^{\NN_k}_{\CC_\mu}\cup
\Lambda^{\NN_k}_{\CC_\mu}$.

- If $\mu$ is associated to some $q\in S_{tr}$ (so that $\mu\in\{\mu^+_q,\mu^-_q\}$), then the two doors $\PPP'_\tau(\mu^+_q),\PPP'_\tau(\mu^-_q)$ have equal bases (notice that both $\mu^+_q,\mu^-_q$ belong to $\VV(\Omega^{k-1})\setminus\VV(\Omega^{k-2})$ and are in the situation (s1), in this case).

\begin{remark}{\em
	It is worth noticing that if we do not assume property ($P_1$), or similar, we could not obtain a door $\PPP'_\tau$ with the above properties, as Figure 	\ref{Fig:PruebaDS_3} suggests. This important remark justifies the introduction of the $(qDS)$ condition, an essential tool in the proof of property ($P_1$).
}
\end{remark}

\begin{figure}[h]
	\begin{center}
		\includegraphics[scale=0.65]{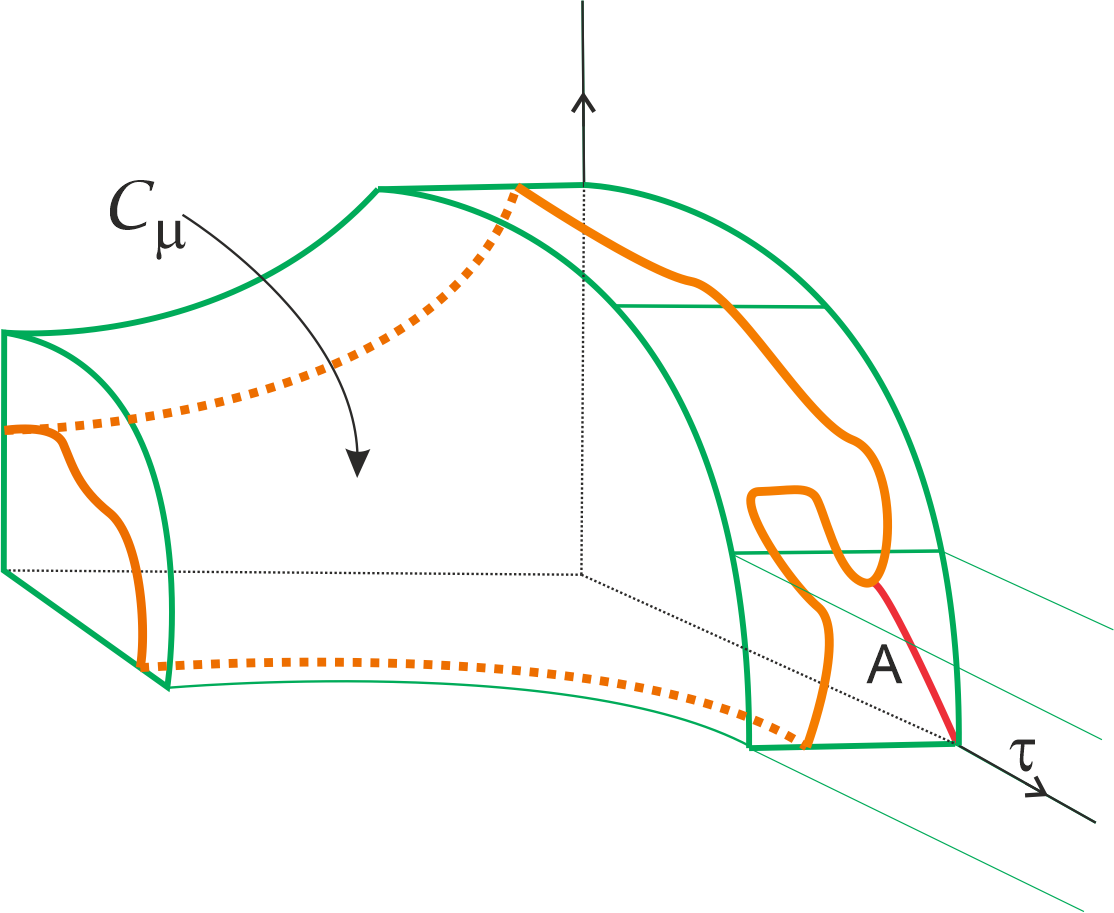}
	\end{center}
	\caption{Situation where the refinement of the unfree door associated to $\tau$ is not possible.}
	\label{Fig:PruebaDS_3}
\end{figure}
The refinement $\TT_\tau$ of the tube $\KK(\tau)$ to be considered for any $\tau\in\tilde{\alpha}^{-1}(\mu)$ is the one determined by setting
$\partial\TT_\tau^{in}=\PPP'_\tau$, in case that $\mu$ is associated to some vertex $q\not\in S_{tr}$, or the union of the two doors $\PPP'_\tau(\mu^\epsilon_q)$, $\epsilon=\pm$, if $\mu$ is associated to some $q\in S_{tr}$. In this way, since $\PPP'_\tau$ is a door in $F_{\CC_\mu}$ for any such $\tau$, we guarantee that the fattening $\NN_{k-1}$ 

Let us check that $\NN_{k-1}$ satisfies the property $(DS)^{\le\mu}$ as follows. Take $A,A'\in\Upsilon^{\NN_{k-1}}_{\CC_\mu}\cup\Lambda^{\NN_{k-1}(\le)}_{\CC_\mu}$ with $A\ne A'$. If both $A,A'$ are fixed doorjambs, we know already that they do not intersect, since $\NN_{k-1}$ satisfies (gsfm). If some of them is a doorjamb of a free $\NN_{k-1}$-door (i.e., an unfixed doorjamb of $\PPP'_\tau$ for some edge $\tau$ starting at $\mu$), then by construction $A\cap A'=\emptyset$. Otherwise, the sets 
$\Sat_{|\NN_{k-1}|}(A)\cap F_{\NN_k(\nu)}$ and $\Sat_{|\NN_{k-1}|}(A')\cap F_{\NN_k(\nu)}$ are contained (equal in this case) in respective elements $A_1,A'_1$ of  $\Upsilon^{\NN_{k}}_{\NN_k(\nu)}\cup\Lambda^{\NN_{k}(\le)}_{\NN_k(\nu)}$ with $A_1\ne A'_1$ (unless one of them, say $A$, is fixed and generated at $\nu$, in which case $A_1$ is a fixed mark). In any case, we have $A_1\cap A'_1=\emptyset$ since $\NN_k$ satisfies property $(DS)^{\le\nu}$. We deduce that $A\cap A'=\emptyset$. 

Finally, we check that $\NN_{k-1}$ satisfies property $(qDS)^{>\mu}$. Let  $A\in\Upsilon^{\NN_{k-1}}_{\NN_{k-1}(\mu)}\cup\Lambda^{\NN_{k-1}(>)}_{\NN_{k-1}(\mu)}$. If $A$ is a ``new'' stain of $\NN_{k-1}$ appearing from the modification of some of the tubes $\TT_\tau$ (that is, $A=\Sat_{|\NN_{k-1}|}(J\setminus b(J))\cup F_{\CC_\mu}$, where $J$ is an unfixed doorjamb of $\partial\TT_{\tau}^{out}\cap F_{\NN_{k-1}(\mu')}$ for $\mu'\in \tilde{\omega}(\tau)$, and in the case where this last disc is a door in $F_{\NN_{k-1}(\mu')}$), then $A$ coincides with a doorjamb $J$ of some $\PPP'_\tau$, as germs at $b(J)$. Moreover, $A$ is an interval and does not intersect the interior of any unfree $\NN_{k-1}$-door. This proves $(qDS)^{>\mu}$ for this situation. Otherwise, $A$ is contained in some $B\in \Upsilon^{\NN_k}_{\CC_\mu}\cup\Lambda^{\NN_k}_{\CC_\mu}$ (recall that $\NN_{k-1}(\mu)=\CC_\mu$). By construction $B$ does not cut any doorjamb of a door $\PPP'_\tau$, which shows the first part of property $(qDS)^{>\mu}$ for $A$. Assume now that $\overline{A}\cap b(\NN_{k-1}(\mu))\ne\emptyset$. 
 By property ($P_1$), we have that each connected component $Y$ of $B$ is an interval with extremities $b(Y)\in b(\NN_{k-1}(\mu))$ and $h(Y)\in h(F_{\NN_{k-1}(\mu)})$. Moreover, $b(Y)\in|\Omega|$ for any such $Y$, so that, $Y$ is a well-positioned curve inside one of the doors $\PPP'_\tau$ (by construction, $\overline{Y}$ does not intersect the doorjambs of such $\PPP'_\tau$). But, being $Y\subset\PPP'_\tau$, we have that 
$Y$ is also a connected component of $A$. This proves the first part of $(qDS)^{>\mu}$. Now, if $A'$ is another stain of $\NN_{k-1}$ in $\NN_{k-1}(\mu)$ that intersects $Y$, by construction, $A'$ cannot be one of the unfixed doorjambs of the doors $\PPP'_\tau$ (the ``new'' stains in $\NN_{k-1}$) and then $A'$ is contained in some stain $B'$ of $\NN_k$ in $\CC_\mu$. Take $\Sat_{|\NN_{k-1}|}(A)\cap F_{\NN_k(\nu)}$ and $\Sat_{|\NN_{k-1}|}(B')\cap F_{\NN_k(\nu)}$, which are contained in respective stains $A_1,B'_1$ of $\NN_k$ in $\NN_k(\nu)$. Moreover, $Y_1=\Sat_{|\NN_{k-1}|}(Y)\cap F_{\NN_k(\nu)}$ is connected and contained in $A_1$. Since $Y\cap A'\ne\emptyset$ by hypothesis, we have $Y_1\cap B'_1\ne\emptyset$. Using that $\NN_k$ satisfies $(qDS)^{>\nu}$, we must have $Y_1\subset B'_1$, and hence also $Y\subset A'$. This proves the second part of $(qDS)^{>\mu}$ for $\NN_{k-1}$.

\vspace{.3cm}

(s2) {\em The inner part $\partial\NN_k(\mu)^{in}$ coincides with the fence $F_{\NN_k(\mu)}$.} Let $\tilde{\omega}^{-1}(\mu)=\{\sigma_1,\ldots,\sigma_r\}$ be the family of edges ending at $\mu$ and  $c_i$ the point defined by $\sigma_i\cap F_{\NN_k(\mu)}=\{c_i\}$, for $i=1,\ldots,r$. Since $\NN_k$ is pre-distinguished, the intersection
$\DD_i:=\partial\NN_k(\sigma_i)^{out}\cap F_{\NN_k(\mu)}$ is a predoor in $F_{\NN_k(\mu)}$, and we have $\DD_i\cap\DD_j=\emptyset$, if $i\ne j$. Using that $\NN_k$ is distinguished at any $p\in V(\Omega^k)\setminus V(\Omega^{k-1})$, it holds one of the following possibilities for each $T_{i}=\partial\NN_k(\sigma_i)^{in}$ (see Figure \ref{Fig:PruebaDS_4}):
\begin{enumerate}[(a)]
	\item $\tilde{\alpha}(\sigma_i)=\{\nu_i\}$ and $T_{i}$ is the lid of $\NN_k(\nu_i)$;
	\item $\tilde{\alpha}(\sigma_i)=\{\nu_i\}$ and $T_{i}$ is a door in $F_{\NN_k(\nu_i)}$;
	\item $\tilde{\alpha}(\sigma_i)=\{\nu^+_i,\nu^-_i\}$ and $T_{i}$ is the union of two doors $T^+_{i}, T^-_{i}$, in $F_{\NN_k(\nu^+_i)}$ and $F_{\NN_k(\nu^-_i)}$, respectively, with equal base.
\end{enumerate}

    To cover all cases, we write $\nu_i^{\epsilon}$, $T_i^\epsilon$, etc., with $\epsilon\in\{+,-\}$  if $\#(\tilde{\alpha}(\sigma_i))=2$, or $\epsilon$ does not appear otherwise.

\begin{figure}[h]
	\begin{center}
		\includegraphics[scale=0.60]{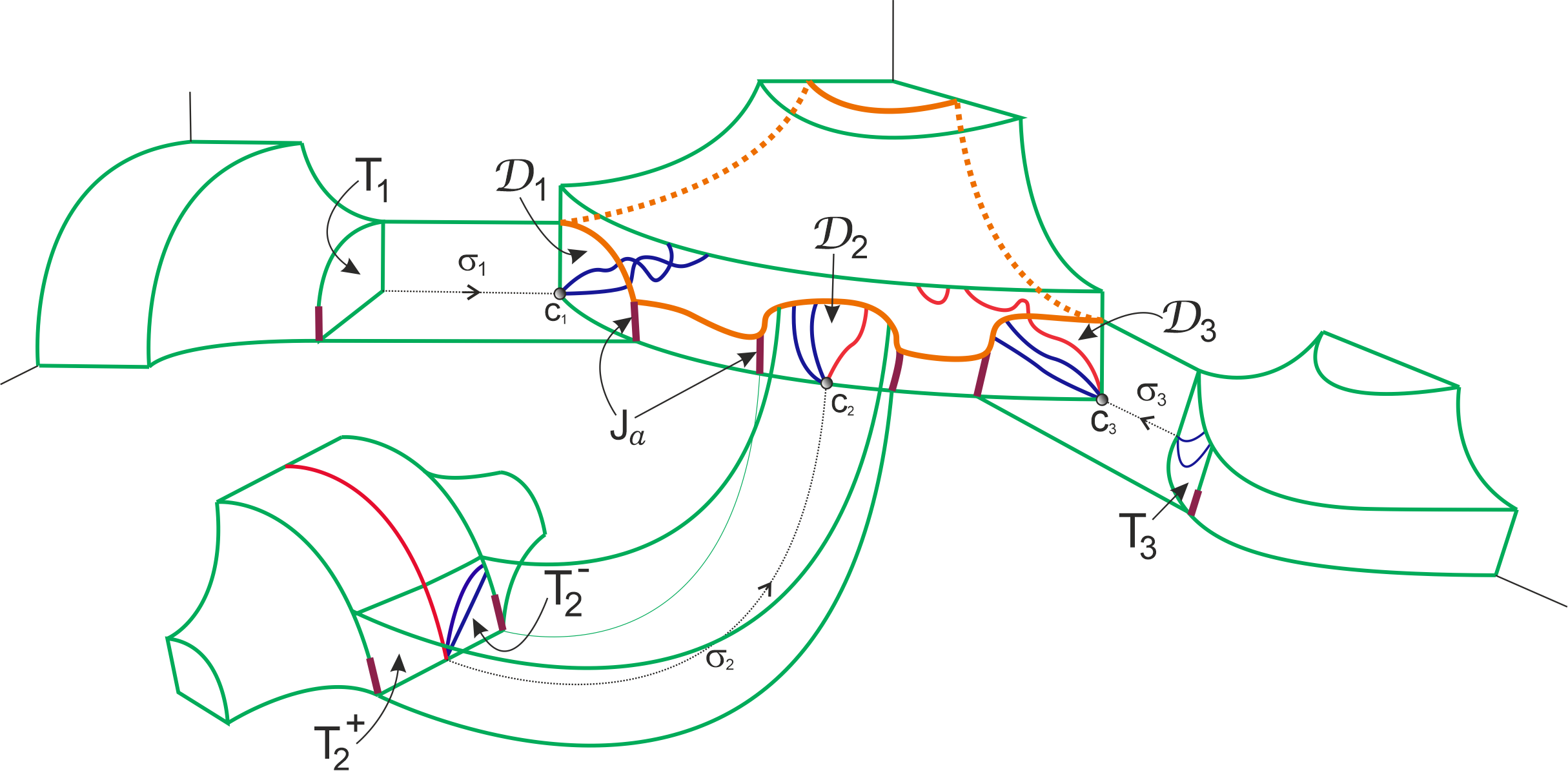}
	\end{center}
	\caption{Situation (s2) with $\mu$ corresponding to a tangential saddle.}
	\label{Fig:PruebaDS_4}
\end{figure}

Let us prove first that the following property is satisfied:
\begin{quote}
($P_2$) Let $A\in\Lambda^{\NN_k(>)}_{\NN_k(\mu)}\cup
\Upsilon^{\NN_k}_{\NN_k(\mu)}$ and let $Y$ be a connected component of $A$ accumulating in $b(\NN_k(\mu))$. Then we have $\overline{Y}\cap b(\NN_k(\mu))=\{c_i\}$ for some $i\in\{1,\ldots,r\}$ and $Y\cap\DD_i$ is an
interval with a extremity at $c_i$ and the other one in $Fr(\DD_{i})\setminus(b(\NN_k(\mu))\cup D)$. 
\end{quote}
{\em Proof of ($P2$)}.- If $A$ is fixed and generated at some $\nu^\epsilon_i$ (case (c)), then ($P_2$) is true (with a single component $Y$) provided that $\NN_k$ is pre-distinguished. If $A$ is fixed but generated at some $\nu<\mu$ with $\nu\ne\nu_j^\epsilon$ for any $j,\epsilon$, then $\Sat_{|\NN_k|}(A)\cap F_{\NN_k(\nu_i^\epsilon)}\ne\emptyset$ for at least one index $i$ (and some $\epsilon\in\{+,-\}$ if $\DD_i$ is in case (c)). Any such non-empty intersection is contained in a fixed stain $A_1$ with the same generating mark as $A$. In the ligth of Lemma~\ref{lm:getting-es-adm} (notice that we have the condition $(DS)^{\le\nu^\epsilon_k}$ by hypothesis), the indices $i,\epsilon$ are unique and $A_1$ is a well-positioned curve relatively to an unfree $\NN_k$-door $\PPP_{A_1}$ in $F_{\NN_k(\nu_i^\epsilon)}$. More precisely, $\PPP_{A_1}=T^\epsilon_i$ in cases (b) or (c), whereas $int(\PPP_{A_1})$ and $h(\PPP_{A_1})$ are sent by the flow, respectively, into $int(T_i)$ and $\partial T_i$, in case (a). We deduce that $A=\Sat_{|\NN_k|}(A_1)\cap F_{\NN_k(\mu)}$ is connected and contained in $\DD_i$ and that ($P_2$) is satisfied.
Finally, suppose that $A$ is either mobile or fixed but generated at some $\nu>\mu$ and let $Y$ be a connected component of $A$ accumulating in $b(\NN_{k}(\mu))$. Applying Lemma~\ref{lm:stains} and Theorem~\ref{th:path of a trace mark} we get $\overline{Y}\cap b(\NN_k(\mu))=\{c_i\}$, for some $i\in\{1,...,r\}$. Thus, $\Sat_{|\NN_k|}(A\cap\DD_i)\cap F_{\NN_k(\nu^\epsilon_i)}$ is non-empty (for some $\epsilon\in\{+,-\}$, if $\DD_i$ is in case (c)) and hence contained in some stain $A_1\in\Upsilon^{\NN_k(>)}_{\NN_k(\nu^\epsilon_i)}
\cup\Lambda^{\NN_k(>)}_{\NN_k(\nu^\epsilon_i)}$. Moreover, as above, $A_1$ intersects the interior of some unfree $\NN_k$-door $\PPP_{A_1}$ and accumulates in $b(\NN_k(\nu_i^\epsilon))$. Using that $\NN_k$ already satisfies $(qDS)^{>\nu_i^\epsilon}$, we get that $A_1\cap\PPP_{A_1}$ is a union of mutually disjoint well-positioned curves inside $\PPP_{A_1}$, one of them containing necessarily $\Sat_{|\NN_k|}(Y\cap\DD_i)\cap\PPP_{A_1}$. Since $\PPP_{A_1}$ is mapped by the flow into $\DD_i$, property $(P_2)$ holds for $Y$ as in the previous situation.

 \strut

Now, for each  $i\in\{1,\ldots,r\}$, we consider a framing of the pre-door $\DD_i$ by distinguishing two cases as follows:

(C1) {\em The s-component $\mu$ is not associated to a transversal saddle.-} (see Figure \ref{Fig:PruebaDS_4}) In this case, the base $b(\DD_i)$ has one or two extremities in $D\setminus|\Omega|$. For any such extremity $a$, we have to choose an unfixed doorjamb $J_a$ of $\DD_i$ with base point at $a$. Being $\Gamma_a$ the face of $\Omega$ containing $a$, we have that $a$ is an accumulation point of a mobile stain $A_a\in\Lambda^{\NN_k(<)}_{\NN_k(\mu)}$, generated at $\mu_1:=\tilde{\alpha}(\G_a)$, and that the germs of $\overline{A}_a$ and $\partial\DD_i$ at $a$ coincide. We claim that $A_a$ does not intersect any stain $B\in\Upsilon^{\NN_k}_{\NN_k(\mu)}\cup
\Lambda^{\NN_k}_{\NN_k(\mu)}$ unless the germs of $B$ and $A_a$ at $a$ are equal. To prove this claim, suppose otherwise that those germs are different and that $A_a\cap B\ne\emptyset$. Being $J_1$ the generating mark of $A_a$ (a doorjamb of a free $\NN_k$-door at $\mu_1$) and $l_1=l(\mu_1)$ the length of $\mu_1$, we would have that $\Sat_{|\NN_{l_1}|}(B)\cap F_{\NN_{l_1}(\mu_1)}$ is contained in some stain $B_1$ of $\NN_{l_1}$ in $\NN_{l_1}(\mu_1)$ that intersects $J_1$, but the germs of $B_1$ and $J_1$ at $b(J_1)$ do not coincide. This contradicts the recurrence hypothesis that $\NN_{l_1}$ satisfies $(qDS)^{>\mu_1}$ (notice that $l_1\ge k$).
We notice also that $A_a\subset\partial\DD_i$ (using for instance Lemma~\ref{lm:getting-es-adm} and taking into account that $\NN_k$ satisfies property $(DS)^{\le\eta}$ for any $\eta\in\VV(\Omega)\setminus\VV(\Omega^{k-1})$). 

If $A_a$ does not intersect  $h(F_{\NN_k(\mu)})$ then we set $J_a:=A_a$. But it is possible that $A_a$ cuts this handrail, in which case, the whole stain $A_a$ could not be chosen as the doorjamb $J_a$ of a new free $\NN_{k-1}$-door. Instead, we take $J_a$ to be a closed interval inside $A_a$ with set of extremities equal to $\{a,b\}$ and which does not intersect $h(F_{\NN_k(\mu)})$. Without any further modification, this would be an inconvenience for having the property $(DS)^{\le\mu}$ (since $A_a$ and $J_a$ would be two different stains that intersect). We propose to reconsider a new sequence of refinements $\RR=\NN'_{l+1}>\NN'_l>\cdots>\NN'_k$ such that $|\NN'_j|=\overline{|\NN_j|\setminus\Sat_{|\RR|}(Q)}$ where $Q$ is a closed disc inside $\DD_i$ whose boundary is the union of two intervals  (see Picture \ref{Fig:PruebaDS_4_2}): one is given by $L:=A_a\setminus J_a$ and the other one is an interval $L'$ going from $b$ to the  extremity of $L$ different from $b$, entirely contained in $int(\DD_i)$ (except for its extremities) and such that $Q$ cuts no stain of $\NN_j$ in $\RR(\mu)$ for any $j$, except $A_a$ (this is possible since $A_a$ cuts no other stain of $\NN_k$).

\begin{figure}[h]
	\begin{center}
		\includegraphics[scale=0.60]{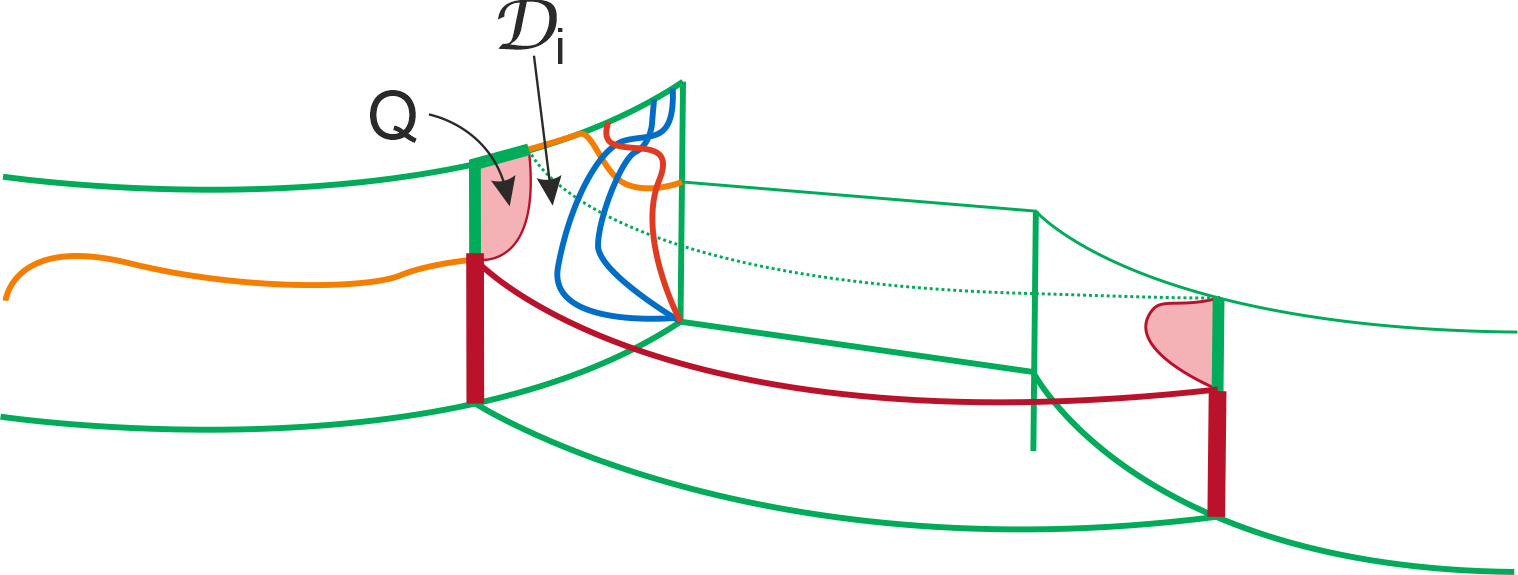}
	\end{center}
	\caption{Situation where we have to remove a disc $Q$ from $\mathcal{D}_i$.}
	\label{Fig:PruebaDS_4_2}
\end{figure}

 The fattenings in the resulting sequence  $(\NN'_j)_{j=l+1,...,k}$ are hence pre-distinguished, they satisfy the properties (\ref{eq:prop-scholium}) in the schema proposed in the Scolium and moreover, they satisfy the corresponding recurrence properties $(DS)^{\le\nu}$ and $(qDS)^{>\nu}$ stated in the Claim. With this modification, $J_a$ is a mobile stain of $\NN'_k$ in $\NN'_k(\mu)$, with generating mark at $\mu_1$ equal to a subinterval $J'_1\subset J_1$ (concretely $J_1\setminus\Sat_{|\NN_k|}(L)$), whereas the other boundary interval $\Sat_{|\NN_k|}(L')$ becomes part of the handrail of the unfree $\NN'_k$-door having $J'_1$ as a doorjamb. In other words, we rename again $\NN_k:=\NN'_k$ and we assume that we are in the precedent case; that is, that $J_a:=A_a$ is equal to the mobile stain with generating mark $J_1$. 

(C2) {\em The s-component $\mu$ is associated to a transversal saddle $q$.-} In this case, we have $r=2$ and, by the Morse-Smale condition, $\tilde{\alpha}(\sigma_i)=\{\nu_i\}$ and $F_{\NN_k(\nu_i)}=\partial\NN_k(\nu_i)^{out}$, for $i=1,2$. Put  $Y_i=\Sat_{|\NN_k|}(W^2_q\cap F_{\NN_k(\mu)}\setminus\sigma_i)\cap F_{\NN_k(\nu_i)}$, $i=1,2$, the fixed stain of $\NN_k$ at $\nu_i$ generated at $q$. They are connected and well-positioned inside the door $T_{i}$, because $\NN_k$ is pre-distinguished. In particular, $Y_i$ cuts the handrail $h(T_{i})$ in a single point $b_i$ which does not belong to any of the doorjambs of $T_{i}$ (see Figure \ref{Fig:PruebaDS_5}). 

\begin{figure}[h]
	\begin{center}
		\includegraphics[scale=0.65]{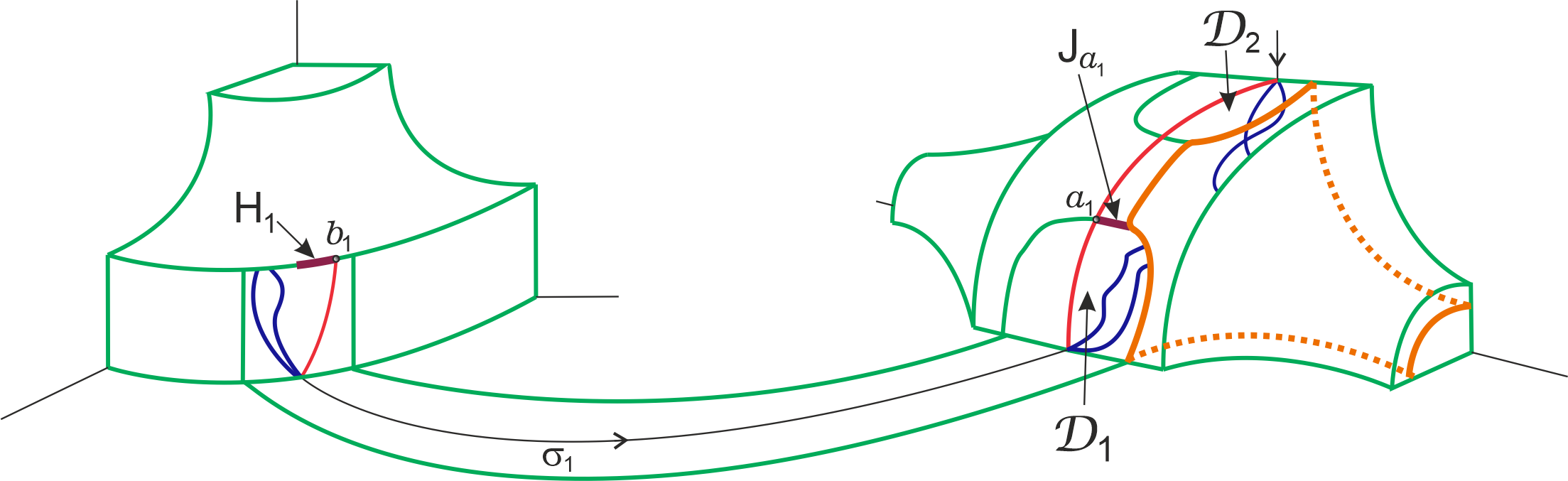}
	\end{center}
	\caption{Situation (s2) with $\mu$ corresponding to a transversal saddle.}
	\label{Fig:PruebaDS_5}
\end{figure}

Notice that the flow sends $b_i$ to one extremity $a_i$ of $b(\DD_i)$ ($\subset W^2_q$), the one that does not belong to $D$, and sends one of the two local connected components of $h(T_{i})\setminus\{b_i\}$ at $b_i$, say $H_i$, into $\partial\DD_i\setminus W^2_q$. We choose the unfixed doorjamb $J_{a_i}$ of $\DD_i$ at $a_i$ to be the image of $\overline{H_i}$ with the requirement that $H_i$ is sufficiently small so that it does not intersect any stain (fixed or mobile) of $\NN_k$ in $\NN_k(\nu_i)$. This is possible thanks to the fact that $\NN_k$ satisfies  $(DS)^{\le\nu_i}$ and $(qDS)^{>\nu_i}$. As a consequence, $J_{a_i}$ does not intersect any 
stain $B\in\Upsilon^{\NN_k}_{\NN_k(\mu)}\cup
\Lambda^{\NN_k(\ne)}_{\NN_k(\mu)}$.

In both cases (C1), (C2), we obtain a system of entrances $\EE_\mu=\{\DD_1,\ldots,\DD_r\}$ in $\NN_k(\mu)$. By Lemma~\ref{lm:cled} along with the choice of frammings for $\DD_i$ just made, we can extend $\EE_\mu$ to a complete system $\wt{\EE}_\mu$ in such a way that no new added door $\DD\in\wt{\EE}_\mu\setminus\EE_\mu$ intersects with any element of $\Upsilon^{\NN_k}_{\NN_k(\mu)}\cup
\Lambda^{\NN_k(\ne)}_{\NN_k(\mu)}$.
We set then $\CC_\mu$ to be the refinement of $\NN_k(\mu)$ given by the complete system $\wt{\EE}_\mu$. On the other hand, in this situation (s2), there is a unique edge $\tau$ starting at $\mu$. Thus, we consider the refinement $\TT_\tau$ of the tube $\NN_k(\tau)$ as the one determined by $\partial\TT_\tau^{in}=L_{\CC_\mu}$. In this way, we guarantee that the next fattening $\NN_{k-1}$ with values $\NN_{k-1}(\mu)=\CC_\mu$ and $\NN_{k-1}(\tau)=\TT_\tau$, will be distinguished at $\mu$.

Let us show that $\NN_{k-1}$ satisfies $(DS)^{\le\mu}$. Let $A,A'$ be two elements of  $\Upsilon^{\NN_{k-1}}_{\CC_\mu}
\cup\Lambda^{\NN_{k-1}(\le)}_{\CC_\mu}$ and suppose that they have non-empty intersection. By construction, both $A,A'$ are contained in the union of the $\NN_{k-1}$-doors $\DD_i$ and hence we can assume that $A\cap A'\cap\DD_i\ne\emptyset$, for some fixed index $i$. Hence the sets $\tilde{A}=\Sat_{|\NN_k|}(A)\cap F_{\NN_k(\nu_i^\epsilon)}$ and $\tilde{A}'=\Sat_{|\NN_k|}(A')\cap F_{\NN_k(\nu_i^\epsilon)}$ have non-empty intersection inside $T^\epsilon_i$, for a convenient choice of $\epsilon$ if necessary. These sets are contained in respective stains $A_1,A'_1$ of $\NN_k$ in $\NN_k(\nu_i^\epsilon)$ if both $A,A'$ are either fixed or generated at some local s-component strictly smaller than $\mu$. By the recurrence hypothesis that $\NN_k$ satisfies $(DS)^{\le\nu}$ for any $\nu<\mu$, we get that $A_1=A'_1$ which gives also $A=A'$. If for instance $A=J_a$ is one of the doorjambs of a new free $\NN_{k-1}$-door constructed above in case (C1), since $A=J_a$ is also a mobile stain generated at some $\mu_1<\mu$, as we have already discussed, $\tilde{A}$ is also a stain of $\NN_{k}$ in $F_{\NN_k(\nu^\epsilon_i)}$ with the same generating mark. We conclude as well $A=A'$. Finally, if for instance
 $A=J_{a_1}$ or $A=J_{a_2}$ in case (C2) above, then $A$ cannot cut any other stain of $\NN_k$ in $\NN_k(\mu)$. 


To finish, we show that $\NN_{k-1}$ satisfies the property $(qDS)^{>\mu}$. Consider a stain $A\in\Upsilon^{\NN_{k-1}}_{\NN_{k-1}(\mu)}
\cup\Lambda^{\NN_{k-1}(>)}_{\NN_{k-1}(\mu)}$. Then, $A$ is contained in a stain $B$ in $\Upsilon^{\NN_{k}}_{\NN_{k}(\mu)}
\cup\Lambda^{\NN_{k}(>)}_{\NN_{k}(\mu)}$ with the same generating mark as $A$. By construction, $B$ does not intersect any of the new doorjambs in $F_{\NN_{k-1}(\mu)}$ and hence so does $A$. This shows the first part of $(qDS)^{>\mu}$. Now, assume that $A$ accumulates in $b(\NN_{k-1}(\mu))$, necessarily only along the points $\{c_i\}$ by property $(P2)$. Let $Y$ be a connected component of $A$ and take an index $i$ such that $Y\subset\DD_i$. We would have that $\Sat_{|\NN_k|}(A)\cap F_{\NN_k(\nu^\epsilon_i)}$ is contained in some stain $A_1$ of $\NN_k$ in $\NN_k(\nu^\epsilon_i)$ which, moreover, accumulates to $b(\NN_k(\nu^\epsilon_i))$ and which contains the connected non-empty set $\tilde{Y}=\Sat_{|\NN_k|}(Y)\cap F_{\NN_k(\nu^\epsilon_i)}$. Using that $\NN_k$ satisfies $(qDS)^{>\nu^\epsilon_i}$, we have that $\tilde{Y}$ is contained in some connected component $Y_1$ of $A_1$ which is a well-positioned curve inside $T^{\epsilon}_i$. Since $Y_1\subset T^\epsilon_i$, this shows that $\tilde{Y}=Y_1$ and that $Y$ is a well-positioned curve inside $\DD_i$.
 Take now another stain $A'$ of $\NN_{k-1}$ in $\NN_{k-1}(\mu)$ with $Y\cap A'\ne\emptyset$. Then  $\Sat_{\NN_{k-1}}(A')\cap F_{\NN_k(\nu^\epsilon_i)}$  is contained in some stain $A'_1$ of $\NN_{k}$ in $\NN_{k}(\nu^\epsilon_i)$ such that $Y_1\cap A'_1\ne\emptyset$. Again, since $\NN_k$ satisfies $(qDS)^{>\nu_i^\epsilon}$ we must have that $Y_1\subset A'_1$ which shows as above that $Y\subset A'$. This gives the second part of $(qDS)^{>\mu}$.

\strut

Summarizing, once we have analyzied the different situations (s0), (s1), (s2) above, we have constructed a pre-distinguished refinement $\NN_{k-1}<\NN_k$ that is, in particular, distinguished at $\mu$ and satisfies $(DS)^{\le\mu}$ and $(qDS)^{>\mu}$, for any $\mu\in\VV(\Omega^{k-1})\setminus\VV(\Omega^{k-2})$. On the other hand, $\NN_{k-1}$ coincides with $\NN_k$ at any $\nu\in\VV(\Omega)$ with $l(\nu)\ne k-2$. Hence $\NN_{k-1}$ continues the sequence (\ref{eq:process-dist}) according to Scholium (\ref{eq:prop-scholium}) and, by the recurrence hypothesis, the claim to be proved holds for the value $k-1$. This ends the proof of Theorem~\ref{th:good-sat}'. 
$\hfill\square$.


\section{Extendable fattenings. Proof of the Main Theorem}\label{sec:extendable}

In this section, we gather all the results showed so far in order to give a proof of Theorem~\ref{th:main}. As mentioned in the introduction, fitting domains will be built from convenient fattenings. Apart from being distinguished and with the property of good saturations, we need an additional feature that permit to close up the free doors associated to faces of the graph.

\begin{definition}\label{def:extendable}
	Let $\KK$ be a distinguished fattening of $\Omega$ and let $U$ be a neighborhood of $D$ containing $|\KK|$. Given a face $\G$ of $\Omega$, we say that $\KK$ is {\em extendable on $\G$ inside $U$} if the following conditions hold:
	\begin{enumerate}[(i)]
		\item For any $x\in\DD^{out}_{\KK,\G}$, the positive $U$-leaf through $x$ cuts $\KK(\tilde{\omega}(\G))$ and the first intersection point, denoted by $\phi(x)$, belongs to $\DD^{in}_{\KK,\G}$.
		\item The map $\phi:\DD^{out}_{\KK,\G}\to\DD^{in}_{\KK,\G}$  is bijective (thus a homeomorphism) and preserves the respective doorjambs and handrails.
		\item Given $x\in\DD^{out}_{\KK,\G}$, if $\kappa_x$ denotes the piece of $U$-leaf from $x$ to $\phi(x)$, we have that $\kappa_x$ cuts $|\KK|$ if, and only if, $x$ belongs to one of the doorjambs of $\DD^{out}_{\KK,\G}$. In this case, $\kappa_x$ is contained in $Fr(|{\KK}\arrowvert_{\partial\G}|)$.
	\end{enumerate}
	The fattening $\KK$ is called {\em extendable inside $U$} if it is so on any face of the graph. It is called {\em extendable} if it is extendable inside $U$ for some $U$.
\end{definition}


Assume that $\KK$ is distinguished and extendable  inside some neighborhood $U$ of $D$. Hence, given a face $\G$ of $\Omega$, there is a flow-box $\KK_\G$, formed by the union of pieces $\kappa_x$ of $U$-leaves as defined in item (iii) above where $x$ runs over $\DD^{out}_{\KK,\G}$. Define the {\em extended support} of $\KK$ as the set
		$$
		\wh{\KK}:=|\KK|\cup\bigcup_{\scriptsize{\G\mbox{ a face }}}\KK_\G.
		$$
Notice that $\wh{\KK}$ does not depend on $U$, as long as $\KK$ is extendable inside $U$. 
In this situation, using Proposition~\ref{pro:trans-distinguished}, we can prove the following lemma that describes the type of any point in the frontier of $\wh{\KK}$. 

\begin{lemma}\label{lm:extended-support}
The extended support $\wh{\KK}$ is a compact semianalytic neighborhood of $D$. Moreover, $Fr(\wh{\KK})$ is a topological, piecewise smooth surface whose transversal frontier $Fr(\wh{\KK})^\pitchfork$ (cf. Section~\ref{sec:SHNR-foliations}) is the union of the following discs:
	\begin{enumerate}[(a)]
	\item The lids of the c-nbhds $\KK(\nu_q)$, where $q\in N$.
	\item The free $\KK$-doors of $F_{\KK(\nu_p^\epsilon)}$, where $p\in S_{tr}$ and $\epsilon\in\{+,-\}$.
\end{enumerate}
More precisely, if $H$ is one of the discs as in (a) or (b) then, always relatively to $\wh{\KK}$, $\mbox{int}(H)$ is formed of points of type i-e or e-i. Moreover,

- If $H=L_{\KK(\nu_q)}$ and $q$ is not a three dimensional saddle, then all the points of $\partial H$ are of type e-i or i-e.

- If $H=L_{\KK(\nu_q)}$ and $q$ is a three dimensional saddle, then all the points of $\partial H$ are of type e-t or t-e (as in case (ii) of Proposition~\ref{pro:trans-distinguished}).

- If $H$ is a free $\KK$-door then the type of the different points in $\partial H$ is exactly the one described in case (i) of Proposition~\ref{pro:trans-distinguished} taking $H$ in the role of $\DD$.
\end{lemma}	
\begin{proposition}\label{pro:extendable-after-gdsfd}
	Let $\mathcal{M}$ be a non s-resonant and of Morse-Smale type HAFVSD. Let $\KK$ be a distinguished fattening over $\Omega$ having the property (DS). Then there exists a refinement $\wt{\KK}$ of $\KK$ which is distinguished, extendable and has also the property (DS).
\end{proposition}


\begin{proof}
Given a face $\G$ of $\Omega$, let us denote  $\DD^{\alpha}_\G=\DD^{out}_{\KK,\G}$ and $\DD^\omega_\G=\DD^{in}_{\KK,\G}$ the respective free $\KK$-doors at $\tilde{\alpha}(\G)$ and $\tilde{\omega}(\G)$ corresponding to the face $\G$. Since $\KK$ is distinguished, 
we can consider a neighborhood $U$ of $D$ in $M$ and pre-doors $\wt{\DD}^{\alpha}_\G\subset\DD^{\alpha}_{\G}$ and $\wt{\DD}^{\omega}_\G\subset\DD^{\omega}_{\G}$ for any $\G$ such that:
\begin{enumerate}[(i)]
	\item For $u\in\{\alpha,\omega\}$, we have $b(\wt{\DD}^{u}_\G)=b(\DD^{u}_{\G})$.
	\item For any point $x\in\wt{\DD}^{\alpha}_\G$, the positive $U$-leaf starting at $x$ cuts a first time $\wt{\DD}^{\omega}_\G$ at a point $\phi_\G(x)$ so that $\phi_\G:\wt{\DD}^{\alpha}_\G\to\wt{\DD}^{\omega}_\G$ is a homeomorphism.
	\item For $u\in\{\alpha,\omega\}$, there is a framing of $\wt{\DD}^{u}_\G$ whose set of doorjambs $\{\wt{J}^u_{\G,1},\wt{J}^u_{\G,2}\}$ satisfies $\wt{J}^u_{\G,i}=\wt{\DD}^{u}_\G\cap J^u_{\G,i}$, for $i=1,2$, where $\{J^u_{\G,1},J^u_{\G,2}\}$ is the set of doorjambs of $\DD^{u}_{\G}$. Moreover, $\phi_\G$ maps the handrail and the doorjambs of $\wt{\DD}^{\alpha}_\G$ to the handrail and the doorjambs of $\wt{\DD}^{\omega}_\G$, respectively.
	\item If $x\in\mbox{int}(\wt{\DD}^\alpha_\G)$, then the $U$-leaf from $x$ to $\phi_\G(x)$ does not intersect $|\KK|$.
\end{enumerate}
In particular, if $u\in\{\alpha,\omega\}$, the framed predoor $\wt{\DD}^u_\G$ is compatible with any other free or unfree $\KK$-door at $\tilde{u}(\G)$ different from $\DD^u_{\G}$.

For any face $\G$ of $\Omega$, we consider the  subsets of $|\KK|$ (see Figure \ref{Fig:PruebaEXT_1}):

$$
E^\alpha_\G=\Sat_{|\KK|}^- \left(\rm{Int}_{F_{\KK(\tilde{\alpha}(\G))}}(\DD^{\alpha}_{\G}\setminus\wt{\DD}^\alpha_\G)\right), \
E^\omega_\G=\Sat_{|\KK|}^+
\left(\rm{Int}_{F_{\KK(\tilde{\omega}(\G))}}(\DD^{\omega}_{\G}\setminus\wt{\DD}^\omega_\G)
	\right), 
$$


\begin{figure}[h]
	\begin{center}
		\includegraphics[scale=0.60]{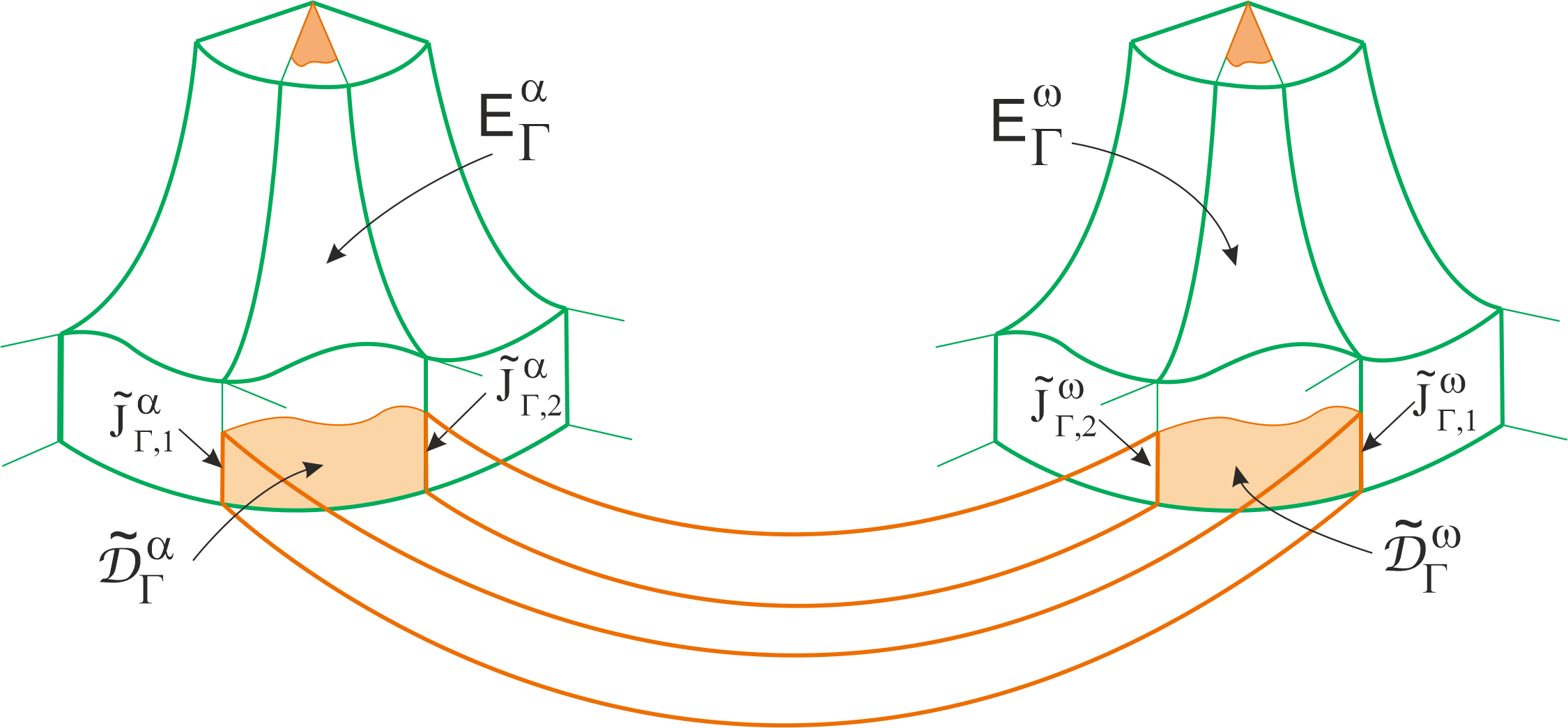}
	\end{center}
	\caption{Extending $\KK$ to the component $\Gamma$.}
	\label{Fig:PruebaEXT_1}
\end{figure}

\noindent and put $\mathbf{E}=\bigcup_{\G}\,E^\alpha_\G\cup E^\omega_\G$. Define the map
$$
\wt{\KK}:\VV(\Omega)\cup E(\Omega)\to\PPP(U)
$$
by setting $\wt{\KK}(\rho)=\KK(\rho)\setminus (\mathbf{E}\cap\KK(\rho))$, where $\rho$ is either a local s-component or an edge of the graph $\Omega$. We claim that $\wt{\KK}$ is a distinguished refinement of $\KK$ satisfying the required properties. To check this, let us point out the following observations, valid for any $\nu\in\VV(\Omega)$:
\begin{enumerate}[(a)]
	\item By means of Lemma~\ref{lm:getting-es-adm}, given $\G$ a face of $\Omega$, the set $E^\alpha_\G$ cuts $F_{\KK(\nu)}$ only if $\nu$ is in the path $\Theta(\tilde{\alpha}(\G))$ (cf. Theorem~\ref{th:path of a trace mark}). In this case, the set $E^\alpha_\G\cap F_{\KK(\nu)}$ is contained in the interior, with respect to $F_{\KK(\nu)}$, of an unfree $\KK$-door. In fact, $E^\alpha_\G\cap F_{\KK(\nu)}$ is contained in the ``wedge'' domain bounded by the two mobile stains $\Sat^-_{|\KK|}(J^\alpha_i)\cap F_{\KK(\nu)}$, $i=1,2$, both being well-positioned curves inside such a door. 
	\item The set $\wt{\KK}(\nu)$ is a c-nbhd at $\nu$ and also refinement of $\KK(\nu)$. Its fence is obtained by removing from $F_{\KK(\nu)}$ the subsets of the form $F_{\KK(\nu)}\cap E^\alpha_\G$, where $\G$ runs over the set of faces. If $E^\alpha_\G$ cuts the fence $F_\KK(\nu)$, the
	handrail $h(\wt{\KK}(\nu))$ differs from $h(\KK(\nu))$ precisely along the frontier of $F_{\KK(\nu)}\cap E^\alpha_\G$ inside $F_{\KK(\nu)}$, as depicted in Figure \ref{Fig:PruebaEXT_2}. In particular, $J^\alpha_{\G,i}\setminus\wt{J}^\alpha_{\G,i}$ is a segment of the handrail of $F_{\wt{\KK}(\tilde{\alpha}(\G))}$. 

\begin{figure}[h]
	\begin{center}
		\includegraphics[scale=0.60]{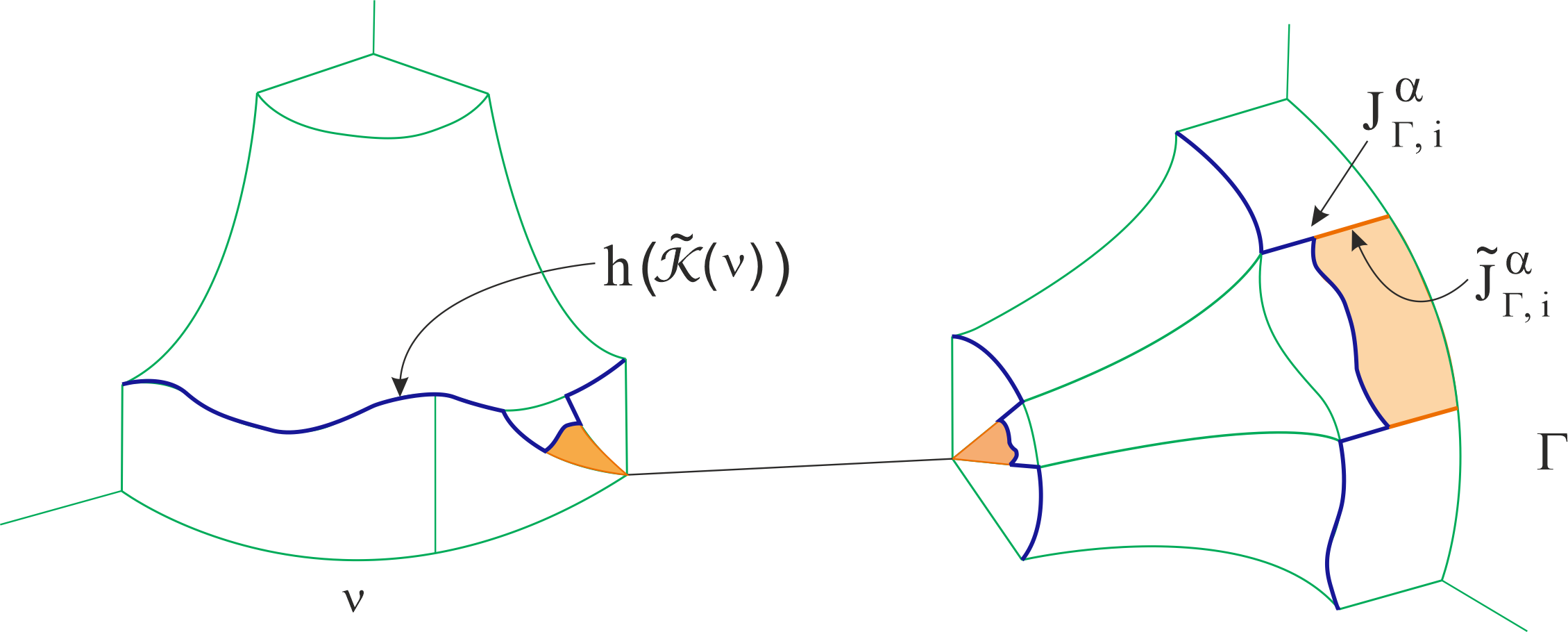}
	\end{center}
	\caption{The new handrail $h(\wt{\KK}(\nu))$.}
	\label{Fig:PruebaEXT_2}
\end{figure}
	\item Fixed stains in $\wt{\KK}$ coincide with those in the fattening $\KK$ thanks to property (DS). Moreover, if $\nu$ belongs to the path $\Theta(\tilde{\alpha}(\G))$ for some face $\G$ such that $\nu\ne\tilde{\alpha}(\G)$, in light of Lemma~\ref{lm:getting-es-adm}, the mobile stains of the form $\Sat^-_{|\wt{\KK}|}(\wt{J}^\alpha_{\G,i})\cap F_{\wt{\KK}(\nu)}\in\Lambda^{\wt{\KK}}_{\wt{\KK}(\nu)}$, for $i=1,2$, are contained in the corresponding elements of $\Lambda^{\KK}_{\KK(\nu)}$ which, in addition, are well-positioned inside an unfree $\KK$-door. In fact, those mobile stains of $\wt{\KK}$ are contained in an unfree $\wt{\KK}$-door at $\nu$ and do not intersect any doorjamb of a free $\wt{\KK}$-door.
	\item The same properties (a), (b), (c) above hold if we replace $\alpha$ with $\omega$ and $\Sat^-$ with $\Sat^+$.
\end{enumerate} 

 In sum, the fattening $\wt{\KK}$ is an extendable distinguished refinement of $\KK$. Moreover, $\wt{\KK}$ and $\KK$ have equal fixed stains,
 the free $\wt{\KK}$-doors and the free $\KK$-doors at local s-components associated to transversal saddle points coincide, and the free $\wt{\KK}$-doors associated to faces of the graph are contained in the corresponding free $\KK$-doors. Moreover, the doorjambs of these last free $\wt{\KK}$-doors are contained in the corresponding doorjambs of free $\KK$-doors and the doorjambs of the $\wt{\KK}$-doors $\DD^{\alpha}_{\G}$, $\DD^{\omega}_{\G}$ associated to the same face $\G$ are mapped one to the other by the flow. This proves that $\wt{\KK}$ has the property (DS) and we are done.
\end{proof}
\begin{remark}{\em 
	In the statement of Proposition~\ref{pro:extendable-after-gdsfd}, we may replace the hypothesis that $\KK$ satisfies the (DS) condition by the slightly weaker hypothesis that $\KK$ has good saturations. In fact, in the items (a)-(c) considered in the proof above we simply use that $\KK$ has the property $(DS)^{\le\nu}$ (cf. Definition~\ref{def:qds}) for $\nu$ equal to the $\tilde{\alpha}$-limit of a face of the graph. This last property can be obtained as a consequence of having good saturations.
}
\end{remark}

\strut

We end by proving Theorem~\ref{th:main}, the main result of the paper. Let us restate it here in slightly more precise terms. 

Let $\MM=(M,D,\LL)$ be a non s-resonant of More-Smale type HAFVSD. For any $p\in N\cup S_{tr}$ (i.e., any singular point which is not a tangential saddle point) fix a realization $W(p)$ of the local invariant manifold $W(p)$ which is transversal to $D$. More precisely, $W(p)$ is a neighborhood of $p$, if $p\in N$ is not a three-dimensional saddle, $W(p)=W^1_p$, if $p\in N$ is a three-dimesnional saddle and $W(p)=W^2_p$, if $p$ is a transversal saddle point.

\begin{theorem}\label{th:main2}
	 Let $V$ be a neighborhood of $D$ in $M$ and $V_p$ neighborhoods of $W(p)\cap V$ in $M$ for any $p\in N\cup S_{tr}$ such that $V_p\cap V_q=\emptyset$, if $p\ne q$. Then there exist a compact semianalytic neighborhood $U\subset V$ of $D$ in $M$ and compact semianalytic discs $T_p\subset Fr(U)\cap V_p$, for any $p\in N\cup S_{tr}$, satisfying the following:
	\begin{enumerate}[(i)]
		
		\item  The frontier $Fr(U)$ is a topological, piecewise smooth surface given by the disjoint union		
		$$
	Fr(U)= Fr(U)^\smallsmile\bigcup_{p\in N\cup S_{tr}}T_p.
		$$  
		\item Each disc $T_p$ contains $Fr(U)\cap W(p)$ and, in turn, it is contained in a smooth analytic surface in $V_p$ everywhere transversal to $\LL$. Moreover, $W(p)\cap T_p$ is equal to: $T_p$, when $\dim W(p)=3$; a singleton in $int(T_p)$, when $\dim W(p)=1$; a closed interval $I_p$ with extremities in $\partial T_p$ and such that $\dot{I_p}\subset int(T_p)$, when $\dim W(p)=2$. 
		\end{enumerate}
		Furthermore, the type of points in $T_p$ relatively to $Fr(U)$ is given by (assuming for instance that $W(p)$ is the stable manifold, otherwise change each type a-b with b-a):
		\begin{itemize}
		\item Points of $int(T_p)$ are of type e-i.
		\item If $\dim W(p)=3$, points in $\partial T_p$ are of type e-i.
		\item If $\dim W(p)=1$, points in $\partial T_p$ are of type type e-t.		
		\item  If $\dim W(p)=2$, then there are exactly four points in $\partial T_p$ of type t-t, none of them in $I_p$. Besides, among the four intervals in which these points divide $\partial T_p$, there are two of them, say $L_p^1$, $L_p^2$, intersecting $I_p$ and  formed by points of type t-i, while the other two do not intersect $I_p$ and are formed by points of type e-t.  
	\end{itemize}
	In addition, we may assume that the elements in the family 
	\begin{equation}\label{eq:last-disj-family}
	\{\Sat_U(W(p)))\}_{p\in S_{tr}}\,\cup\,\{\Sat_U(L^1_p\setminus I_p),\Sat_U(L^2_p\setminus I_p))\}_{p\in S_{tr}}
	\end{equation} 
	are mutually disjoint subsets of $U$.
\end{theorem}

\begin{proof}
Let us consider first a fattening $\KK_0$ over $\Omega$ with $|\KK_0|\subset V$ and such that, for any $p\in N\cup S_{tr}$, we have:

(a) If $\dim W(p)=3$, then $\KK_0(\nu_p)\subset V_p$.

(b) If $\dim W(p)=1$, then $L_{\KK_0(\nu_p)}\subset V_p$.

(c) If $\dim W(p)=2$, then $F_{\KK_0(\nu_p^+)}\cup F_{\KK_0(\nu_p^-)}\subset V_p$.

By virtue of Theorem~\ref{th:distinguished}, there is a distinguished refinement $\KK_1<\KK_0$. Now Theorem~\ref{th:good-sat}' along with Proposition~\ref{pro:extendable-after-gdsfd} allows us to consider an extendable distinguished refinement $\KK<\KK_1$. Put $U=\wh{\KK}$, the extended support of $\KK$. Being $\KK$ a refinement of $\KK_0$, it necessarily satisfies the same properties (a), (b), (c) above. Thus $U$ satisfies the requirements of Theorem~\ref{th:main2} by means of  Lemma~\ref{lm:extended-support}. More precisely, the disc $T_p$ in the statement is given by $T_p=L_{\KK(\nu_p)}$, if $\dim W(p)\in\{1,3\}$, and by $T_p=\DD_{\KK,\nu_p^+}\cup\DD_{\KK,\nu_p^-}$, where $\DD_{\KK,\nu_p^\epsilon}$ is the (unique) free $\KK$-door at $\nu_p^\epsilon$, for $\epsilon\in\{+,-\}$, if $\dim W(p)=2$. Finally, in light of Proposition~\ref{pro:extendable-after-gdsfd}, the sets in the family (\ref{eq:last-disj-family}) are mutually disjoint since $\KK$ has the property (DS).
\end{proof}



\begin{thebibliography}{AAA}
	%
	\bibitem{And} \obra{Andronov, A.A. et al.}
	{Qualitative theory of second-order dynamic systems}{J. Wiley, New York 1973, ISBN 978-0470031957}
	%
	%
	\bibitem{Alo-C-C1}\obra{Alonso-Gonz\'{a}lez, C; Camacho, M.I; Cano, F.}{Topological Classification of Multiple Saddle Connections}{Discrete and Continuous Dynamical Systems \textbf{15} (2006), 395-414}
	%
	\bibitem{Alo-C-C2}\obra{Alonso-Gonz\'{a}lez, C; Camacho, M.I; Cano, F.}{Topological Invariants for Singularities of Real Vector Fields in Dimension Three}{Discrete and Continuous Dynamical Systems \textbf{20} (2008), 275-291}
	%
	\bibitem{Alo-C-R}\obra{Alonso-Gonz\'{a}lez, C; Cano, F.; Rosas, R.}{Infinitesimal Poincar\'{e} Bendixson problem in dimension three}{International Mathematics Research Notices \textbf{21} (2014), 5994-6019}
	%
	\bibitem{Alo-S2}\obra{Alonso-Gonz\'{a}lez, C; Sanz S\'{a}nchez, F.}{Stratification of three-dimensional real flows II: A generalization of Poincar\'{e} theorem}{In preparation}
	%
	%
	\bibitem{Bru} \obra{Brunella, M.}
	{Instability of Equilibria in Dimension Three}{Annales Institut Fourier,
		{\bf 48} 5 (1998), 1345-1357}
	%
	\bibitem{Cam}\obra{Camacho, M. I.}{A contribution to the Topological Classification of Homogeneous Vector Fields in $\R^3$}{Journal of Differential Equations, {\bf 57} (1985), 158-171}
	%
	\bibitem{Cam2}\obra{Camacho, M. I.}{Genericity of Hyperbolic Homogeneous Vector Fields}{Journal of Differential Equations, {\bf 97} (1992), 288-309}
	%
	\bibitem{Cam-C-S}\obra{Camacho, C; Cano, F; Sad, P.}{Desingularization of absolutely isolated singularities of vector fields}{Inventiones Mathematicae \textbf{98} (1989),  351-369}
	%
	\bibitem{Car-S} \obra{Carrillo, S.; Sanz, F.}{Briot-Bouquet's theorem in high dimension}{Publications Mathématiques  \textbf{58}  (2014), 135-152}
	%
	\bibitem{Dol-S}\obra{Dolich, A.; Speissegger, P.}{An ordered structure of rank two related to Dulac's problem}{Fundamenta Matematicae. \textbf{198} no. 1 (2008), 17–60}
	%
	\bibitem{Dum}\obra{Dumortier, F.}{Singularities of vector fields on
		the plane}{Journal of Differential Equations,
		\textbf{23} (1977) 53-106}
	%
	\bibitem{Gro}\obra{Grobman, M. D.}{Homeomorphisms of systems of
		differential equations}{Doklady Akademii Nauk SSSR, \textbf{128} (1959), 880-881}
	%
	\bibitem{Har}\obra{Hartman, P.}{On the local linearization of differential equations}
	{Proceedings of the AMS, 14 (1963), 568-573}
	%
	\bibitem{Hir-P-S}\obra{Hirsch, M. W., Pugh, C., Shub, M.}{Invariant
		Manifolds}{Lecture Notes
		in Mathematics, 583. Springer-Verlag (1977)}
	%
	\bibitem{Ily-Y}\obra{Ilyashenko, Y.; Yakovenko, S.}{Lectures on Analytic Differential Equations}{Graduate Studies in Mathematics, 86. AMS, 2008}
	%
	\bibitem{Kai-R-S}\obra{Kaiser, T.; Rolin, J.-P.; Speissegger, P.}{Transition maps at non-resonant hyperbolic singularities are o-minimal}{J. Reine Angew. Math. \textbf{636} (2009), 1–45}
		%
	\bibitem{Mar-V2}\obra{Marín, D.; Villadelprat, J.}{ Asymptotic expansion of the Dulac map and time for unfoldings of hyperbolic saddles: general setting}{J. Differential Equations \textbf{275} (2021), 684–732}
	%
	\bibitem{Mar-R-S}\obra {Mart\'{\i}n,R.; Rolin, J.-P.; Sanz, F.}
	{Local Monomialization of Generalized Analytic Functions}{Revista de
		la Real Academia de Ciencias Exactas, Fisicas y Naturales, Serie A.
		Matematicas, \textbf{107} Issue 1 (2013) 189-211}
%
	\bibitem{Pal-M}\obra{Palis, J.; de Melo, W.}{Geometric Theory of
		Dynamical Systems}{Springer-Verlag New York (1982)}
	%
	\bibitem{Pan}\obra{Panazzolo, D.}{Resolution of singularities of real-analytic vector fields in dimension three}{Acta Mathematica  \textbf{197}  (2006), 167-289}
	%
	\bibitem{Per}\obra{Perko, L.}{Differential equations and dynamical
		systems}{Texts in Applied
		Mathematics, 7. Springer-Verlag, New York,(1991)}
	%
	\bibitem{Poi}\obra{Poincar\'e, H.}{M\'emoire sur les courbes d\'efinies
		par une \'equation diff\'erentielle}{Jouranl de Math\'ematiques Pures et Appliquées, \textbf{7} (1881),
		375-422}
	%
	\bibitem{Sei}\obra{Seidenberg, A.}{Reduction of the singularities of
		the differential equation
		$Ady=Bdx$}{American Journal of Mathematics, \textbf{90} (1968), 248-269}
	%
\end{thebibliography}
\end{document}